\documentclass[10pt]{article}

\usepackage[utf8]{inputenc} 
\usepackage[T1]{fontenc}
\usepackage[a4paper]{geometry}
\usepackage[french,english]{babel}

\usepackage{amsmath, amssymb,mathrsfs,empheq}
\usepackage{amsthm}
\usepackage{graphicx,enumerate}
\usepackage{dsfont}
\usepackage{fancybox}
\usepackage{morefloats}
\usepackage{color}
\usepackage{cancel}
\usepackage{url}
\usepackage{csvsimple,pgf-pie,pgfplots,tikz}
\usetikzlibrary{arrows.meta,decorations.pathmorphing,backgrounds,positioning,fit,petri}

\pgfplotsset{compat=1.18}

\usepackage{hyperref} 

\usepackage{ifthen}
\newboolean{details}
\newcommand{\details}[1]{\ifthenelse{\boolean{details}}{\textcolor{blue}{#1}}{}}

\allowdisplaybreaks[4]

\theoremstyle{plain}
\newtheorem{prop}{Proposition}[section]
\newtheorem{thm}[prop]{Theorem}
\newtheorem{lem}[prop]{Lemma}

\theoremstyle{definition}

\newtheorem{defi}[prop]{Definition}

\theoremstyle{remark}
\newtheorem{rmq}{Remark}[section]

\numberwithin{equation}{section}
\newcounter{assumption}
\newcommand{\R}{\mathbb{R}}
\newcommand{\Z}{\mathbb{Z}}
\newcommand{\N}{\mathbb{N}}
\newcommand{\dd}{\mathrm{d}}
\newcommand{\ep}{\varepsilon}
\newcommand{\alp}{\alpha}

\newcommand{\e}{{\mathrm{e}}}
\newcommand{\dt}{{\Delta t}}
\newcommand{\dx}{{\Delta x}}

\newcommand{\dz}{{\Delta z}}

\newcommand{\ds}{\displaystyle}
\newcommand{\tin}{{\mathrm{in}}}

\newcommand{\ccl}{[\![}
\newcommand{\ccr}{]\!]}
\newcommand{\pt}{\renewcommand{\labelitemi}{$\bullet$}}
\renewcommand{\phi}{\varphi}
\renewcommand{\ge}{\geqslant}
\renewcommand{\le}{\leqslant}

\renewcommand\atop[2]{\genfrac{}{}{0pt}{}{#1}{#2}}

\newcommand{\dHe}{{H_\eta}} 
\newcommand{\dRe}{{R_\eta}} 
\newcommand{\K}{{\mathcal{K}}} 
\newcommand{\dHeAP}{{H_\eta^{\mathrm{AP}}}}
\newcommand{\dReAP}{{R_\eta^{\mathrm{AP}}}}
\newcommand{\Hcont}{{\mathcal{H}}} 
\newcommand{\Rcont}{{\mathcal{R}}} 
\newcommand{\Lin}{{\mathrm{Lip}_\tin}} 
\newcommand{\Linf}{{\mathrm{L}_{\max}}} 
\newcommand{\LR}{{\mathrm{Lip}_\Rcont}} 
\newcommand{\Lx}{\mathrm{Lip}_x}
\newcommand{\Lt}{\mathrm{Lip}_t}
\newcommand{\Lxn}{{\Lx(t^n)}}
\newcommand{\LxT}{{\Lx(T)}}
\newcommand{\tn}{{t^n}}
\newcommand{\tnp}{{t^{n+1}}}
\newcommand{\LH}{{\mathcal{L}}} 
\newcommand{\CLH}{{C_\LH}} 
\newcommand{\Ch}{{C_\dHe}}
\newcommand{\CHL}{{\mathrm{Lip}_\dHe}} 
\newcommand{\CHLAP}{{\mathrm{Lip}_\dHeAP}}
\newcommand{\bM}{{b_M}} 
\newcommand{\bm}{{\bM^{-1}}}
\newcommand{\Lb}{{\bM}} 
\newcommand{\ua}{{\underline{a}}}
\newcommand{\oa}{{\overline{a}}}
\newcommand{\ub}{{\underline{b}}}
\newcommand{\ob}{{\overline{b}}}
\newcommand{\ubeta}{{\underline{\beta}}}

\newcommand{\bu}{{\boldsymbol{u}}}
\newcommand{\bv}{{\boldsymbol{v}}}
\newcommand{\bw}{{\boldsymbol{w}}}
\newcommand{\bI}{{\boldsymbol{I}}}
\newcommand{\bJ}{{\boldsymbol{J}}}
\newcommand{\bx}{{\boldsymbol{x}}}
\newcommand{\wn}{{w\left(\tn\right)}}
\newcommand{\Jdt}{{I_\dt}}
\newcommand{\Idt}{{I_\dt}}

\newcommand{\Jl}{{I_0}}
\newcommand{\Jlp}{{I_0^+}}

\newcommand{\Jdtn}{{I_{\dt^n}}}

\newcommand{\Jdtkm}{{I^{k,-}_\dt}}

\newcommand{\Jdtkp}{{I^{k,+}_\dt}}

\newcommand{\Jlkm}{{I^{k,-}_0}}

\newcommand{\Jlkp}{{I^{k,+}_0}}

\newcommand{\MsJ}{{\mathcal{M}_s^I}}
 \newcommand{\MsJnp}{{\mathcal{M}_s^{\Jdt(\tn+s)}}}
\newcommand{\MepvI}{{\mathcal{M}_\ep^J}}

\newcommand{\MepvInp}{{\mathcal{M}_\ep^{J^{n+1}}}}
\newcommand{\MzerovI}{{\mathcal{M}_0^I}}
\newcommand{\MzerovInp}{{\mathcal{M}_0^{I^{n+1}}}}
\newcommand{\udt}{{u_\dt}}
\newcommand{\ul}{{u_0}}
\newcommand{\udtn}{{u_{\dt^n}}}
\newcommand{\uin}{{u^\tin}}
\newcommand{\uinconv}{{u^\tin_{\mathrm{conv}}}}
\newcommand{\uinnotconv}{{u^\tin_\mathrm{not\;conv}}}
\newcommand{\Cc}{{\mathcal{C}^0}}
\newcommand{\vep}{{v_\ep}}
\newcommand{\Jep}{{J_\ep}}

\newcommand{\vin}{{v^\tin_\ep}}
\newcommand{\tvik}{{\tilde{v}_{i,k}}}
\newcommand{\twik}{{\tilde{w}_{i,k}}}

\newcommand{\cm}{{c_m^\tin}}
\newcommand{\cM}{{c_M^\tin}}
\newcommand{\phiAP}{{\phi^{\mathrm{AP}}}}

\newcommand{\vk}{{v^k}}
\newcommand{\vl}{{v^\infty}}
\newcommand{\wk}{{w^k}}
\newcommand{\wl}{{w^\infty}}
\newcommand{\vdtk}{{v^k_\dt}}
\newcommand{\wdtk}{{w^k_\dt}}

\newcommand{\psim}{{\Psi^-_{\alpha,\sigma}}}
\newcommand{\psip}{{\Psi^+_{\alpha,\sigma}}}
\newcommand{\phip}{{\phi^+_{\alpha,\sigma}}}
\newcommand{\owk}{{\overline\wdtk}}
\newcommand{\Mm}{{\mathcal{M}}}

\newcommand{\bp}{{\boldsymbol{p}}}
\newcommand{\bq}{{\boldsymbol{q}}}
\newcommand{\be}{{\boldsymbol{e}}}
\newcommand{\bi}{{\boldsymbol{i}}}

\renewcommand{\leq}{\le}

\newcommand{\mail}[1]{\href{mailto:#1}{\texttt{#1}}}

\title{Numerical approximation of a class of constrained Hamilton-Jacobi equations
}
\author{
Beno\^it Gaudeul\footnote{
Universit\'e Paris-Saclay, CNRS, Laboratoire de math\'ematiques d’Orsay, 91405, Orsay, France.
\mail{benoit.gaudeul@universite-paris-saclay.fr}
},
H\'el\`ene Hivert\footnote{Univ Rennes, Inria, Géosciences Rennes - UMR 6118, F-35000 Rennes, France. \mail{helene.hivert@inria.fr}} }
\date{}

\begin{document}

\setboolean{details}{false}
\selectlanguage{english}
\maketitle

\begin{abstract}
In this paper, we introduce a framework for the discretization of a class of constrained Hamilton-Jacobi equations, a system coupling a Hamilton-Jacobi equation with a Lagrange multiplier determined by the constraint. The equation is non-local, and the constraint has bounded variations. We show that, under a set of general hypothesis, the approximation obtained with a finite-differences monotonic scheme, converges towards the viscosity solution of the constrained Hamilton-Jacobi equation.

Constrained Hamilton-Jacobi equations often arise as the long time and small mutation asymptotics of population models in quantitative genetics.  As an example, we detail the construction of a scheme for the limit of an integral Lotka-Volterra equation. We also construct and analyze an Asymptotic-Preserving (AP) scheme for the model outside of the asymptotics. We prove that it is stable along the transition towards the asymptotics.

The theoretical analysis of the schemes is illustrated and discussed with numerical simulations. The AP scheme is also used to conjecture the asymptotic behavior of the integral Lotka-Volterra equation, when the environment varies in time.
\end{abstract}

\setcounter{tocdepth}{2}
\tableofcontents

\newpage

\section{Introduction}

We are interested in the design and analysis of numerical schemes for a class of constrained Hamilton-Jacobi equations,
\begin{equation}
 \label{eq:HJ}
 \tag{HJ}
 \left\{
\begin{array}{l l}
 \ds\partial_t u(t,x) + b\left(x,I(t)\right)\Hcont\left(\nabla_x u(t,x)\right) + \Rcont\left(t,x,I(t)\right)= 0, & \ds  x\in\R^d, \; t> 0 \vspace{4pt} \\
 \ds  \min\limits_{x\in\R^d} u(t,x) = 0, & \ds t\ge 0,
\end{array}
 \right.
\end{equation}
supplemented with an initial condition $u^\tin$, and where the Hamiltonian $p\mapsto \Hcont(p)$ is strictly convex.
Equations such as \eqref{eq:HJ} often arise as long time and small mutations asymptotics
of models of population structured by a phenotypic trait \cite{DiekmannJabinMischlerPerthame2005, PerthameTranspEqBio, CarrilloCuadradoPerthame2007, DesvillettesJabinMischlerRaoul2008, NordmannPerthameTaing2018, LorzLorenziClairambaultAl2015, LorenziPouchol2020}. This is a particular case of models of the theory of adaptive evolution \cite{MetzGeritzMesznaAl1996, GeritzMetzKisdiAl1997, GeritzKisdiMeszenaAl1998, MeszenaGyllenbergAl2005,Diekmann2004}. They describe the size $n_\ep(t,x)$ of a population, where $t$ denotes the time, $x$  the phenotypic trait, and $\ep\in(0,1]$  a scaling parameter.
In what follows, we will consider as an example the following Lotka-Volterra integral equation
\begin{equation}
\label{eq:LotkaVolterra_Integral}
  \ep \partial_t n_\ep(t,x) - \frac{1}{\ep^d} \int_{\R^d} \K\left(  \frac{y-x}{\ep}\right) b\left(y,J_\ep(t)\right) n_\ep(t,y)\dd y - \Rcont(t,x,J_\ep(t)) n_\ep(t,x)=0, \;\; x\in\R^d, \; t>0,
\end{equation}
where $J_\ep$ denotes a weighted size of the population. The weight $\psi$ being given, it is defined by
\[
 J_\ep(t)=\int_{\R^d}\psi(x) n_\ep(t,x) \dd x,
\]
for all $t\ge 0$.  Equation \eqref{eq:LotkaVolterra_Integral} is supplemented with an initial condition $n_\ep^\tin \in L^1(\R^d)$.
The parameters in \eqref{eq:LotkaVolterra_Integral} and \eqref{eq:HJ} are chosen according to the biological context. Indeed, with the above notations, the evolution of the population is driven by births, through the birth rate $b(x,J_\ep)$,
and deaths. The net growth rate is denoted here $\Rcont(t,x,J_\ep)$. Note that $b$ and $\Rcont$ depend on the phenotypic trait $x$, meaning that some individuals may be advantaged because they are better adapted. They also depend on the population burden on the environment, through the size of the population $J_\ep$. The phenotypic trait of the parent is transmitted to its offspring,  possibly with mutations.
In \cite{BarlesPerthame2007, BarlesMirrahimiPerthame2009} and in \eqref{eq:LotkaVolterra_Integral}, it is modelled by an integral, where $\K$ denotes a probability kernel. It could also have been represented by a Laplacian  \cite{BarlesPerthame, BarlesMirrahimiPerthame2009, LorzMirrahimiPerthame2011}.

Considering the limit $\ep\to 0$ in models such as \eqref{eq:LotkaVolterra_Integral}, stands for the study of the population, in an asymptotic regime of long time and small mutations. It is usually referred to, as the separation of ecological and evolutionary time scales. If the population does not extinct or grow uncontrolled, when $\ep\to 0$, the population  $n_\ep(t,x)$ is expected to concentrate around a set of dominant traits, see for instance \cite{BarlesPerthame, CarrilloCuadradoPerthame2007}.
The dynamics of this concentration is usually studied using the Hopf-Cole transform, a logarithmic transform of the unknown, as in \cite{BarlesPerthame2007, CarrilloCuadradoPerthame2007, DiekmannJabinMischlerPerthame2005, LorzMirrahimiPerthame2011}.
Namely, $v_\ep=-\ep\ln(n_\ep)$ is introduced. With the example \eqref{eq:LotkaVolterra_Integral}, it satisfies
\begin{equation}
 \label{eq:P_ep}
 \tag{$P_\ep$}
 \left\{
\begin{array}{l}
  \ds \partial_t \vep(t,x) + \int_{\R^d} \K(z)\; b\left(x+\ep z, \Jep\left(t\right)\right)\; \e^{-\left(\vep\left(t,x+\ep z\right) - \vep\left(t,x\right)\right)/\ep} \dd z + \Rcont\left(t,x,\Jep\left(t\right)\right) = 0 \vspace{4pt} \\
 \ds \Jep\left(t\right) = \int_{\R^d} \psi(x) \e^{-\vep\left(t,x\right)/\ep} \dd x,
\end{array}
 \right.
\end{equation}
with initial condition $v_\ep^\tin=-\ep\ln(n_\ep^\tin)$. Under suitable assumptions, see \cite{BarlesPerthame2007, BarlesMirrahimiPerthame2009}, $v_\ep$
is expected to converge, when $\ep\to 0$, to the solution $u$ of a constrained Hamilton-Jacobi equation such as \eqref{eq:HJ}, with
\begin{equation}
\label{eq:discussion_H_Sans1}
\forall p\in\R^d, \;\Hcont\left(p\right)=\int_{\R^d} \K\left(z\right) e^{-z\cdot p}dz. 
\end{equation}
The size of the population $J_\ep$, also converges, its limit is denoted $I$ in \eqref{eq:HJ}. In the asymptotic regime, it is not defined as the population size anymore, but it behaves as a Lagrange multiplier regarding the constraint $\min u= 0$.

The main difficulty in the analysis of \eqref{eq:HJ} comes from the regularity of its solution. Indeed,  $u$ is expected to enjoy no more than Lipschitz regularity, as it can be expected for viscosity solutions of Hamilton-Jacobi equation \cite{BarlesBook, CrandallIshiiLions1992, Evans}, but  $I$ may have jumps discontinuities
\cite{BarlesPerthame, PerthameTranspEqBio, BarlesMirrahimiPerthame2009}.
Because of these discontinuities, the appropriate functional space for the well-posedness of \eqref{eq:HJ} is $W^{1,\infty}_{loc}(\R^d)\times BV_{loc}$. The uniqueness of the pair $(u,I)$ solution of \eqref{eq:HJ} has been addressed in some particular cases \cite{BarlesPerthame, MirrahimiRoquejoffre2016, Kim2021}, and then in a more general setting close to the problem under study \cite{CalvezLam2020}.
Keeping in mind the biological models it comes from, the hypothesis  set on the parameters in \eqref{eq:HJ} are done following \cite{BarlesMirrahimiPerthame2009}, and assumptions are added for \eqref{eq:HJ} to be in the framework of \cite{CalvezLam2020}.
In what follows, we will then suppose that $\Hcont$, $b$, $\Rcont$, and $u^\tin$ satisfy the following assumptions:
\begin{itemize}
\item $\Hcont:\R^d\longrightarrow\R^d$, is of class $\mathcal{C}^2$, and such that
\begin{equation}
\label{hyp:H_convex_superlinear}
\refstepcounter{assumption}
\tag{A\theassumption}
\Hcont \text{\;is strictly convex, and\;}
 \lim\limits_{|p|\to+\infty} \frac{\Hcont(p)}{|p|} =+\infty,
\end{equation}
and moreover
\begin{equation}
 \label{hyp:H0}
 \refstepcounter{assumption}
\tag{A\theassumption}
 \Hcont\left(0\right)=0, \;\;\mathrm{and}\;\;\forall p\in\R^d, \;\Hcont\left(p\right)\ge 0.
\end{equation}
\begin{rmq}
\label{rmq:GeneralCase}
We assume here that  $\Hcont\left(0\right)=0$, but it is in general not true in the models, see \cite{BarlesPerthame2007, BarlesMirrahimiPerthame2009, NordmannPerthameTaing2018} for instance. However, this assumption is done for a sake of simplicity and can be relaxed. Indeed, remarking that
\[
 \partial_t u + b\Hcont\left(\nabla_x u\right) + \Rcont = \partial_t u + b \left(\Hcont\left(\nabla_x u\right) - \Hcont\left(0\right)\right) + \Rcont+ b\Hcont\left(0\right),
\]
gives the key to adapt the results to more general convex Hamiltonians which are minimal at $p=0$.
\end{rmq}
\item
$b:\R^d\times\R\to\R$, is a smooth positive and bounded function. We suppose that there exists $\bM>1$ such that
\begin{equation}
\label{hyp:b_bounds}
\refstepcounter{assumption}
\tag{A\theassumption}
 \forall x\in\R^d, \;\forall I\in\R, \; 0<\bm\le b\left(x,I\right)\le \bM,\;\;\mathrm{and}\;\;\forall (i,j)\in\ccl 1,d\ccr^2, \;\; |\partial_{x_i,x_j}^2 b\left(x,I\right)|\le \bM,
\end{equation}
and such that it is also a Lipschitz constant of $b$ in both variables
\begin{equation}
 \label{hyp:b_lipschitz}
 \refstepcounter{assumption}
\tag{A\theassumption}
 \forall x, y\in\R^d, \;\forall I, J\in\R, \; \left|b\left(x,I\right)-b\left(y,J\right)\right|\le \Lb\left( \left|x-y\right|+\left|I-J\right|\right).
\end{equation}
Moreover, we suppose that
\begin{equation}
\label{hyp:b_decreasing}
\refstepcounter{assumption}
\tag{A\theassumption}
\forall x\in\R^d, \;\;I\mapsto b\left(x,I\right) \;\text{is non-increasing}.
\end{equation}
This hypothesis is relevant when considering the biological meaning of $b$ and $I$. Indeed, it implies that the larger the population is, the smaller its birth rate is, as the environment may not offer enough resources for the population.
\item $\Rcont:\R_+\times\R^d\times\R\longrightarrow\R$,
is  smooth, and decreasing in its last variable. More precisely, there exists a constant $K_1$ such that
\begin{equation}
 \label{hyp:R_decreasing}
 \refstepcounter{assumption}
\tag{A\theassumption}
 \forall t\ge 0, \; \forall x\in\R^d,\;\forall I\in\R, \; -K_1\le \partial_I \Rcont\left(t,x,I\right) \le -K_1^{-1}<0,
\end{equation}
and $I_M\ge I_m> 0$,  such that
\begin{equation}
 \label{hyp:R_ImIM}
 \refstepcounter{assumption}
\tag{A\theassumption}
 \frac{I_M}{K_1}-\bM \ge 1, \;\; \min\limits_{t\ge 0, x\in\R^d} \Rcont\left(t,x,I_m\right)\ge 0, \;\; \max\limits_{t\ge 0, x\in\R^d} \Rcont\left(t,x,I_M\right)\le 0.
\end{equation}
We emphasize on the fact, that the first condition is purely technical, and will appear naturally in what follows. Note that since $\Rcont$ is decreasing in its last variable, $I_M$ can be chosen according to the third condition in \eqref{hyp:R_ImIM}, and then be increased enough so that the first one is satisfied. Let us also define $K_2$ such that
\begin{equation}
 \label{hyp:R_bounded}
 \refstepcounter{assumption}
\tag{A\theassumption}
 \forall t\ge 0, \;\forall I\in[\min(I_m/2,I_m-1), 2I_M], \; \|\Rcont\left(t,\cdot,I\right)\|_{W^{2,\infty}(\R^d)}\le K_2,
\end{equation}
and suppose in addition that $(t,I)\longrightarrow \|\Rcont\left(t,\cdot,I\right)\|_{L^\infty(\R^d)}$ is bounded on any bounded set of $\R_+\times\R$.
Finally, the dependence in time of $\Rcont$ is supposed to be $\LR$-Lipschitz,
\begin{equation}
 \label{hyp:R_t-Lipschitz}
 \refstepcounter{assumption}
\tag{A\theassumption}
  \forall x\in\R^d, \;\forall I\in\R, \forall t,\;s\in\R_+, \; |\Rcont\left(t,x,I\right)-\Rcont\left(s,x,I\right)|\le \LR |t-s|,
\end{equation}
and to simplify notations, we introduce
\begin{equation}
\label{hyp:R_K}
K=\max\{K_1,\;K_2,\; \LR\}.
\end{equation}
\begin{rmq}
One can notice that $I\mapsto b\left(x,I\right)\Hcont\left(p\right) + \Rcont\left(t,x,I\right)$ is decreasing for all $t\in\R_+$, $x\in\R^d$ and $p\in\R^d$. This is indeed a consequence of  \eqref{hyp:H0}-\eqref{hyp:b_decreasing}-\eqref{hyp:R_decreasing}. It is an important condition for $I$ to be uniquely determined as the Lagrange multiplier of the viscosity solution of \eqref{eq:HJ}, see \cite{CalvezLam2020}.
\end{rmq}
\item $u^\tin$ is $\Lin$-Lipschitz,
\begin{equation}
 \label{hyp:u0_Lipschitz}
 \refstepcounter{assumption}
\tag{A\theassumption}
 \forall x,y\in\R^d, \;|u^\tin(x)-u^\tin(y)|\le \Lin |x-y|,
\end{equation}
is coercive
\begin{equation}
 \label{hyp:u0_coercive}
 \refstepcounter{assumption}
\tag{A\theassumption}
 \exists\; \oa >\ua >0, \;\exists\; \ub,\;\ob\in\R, \;\forall x\in\R^d, \;
 \ua|x|+\ub \le u^\tin(x)\le \oa|x|+\ob,
\end{equation}
and  its minimum is equal to $0$
\begin{equation}
 \label{hyp:u0_minimum}
 \refstepcounter{assumption}
\tag{A\theassumption}
 u^\tin\ge 0, \;\; \mathrm{and} \;\; \exists\; x_0\in\R^d, \;u^\tin(x_0)=0,
\end{equation}
even though this minimum point is not necessarily unique.
Note that $|\cdot|$  denotes here
the Euclidean norm on $\R^d$.
\end{itemize}

Considering  \eqref{eq:HJ} with $\Rcont(t,x,I)=\Rcont(x,I)$, the following result holds
\begin{thm}[\cite{CalvezLam2020}]
 \label{thm:CalvezLam2020}
 Suppose that assumptions \eqref{hyp:H_convex_superlinear} to \eqref{hyp:u0_minimum} are satisfied, and that $\Rcont$ does not depend on $t$.
 \begin{enumerate}
  \item Let $(u_1,I_1)$ and $(u_2,I_2)$  two solutions of \eqref{eq:HJ} in $W^1_\mathrm{loc}\times BV_\mathrm{loc}$ with the same initial data $u^\tin$.
  Then $u_1=u_2$ and $I_1=I_2$ almost everywhere.
  \item Let $I\in BV(0,T)$ be given. Then the variational solution $u$ of
  \begin{equation}
   \label{eq:HJSansContrainte}
   \partial_t u(t,x) + b(x,I(t)) \Hcont\left(\nabla_x u(t,x)\right) + \Rcont(x,I(t))=0, \; t>0, x\in\R^d,
  \end{equation}
with initial data $u^\tin$, is the unique locally Lipschitz viscosity solution of \eqref{eq:HJSansContrainte} over $(0,T]\times\R^d$. Moreover, $u$ is independent of the choice of a representative of $I$ in $BV$. Namely, if \eqref{eq:HJSansContrainte} is considered with two source terms $I_1$ and $I_2$ in $BV(0,T)$ such that $I_1=I_2$ almost everywhere, then $u_1=u_2$.
 \end{enumerate}
\end{thm}
\begin{rmq}
\label{rmq:conjecture_unicite}
The uniqueness result for the solution of \eqref{eq:HJ} is \emph{a priori} only true when $\Rcont$ is independent of $t$, and when the previous assumptions are satisfied.
However, we will assume that it is still valid for more general $\Rcont$, with Lipschitz regularity with respect to $t$.
From a modelization point of view, this enables the environment to change in time.
 Proving that Theorem \ref{thm:CalvezLam2020} still holds when $\Rcont$ depends on $t$, with Lipschitz regularity would deserve a dedicated study.
As Lipschitz regularity is standard for time-dependency in Hamilton-Jacobi equations \cite{BarlesBook}, we believe 
this conjecture likely to be true.
\end{rmq}

The analysis of a finite-differences scheme for a simplified \eqref{eq:HJ}, with $b\equiv 1$, $\Hcont(p)=|p|^2$, and $\Rcont$ independent of $t$, has been carried out in \cite{CalvezHivertYoldas2022}.
The goal of this paper is to generalize it, and to propose a framework for the numerical analysis of constrained Hamilton-Jacobi equations, such as \eqref{eq:HJ}.
When the parameters are regular enough, and when dealing with bounded solutions, the approximation of viscosity solutions of  Hamilton-Jacobi equations can be handled with finite-differences monotonic schemes, see \cite{CrandallLions, Souganidis}. Here,  the two main difficulties of the problem are hence the coercivity of $u^\tin$ \eqref{hyp:u0_coercive}, and  the lack of regularity of $I$.
As in \cite{CalvezHivertYoldas2022}, we use  Theorem \ref{thm:CalvezLam2020} to overcome separately these concerns, since it enables decoupling the Hamilton-Jacobi equation from its constraint. Indeed,
at the continuous level, one can consider $u$, the viscosity solution of \eqref{eq:HJ} with $I$ given as a source term, show separately that $\min  u(t,\cdot)=0$ for all $t$, and conclude that $(u,I)$ solves \eqref{eq:HJ}.
Following these considerations, we approximate the classical Hamilton-Jacobi part of the problem with a monotonic scheme which enjoys a discrete maximum principle \cite{CrandallLions, Souganidis}. The approximation of $I$ is treated with a nonlinear problem to solve, so that the constraint is satisfied. In this step, the monotony of $I\mapsto b(x,I) \Hcont(p)+\Rcont(t,x,I)$ yields stability, and counterparts the lack of regularity of $I$.

However, the generalization of \cite{CalvezHivertYoldas2022} to more general Hamiltonians brings additional difficulties, in particular when proving that the scheme preserves the semi-concavity of the solution of \eqref{eq:HJ} (see \cite{Evans}). This property is crucial for the stability of the scheme, and for the convergence of the approximation of $I$. It was easily satisfied in \cite{CalvezHivertYoldas2022}, but it may require much more precise estimates here.
We hence distinguish two classes of schemes, depending on their behavior when investigating the discrete semi-concavity of the solutions.
The scheme proposed in \cite{CalvezHivertYoldas2022} belongs to  the class designated as the \emph{flat setting}. In this case, the discrete Hamiltonian is positive, as is $\Hcont$, but it is not necessarily convex.  As they do not represent a particular difficulty, the semi-concavity estimates for this class of schemes can be established in  a general framework, so that the convergence of the schemes belonging to this class is true in any dimension, and even when adding a smooth dependency in $x$ to $\Hcont$. On the contrary, discrete Hamiltonians are convex but not necessarily positive in the \emph{convex setting}. This framework is natural, as the convexity of $\Hcont$ is a crucial hypothesis for the well-posedness of the continuous problem \eqref{eq:HJ}, but the discrete semi-concavity estimates are more intricate. In this case, we establish the convergence of the scheme only in dimension $1$, and we postpone the investigation of the generalization to a future work.

We then illustrate the construction of schemes for \eqref{eq:HJ}, by detailing the example where the Hamiltonian $\Hcont$ is defined by \eqref{eq:discussion_H_Sans1}.
For this particular case, we propose five approximations of \eqref{eq:HJ}, and we discuss their properties. We also come back to the biological model, and the $\ep$-dependent problem \eqref{eq:P_ep} that converges to \eqref{eq:HJ} in the limit $\ep\to 0$, and we propose and analyze an asymptotic-preserving scheme adapted to this asymptotics. Such schemes enjoy stability properties when $\ep\to 0$, meaning that they are accurate in all the regimes, with no constraint on the discretization parameters. In particular, we show with numerical tests that they catch the asymptotics  of the model, even when the population goes to an extinction, so that the asymptotic regime is no more described by \eqref{eq:HJ}.

The paper is organized as follows: the scheme for \eqref{eq:HJ} is constructed and discussed in Section \ref{sec:MainResults}. In addition, examples of schemes for the integral-defined Hamiltonian \eqref{eq:discussion_H_Sans1} are presented in Section \ref{sec:Example}, where an asymptotic-preserving scheme for \eqref{eq:P_ep} is also constructed and analyzed. The  convergence of the scheme for \eqref{eq:HJ} is then proved in Section \ref{sec:Convergence}. Various properties of the schemes are illustrated and discussed via numerical tests in Section \ref{sec:numericalsims}. Finally, technical points related to the asymptotic-preserving scheme of Section \ref{sec:Example} are proved in Appendix \ref{sec:APProof}.

\section*{Acknowledgements}
\label{sec:Acknowledgements}
\addcontentsline{toc}{section}{\nameref{sec:Acknowledgements}}
The authors warmly thank Vincent Calvez for
 the many very interesting discussions they had about this problem.
 They also thank Pierre Navaro, for his expertise with Julia.
This
project has received funding from the European Research Council (ERC) under the European Union’s Horizon
2020 research and innovation program (ERC consolidator grant WACONDY no 865711)

\section{Construction of the scheme and main results}
\label{sec:MainResults}

In this section, we present a general framework for the construction of schemes for \eqref{eq:HJ} in dimension $d=1$, which converge towards the viscosity solution of \eqref{eq:HJ}. This is a generalization of the scheme proposed in \cite{CalvezHivertYoldas2022} for the special case $\Hcont\left(p\right)=p^2$, and it is designed according to the following principles:
\begin{itemize}\pt
\item The time-dependent Lagrange multiplier $I$ in \eqref{eq:HJ} is implicited in the schemes, to make sure that the constraint $\min u=0$ in \eqref{eq:HJ} is preserved at the discrete level. 
 \item Following  \cite{CrandallLions, Souganidis},  $\Hcont$ is discretized with a monotonous  explicit scheme. However, in comparison to \cite{CrandallLions, Souganidis}, we relax the consistency hypothesis of the discrete Hamiltonian. Indeed, this assumption means that $\Hcont\left(p\right)$ can be computed without approximation for any $p\in\R$, and this is for instance not relevant when $\Hcont$ is defined as an integral, as in \eqref{eq:discussion_H_Sans1}.
To this end, we introduce a discretization parameter $\eta$, accounting for the quality of the approximation of $\Hcont:\R\longrightarrow\R$, and we denote $\dHe:\R^2\longrightarrow\R$ the approximated Hamiltonian. 
\item It may happen, that the net growth rate $\Rcont$ is not computed without approximation, especially when $\Hcont$ is modified as in Remark \ref{rmq:GeneralCase}. As previously, an approximation of $\Rcont$ is then introduced. Still denoting $\eta$ the discretization parameter, we define $\dRe:[0,T]\times\R^d\times \R\longrightarrow\R$ the approximated net growth rate.
\end{itemize}

According to these considerations, the properties of $\dHe$ and $\dRe$ are stated as follows. First of all, as in \cite{CrandallLions, Souganidis, CalvezHivertYoldas2022}, an appropriate discretization of $\Hcont$ should  be increasing in its first variable, and decreasing in the second one. However, this property can be slightly relaxed. Indeed, following the ideas of the Lax-Friedrichs scheme, one can add a viscosity to $\dHe$ to force the monotony in a given bounded set of $\R^2$. To allow for these discretizations in the framework of our study, we let $\LH>0$ and assume that
\begin{equation}
 \label{hyp:Heta_monotonous}
 \refstepcounter{assumption}
\tag{A\theassumption}
 \forall (p,q)\in[-\LH,\LH], \; \forall e>0, \; \dHe\left(p,q+e\right) \le \dHe\left(p,q\right)\le \dHe\left(p+e,q\right).
\end{equation}
We also suppose that no discretization error is committed in \eqref{hyp:H0}, meaning that 
\begin{equation}
\label{hyp:Heta_zero}
\refstepcounter{assumption}
\tag{A\theassumption}
\dHe\left(0,0\right)=0.
\end{equation}
Note that this assumption is purely technical, but simplifies several expressions in what follows. However, one could also avoid this restriction  by considering $\dHe - \dHe\left(0,0\right)$. For the same reason, we also assume that the discrete Hamiltonian $\dHe$ is non-negative on the diagonal
\begin{equation}
 \label{hyp:Heta_nonnegative}
 \refstepcounter{assumption}
\tag{A\theassumption}
 \forall \left|p\right|\le\LH, \; \forall\eta\in(0,1], \; \dHe\left(p,p\right)\ge 0.
\end{equation}
The quality of the approximation of $\Hcont$ by $\dHe$ is quantified as follows
\begin{equation}
 \label{hyp:Heta_approximation}
 \refstepcounter{assumption}
\tag{A\theassumption}
 \exists \CLH\ge 0, \;\forall |p|\le\LH, \; \forall \eta\in(0,1], \; \left|\dHe\left(p,p\right)-\Hcont\left(p\right)\right|\le\CLH\;\eta,
\end{equation}
and the approximation of $\Rcont$ by $\dRe$ is also quantified with $\eta$
\begin{equation}
 \label{hyp:Reta_approximation}
 \refstepcounter{assumption}
\tag{A\theassumption}
 \forall t\ge 0, \; \forall x\in\R, \; \forall I\ge 0,\; \forall \eta\in(0,1],\; |\Rcont\left(t,x,I\right)-\dRe\left(t,x,I\right)|\le K\eta,
\end{equation}
where the constant $K$ is used again, as in \eqref{hyp:R_K}, to simplify notations. In addition, we suppose that $\dRe$ satisfies \eqref{hyp:R_decreasing}-\eqref{hyp:R_ImIM}-\eqref{hyp:R_bounded} and \eqref{hyp:R_t-Lipschitz}, as does $\Rcont$.
Finally, we assume that $\dHe$ is smooth enough for a pseudo differential to be defined, and that the following function is well-defined, for all $L>0$
\begin{equation}
\label{eq:Heta_defCH}
\CHL(L)=b_M\;\sup_{\eta\leq 1}\;\sup_{|p|,|q|\leq L}\;\sup_{(d_p,d_q)\in\partial\dHe(p,q)}\;(|d_p|+|d_q|),
\end{equation}
where $b_M$ is defined in \eqref{hyp:b_bounds}. 
This function will be used for various stability conditions and error estimates.

Suppose that $\dHe$ and $\dRe$ enjoy the above properties, and 
let $T>0$, and the number $N_T$ of time steps be given. The time step is defined as $\dt=T/N_T$, and let $t^n=n\dt$ for $n\in\ccl 0,N_T\ccr$. The trait step is denoted $\dx>0$, and the grid is defined with $x_i=x_0+i\dx$ for all $i\in\Z$. For $n\in\ccl 0,N_T-1\ccr$ and $i\in\Z$, the scheme for \eqref{eq:HJ} is given by 
\begin{equation}
 \label{eq:scheme}
 \tag{$\mathcal{S}$}
 \left\{ 
\begin{array}{l}
\ds  \frac{u^{n+1}_i-u^n_i}{\dt}+b\left(x_i,I^{n+1}\right) \dHe\left( \frac{u^n_i-u^n_{i-1}}{\dx}, \frac{u^n_{i+1}-u^n_i}{\dx} \right) + \dRe\left(t^{n+1},x_i,I^{n+1}\right)=0 \vspace{4pt}\\
\min\limits_{j\in\Z} u^{n+1}_j = 0.
\end{array}
 \right.
\end{equation}
It is initialized with
\begin{equation}
\label{eq:defbu0}
u^0_i=u^\tin(x_i),
\end{equation} 
for all $i\in\Z$, and it is such that  $\min_{j\in\Z} u^0_j=0$, thanks to \eqref{hyp:u0_minimum}.
\begin{defi}
\label{def:WellChosenScheme}
A scheme is \emph{$\LH$-well chosen} if $\dRe$ satisfies \eqref{hyp:R_decreasing}-\eqref{hyp:R_ImIM}-\eqref{hyp:R_bounded}-\eqref{hyp:R_t-Lipschitz}, and \eqref{hyp:Reta_approximation}, if $\dHe$ satisfies \eqref{hyp:Heta_monotonous},  \eqref{hyp:Heta_zero}, \eqref{hyp:Heta_nonnegative}, \eqref{hyp:Heta_approximation},
is smooth enough for \eqref{eq:Heta_defCH} to be well-defined and finite for all $0\le L\le \LH$, and if one one the following conditions holds:
\begin{itemize}\pt
 \item \emph{Convex setting}: $\dHe$ is convex on $[-\LH,\LH]^2$,
 \item \emph{Flat setting}: for all $\eta \in (0,1]$ and $p,q$ such that $-\LH\leq p\leq 0\leq q\leq \LH$, $\dHe\left(p,q\right)=0$.
\end{itemize}
\end{defi}

\begin{rmq}
 It is possible for a scheme to be $\LH$-well chosen for all $\LH$. In that case, it will simply be said that it is well-chosen.
\end{rmq}

\begin{rmq}
 When in the flat setting, \eqref{hyp:Heta_nonnegative} is automatically satisfied. Indeed, if $p\ge0$ we have $0= \dHe\left(0,p\right)\leq\dHe\left(p,p\right)$, and similarly if $p<0$.
\end{rmq}

\begin{rmq}
 One can notice that the flat and convex settings are not mutually exclusive.
 Moreover, it is always possible, when in the convex setting, to build a scheme that is also in the flat setting. Indeed,
  since  $\Hcont\left(p\right)\ge0$,  one can consider  $\max\left(\dHe\left(p,q\right),0\right)$.
\end{rmq}

Focus now on the properties of the trait-time mesh and the choice of $\eta$.
Since  an explicit discretization for the
derivative of $u$ with respect to trait $x$ in
\eqref{eq:HJ} is chosen, stability conditions (CFL) must be satisfied. They allow the monotony of the numerical scheme in most results. It assumes
\begin{equation}
\label{eq:CFL1}
\dt \CHL(\Lin+KT)\leq \dx,
\;\;\mathrm{and}\;\; \dt \CHL(\oa+KT)\le \dx,
\end{equation}
where $\Lin$ is defined in \eqref{hyp:u0_Lipschitz}, $\oa$ in \eqref{hyp:u0_coercive}, and $K$ in \eqref{hyp:R_K}. Note that $\dt\CHL(\ua+KT)\le \dx$ is then automatically satisfied since $\ua\le\oa$, where $\ua$ is also defined in \eqref{hyp:u0_coercive}.
However,  a stronger condition is needed to conclude the proof of the convergence of \eqref{eq:scheme} towards the viscosity solution of \eqref{eq:HJ},
\begin{equation}\label{eq:CFL2}
\dt \CHL(14(\Lin+KT)+1)\leq \dx,
\end{equation}
so that 
\begin{equation}
\label{eq:Linf}
\Linf=14 (\Lin+KT)+1, \;\;\text{and}\;\;\Ch=\CHL(\Linf),
\end{equation}
are defined, and that
the ratio $\dx/\dt$ is fixed, such that
\begin{equation}
\label{eq:CFLsatisfier}
\dt\max\left(2\CHL(\oa+KT),\;2\CHL(\Lin+KT),\; \Ch\right)\;\le\;\dx.
\end{equation}
Note that in most cases, the maximum of the two last items is $\Ch$, and that this condition is restrictive.
However, it can be relaxed in practice, see Section \ref{sec:cflgest}.
Any monotonous function of $\dt$ that tends to $0$ when $\dt\to 0$ can be chosen to define $\eta$. In what follows,
\begin{equation}
\label{eq:etachoice}
\eta=\min(\sqrt{\dx},1),
\end{equation}
will be considered, as this choice simplifies the estimates in the convergence proof.
The cut-off for $\dx>1$ is introduced  to use the definition of $\CHL(L)$ in \eqref{eq:Heta_defCH}.
Finally, a restriction on the discretization step is necessary, that is
\begin{equation}\label{eq:smooth_R}
\dx \le \frac{1}{Kw_T\CHL(\Lin+KT)},\;\;\mathrm{and}\;\; \dx\leq\frac{4}{(e^T+1)^2},
\end{equation}
with 
\begin{align}
 w_T&= \e^{\beta T} \left(\gamma T+ \sup\limits_{x\in\R} \uin''(x)\right)
 \label{eq:ProofExistence_wT}
 \\
 \gamma &= K + \bM  \max\left\{-\dHe\left(-\Lin-KT,\Lin+KT\right), \dHe\left(\Lin+KT,-\Lin-KT\right)\right\} \label{eq:ProofExistence_alpha}
 \\
 \beta&=4  \CHL(\Lin+KT) . \label{eq:ProofExistence_beta}
\end{align}
Note that these assumptions are mostly technical. Indeed,  removing the first bound requires slightly more regularity on $\Rcont$, as $I_m-1$ must then be lowered in \eqref{hyp:R_ImIM}. The second one gives rise to the $1$ in the definition of $\Linf$ in \eqref{eq:Linf}.

Let us now introduce the notion of an adapted discretization,  to refer to a discretization $\dt$, $\dx$, $\eta$  satisfying these constraints:
\begin{defi}\label{def:adapted}
A triplet $(\dt,\dx,\eta)$ is \emph{adapted to $\dHe$ and $\dRe$} if $(\dHe,\dRe)$ is a $\Linf$-well chosen scheme, and  \eqref{eq:CFLsatisfier}, \eqref{eq:etachoice}, \eqref{eq:smooth_R} are satisfied.
\\
A discretization $\dHe$, $\dRe$, $\dt$, $\dx$, $\eta$ is \emph{adapted} if $(\dt,\dx,\eta)$ is adapted to $(\dHe,\dRe)$.
\end{defi}

\begin{rmq}
$(\dHe,\dRe)$ is assumed to be $\Linf$-well chosen to ensure that  $\Ch$ is well-defined.
\end{rmq}

As scheme \eqref{eq:scheme} is implicit, let us start by stating an existence result, along with qualitative properties on $\bu=(u^n_i)_{n\in\ccl0,N_T\ccr, i\in\Z}$ and on $\bI=(I^{n+1})_{n\in\ccl0,N_T-1\ccr}$.

\begin{prop}[Scheme \eqref{eq:scheme}: existence of solutions and qualitative properties]
\label{thm:existence}
Let $\dHe$, $\dRe$, $\dt$, $\dx$, $\eta$ be an adapted discretization (Definition \ref{def:adapted}).
Then, scheme \eqref{eq:scheme} is well-posed: there exists $\bu\in\R^{\Z\times\ccl1,N_T \ccr},$ and $\bI\in\R^{\ccl1,N_T \ccr},$ satisfying \eqref{eq:scheme}, where $\bu^0=(u^0_i)_{i\in\Z}$. Moreover $\bu$ and $\bI$ satisfy the following properties:

\begin{enumerate}
\item\label{it:Lx}  \textbf{Lipschitz property:} $\bu^n$ is $\Lxn$-Lipschitz, 
\begin{equation}\label{eq:disc_lipschitz}
\forall n\in\ccl0,N_T\ccr,\;\forall (i,j)\in\Z^2, \;|u_i^n-u_j^n|\le \Lxn \dx|i-j|,
\end{equation}
where $\Lx(t)=\Lin+tK$, and with $\Lin$ defined in \eqref{hyp:u0_Lipschitz}, and $K$ in \eqref{hyp:R_K}.
\item\label{it:Lt} \textbf{Bounds for $\bu^n$:} let $\ub_t=\ub-tb_M\dHe\left(\ua,-\ua\right)-tK$, $\ob_t=\ob-tb_M\dHe\left(-\oa,\oa\right)+tK$, where $\ua$, $\oa$, $\ub$, $\ob$ are defined in \eqref{hyp:u0_coercive} and $b_M$ in \eqref{hyp:b_bounds}. Then
\begin{equation}\label{eq:discbounds}
\forall n\in\ccl0,N_T\ccr,\;\forall i\in\Z,\; \ua|x_i|+\ub_\tn\leq u_i^n\leq \oa|x_i|+\ob_\tn.
\end{equation}
\item\label{it:Rpos} Let $n\in\ccl1,N_T\ccr$ and $j\in\Z$ such that $u_j^n=0$, then
\begin{equation}\label{eq:Rpositif}
\dRe\left(t^n,x_j,I^n\right)\ge 0.
\end{equation}
\item \label{it:I_bound} \textbf{Bounds for $\bI$:} for all $n\in\ccl 1,N_T\ccr$, $I_m-1\le I^n\le I_M$, where $I_m$ and $I_M$ are defined in \eqref{hyp:R_ImIM}.
\item \label{it:I_BVbound} \textbf{Bound from below for the jumps of $\bI$:} there exists a constant $\kappa>0$ such that for all $n\in\ccl1,N_t-1\ccr$, 
\begin{equation}
\label{eq:I_BVbound}
I^{n+1}-I^n\ge -\kappa\dt.
\end{equation}
Moreover, in the flat setting, \details{$\kappa=K\LR$} $\kappa=K^2$ , with $K$ defined in \eqref{hyp:R_K}.
In the convex setting, \details{$\kappa=K\LR + K\dx\;w\left(T\right)\; \CHL\left(\Lx\left(T\right)\right)$} $\kappa=K^2 + K\dx\;w\left(T\right)\; \CHL\left(\Lx\left(T\right)\right)$ suits, with $\CHL$ defined in \eqref{eq:CFLsatisfier},  $\Lx$ in item \ref{it:Lx},  and $w$ in item \ref{it:sc}.
\item\label{it:sc}\textbf{Semi-concavity:} in the convex setting,
there exists a function $w:[0,T]\to \R$ such that for all $t\in[0,T]$,
\begin{equation}
\label{eq:w_bound}
 w\left(t\right)\le \e^{\beta t} \left(\gamma t+ \sup\limits_{x\in\R} \uin''(x)\right),
\end{equation}
with $\beta$, $\gamma$ defined in \eqref{eq:ProofExistence_beta}-\eqref{eq:ProofExistence_alpha},
and such that
for all $n\in\ccl0,N_T\ccr$, and for all $i\in\Z$
\begin{equation}\label{eq:semiconcavity}
\frac{u_{i-1}^n-2u_{i}^n+u_{i+1}^n}{\dx^2}\leq  \wn.
\end{equation}
\item\label{it:pseudomon} \textbf{Bound from below for $\dHe$:} for all $n\in\ccl0,N_T\ccr$, and for all $i\in\Z$,
\begin{equation}\label{eq:Hbounded}
\dHe\left(\frac{u_i^n-u_{i-1}^n}{\dx},\frac{u_{i+1}^n-u_{i}^n}{\dx}\right)\ge-\dx \CHL(\Lx(\tn))\wn/\bM,
\end{equation}
with $\CHL$ defined in \eqref{eq:Heta_defCH}, $\Lx$ in item \ref{it:Lx}, $w$ in \ref{it:sc} and $\bM$ in \eqref{hyp:b_bounds}.
\end{enumerate}
Moreover, $\bI$ is uniquely determined if 
\begin{equation}
 \label{eq:ConditionForUniqueness}
 \Lb \dHe\left(-\Lx(T),\Lx(T)\right)+K^{-1}\ge 0,
\end{equation}
where $\Lb$ is defined in \eqref{hyp:b_lipschitz}, $\Lx$ in item \ref{it:Lx} and $K$ in \eqref{hyp:R_K}.
\end{prop}

\begin{rmq}\label{rem:worstcaseisok}
In the flat setting,  \ref{it:sc} is \emph{a priori} not true. However \ref{it:pseudomon} holds, since $ \dHe\left(p,q\right)\geq0$ for all $p$ and $q$. This property is a consequence of the flat setting definition (Definition \ref{def:WellChosenScheme}),  and of \eqref{hyp:Heta_monotonous}.
\end{rmq}

\begin{rmq}
As a consequence of Remark \ref{rem:worstcaseisok}, it is worth noticing that \eqref{eq:ConditionForUniqueness} is automatically satisfied in the flat setting. In the convex setting, $\dHe\left(-\Lx\left(T\right),\Lx\left(T\right)\right)$ may be non-positive, so that condition \eqref{eq:ConditionForUniqueness} means that $\dRe$ must be sufficiently decreasing w.r.t. $I$ in comparison to $b$, to have $I^{n+1}$ uniquely determined by \eqref{eq:scheme}. We emphasize on the fact that \eqref{eq:ConditionForUniqueness} is a sufficient condition, and that it seems that in practice \eqref{hyp:b_decreasing}-\eqref{hyp:R_decreasing} are enough, see Section \ref{sec:ProofExistence_conclusion}.
\end{rmq}

\begin{rmq}
 Property \ref{it:sc} is the discrete equivalent of a classical result for Hamilton-Jacobi equations: they preserve an upper bound on the second derivative, see \cite{Evans}. As in the continuous case, the convexity of $\dHe$ is a crucial assumption for this result.
\end{rmq}

\begin{rmq}
 When $\Rcont$ does not depend on $t$, the time-dependent Lagrange multiplier $I$ in \eqref{eq:HJ} is expected to be a non-decreasing function in $BV(0,T)$, see \cite{CalvezLam2020}. In the simpler case studied in \cite{CalvezHivertYoldas2022}, this property was conserved at the discrete level. Here, it is in fact still true in the flat setting, but the convex setting is not enough to have the monotony of $\bI$, see Section \ref{sec:TestsNums_MonotonyI}. However the bound \eqref{eq:semiconcavity} gives an estimate on the decreasing part.
\end{rmq}

\begin{rmq}
Even when it is not expected to be non-decreasing, $I$ defined in \eqref{eq:HJ} is in $BV(0,T)$. Moreover, no more regularity is to be expected, since it can have jumps, see Section \eqref{sec:TestsNums_MonotonyI}. However, the bound from below for the jumps of $\bI$ in \ref{it:I_BVbound} yields that the decreasing jumps cannot be of order larger than $\dt$. It suggests that, in the continuous case \eqref{eq:HJ}, $I$ can only have increasing jumps, or, equivalently, that it can only decrease continuously.
\end{rmq}

Prop. \ref{thm:existence} strongly relies on monotony properties of scheme \eqref{eq:scheme}, it
is proved in Section \ref{sec:ProofExistence}.
 Let us now focus on the convergence of these approximate solutions $(\bu,\bI)$ towards the unique viscosity solution of \eqref{eq:HJ}.
 As in \cite{CalvezHivertYoldas2022}, the proof of this convergence uses compactness arguments, so that we need to pass from the discrete solutions $(\bu,\bI)$ of \eqref{eq:scheme}, to the solutions $(u,I)$ of \eqref{eq:HJ}. To that extend, let us introduce $\Jdt$ the constant by part reconstruction of $\bI$. Formally:
 \begin{equation}\label{eq:defJdt}
 \forall n\in\ccl 0, N_T-1\ccr,\; \forall s\in (0,\dt], \;\Jdt(\tn+s):=I^{n+1}, \;\;\text{and}\;\; \Jdt(0)=I^1.
 \end{equation}
 Note that the choice of $\Jdt(0)$ is made only for technical purposes and does not bear any biological signification.
For the extension of $\bu$, let us introduce for all $s\in(0,\dt]$, and  for all $f:\R\longrightarrow\R$, the operator $\MsJ$
 \begin{equation}
 \label{eq:MsJ}
\forall x\in\R,\; \forall I\in\R,\;  \MsJ(f) (x)= f(x)-s\;b\left(x,I\right)\dHe\left(\frac{f(x)-f(x-\dx)}{\dx},\frac{f(x+\dx)-f(x)}{\dx}\right),
 \end{equation}
and let us define $(t,x)\mapsto\udt (t,x)$ on $[0,T]\times\R$, such that for all $n\in\ccl 0,N_T-1\ccr$, $s\in(0,\dt]$ and $x\in\R$,
\begin{equation}
\label{eq:defudt}
\udt(\tn+s,x)=\MsJnp(\udt(\tn,\cdot))(x)-s\; \dRe\left(t^n+s,x,\Jdt(\tn+s)\right),
\end{equation}
supplemented by the initial condition
$\udt(0,\cdot)=\uin$.
One can notice that for all $n\in\ccl 0, N_T\ccr,i\in\Z$,
\begin{equation}
\label{eq:udt_init}
\udt(\tn,x_i)=u_i^n,
\end{equation}
that is, that the approximate solution corresponds with the scheme on the grid. If $\dRe$ does not depend on $t$, one may also notice that, for all $x\in\R$, $\udt(\cdot,x)$ is piecewise linear.
The following result states the convergence of $(\udt,\Jdt)$  defined in \eqref{eq:defudt}-\eqref{eq:defJdt} towards the viscosity solution of \eqref{eq:HJ}. 

\begin{prop}[Scheme \eqref{eq:scheme} : convergence]
\label{thm:limitconv}
Let $(\dHe,\dRe)$ be a well chosen scheme (Def. \ref{def:WellChosenScheme}), and $(\dt^n,\dx^n,\eta^n)$ be a sequence of discretizations adapted to $(\dHe,\dRe)$ (Def. \ref{def:adapted}) with $\dt^n\to0$ when $n\to\infty$. Let $(\udtn,\Jdtn)_n$ defined as in \eqref{eq:defudt}-\eqref{eq:defJdt}-\eqref{eq:udt_init}. Then, the following results hold:
\begin{itemize}\pt
\item $\udtn$ converges locally uniformly towards $u\in\Cc([0,T],\R)$, i.e. for all compact $\Omega$ of $\R$
\[
\|\udtn-u\|_{L^\infty([0,T]\times \Omega)}\underset{n\to+\infty}\longrightarrow 0,
\]
where $u$ is defined as the viscosity solution of \eqref{eq:HJ}.
\item $\Jdtn$ converges pointwise on $(0,T)$ towards $\Jl\in BV(0,T)$, which satisfies
\[
I_m\leq \Jl\leq I_M,
\]
and $\Jl=I$ almost everywhere on $(0,T)$, where $I$ is defined as the Lagrange-multiplier in \eqref{eq:HJ}.
\end{itemize}
\end{prop}

\begin{rmq}
\label{rmq:truncation_scheme}
 Scheme \eqref{eq:scheme} is introduced on an infinite trait grid $(x_i)_{i\in\Z}$, and Prop. \ref{thm:limitconv} is stated in the same framework. In practice, Scheme \eqref{eq:scheme} can be implemented on a grid that is reduced at each time step, so that no approximation has to be introduced on the boundaries, and that Prop. \ref{thm:limitconv} holds.
 This represents however an increase in terms of computational cost,
 as grids larger than necessary are to be handled with. To work with a fixed grid, one can  approximate the values that are outside of the grid. This has been proposed and numerically tested in \cite{CalvezHivertYoldas2022}, but the convergence cannot be properly established with these assumptions.
 Another issue, when dealing with grids larger than necessary, is that
 stability condition \eqref{eq:CFLsatisfier} may become much more restrictive on a larger grid, especially when trying to deal with data that are only Locally Lipschitz.
 This latter can be dealt with by linearizing $\dHe$ where slopes are stronger than necessary, see Section \ref{sec:cflgest}.
 \end{rmq}

  Prop. \ref{thm:limitconv} states the convergence of scheme \eqref{eq:scheme}, but does not give any clue on the convergence rate. This is due to the fact that the proof relies on compactness arguments for the sequences $(u_\dt)_{\dt\ge 0}$, and  $(I_\dt)_{\dt\ge 0}$, thanks to the stability estimates of Prop. \ref{thm:existence}.  Their limits are then identified with the viscosity solution of \eqref{eq:HJ}, using viscosity procedure. This goes through an appropriate regularization of $I_\dt$, as Lipschitz regularity is needed for time-dependency of source terms in the standard Hamilton-Jacobi framework \cite{BarlesBook, CrandallIshiiLions1992, Evans}. We refer to Section \ref{sec:ProofConvergence} for details.

\section{An example: Lotka-Volterra integral equations}
\label{sec:Example}

In previous section the presentation of the hypothesis on $\dHe$ was deliberately broad so that scheme \eqref{eq:scheme} could be used in various applications, as soon as the equation satisfies the hypothesis of Theorem \ref{thm:CalvezLam2020},
and the discretization is adapted. In this section, we focus on a model of population dynamics described by \eqref{eq:P_ep}, and propose some schemes for \eqref{eq:HJ}-\eqref{eq:discussion_H_Sans1}.
However, in order for \eqref{hyp:H0} to be satisfied, we will rather denote
 \begin{equation}
 \label{eq:discussion_H}
 \Hcont\left(p\right)=\int_\R \K\left(z\right) e^{-zp}dz-1,
 \end{equation}
 and modify $\Rcont$ accordingly.
 Here $\K(z)$
is even, non negative and of integral $1$.
To make sure that $\Hcont$ is properly defined, it is also supposed that $z\mapsto \K(z)\e^{-zp}$, and $z\mapsto |z|\K(z)\e^{-zp}$, are integrable for all $p\in\R$. This is for instance the case when $\K$ is a Gaussian or a compactly supported kernel.

\subsection{Examples of schemes}\label{sec:schemes}

In this part, we assume that $\Rcont$ can be analytically computed, so that no approximation is needed. The choice $\dRe=\Rcont$ is natural in this context, and we discuss here the choice of the approximation $\dHe$ of $\Hcont$.
Several classical schemes are covered by our analysis. Assume that we dispose of an approximation of the integral \eqref{eq:discussion_H}, and thus of $\Hcont_\eta$  such that for all $|p|\leq \Linf$:
\[
|\Hcont_\eta\left(p\right)-\Hcont\left(p\right)|\le \eta,\;\;\mathrm{and}\;\; \Hcont_\eta\left(0\right)=\Hcont\left(0\right)=0.
\]
Suppose also that $\Hcont_\eta$ is convex, even, and that it is increasing for $p>0$, as is $\Hcont$. For instance, considering a quadrature of $\Hcont$ with a symmetric grid, and renormalized for $\Hcont_\eta$ to be $0$ at $p=0$ suits, although other choices may also be adapted.
The so-called flat setting (see Def \ref{def:WellChosenScheme}) includes the scheme (2.5) of \cite{CrandallLions}
\begin{equation}
\label{eq:HCrandLions}
\tag{$\dHe^\text{CL}$}
\forall(p,q)\in\R,\;\;\dHe^\text{CL}(p,q)=\Hcont_\eta(p^+)+\Hcont_\eta(-q^-),
\end{equation}
where $p^+=\max(0,p)$ is the positive part of $p$ and $q^-=\max(0,-q)$ the negative part of $q$. 
One issue with this choice of discretization is that for a local maximum (i.e $p>0$, $q<0$), one can have $\dHe^{\text{CL}}(p,q)>\eta+\sup_{d\in[q,p]}\Hcont_\eta(d)$. 
This means that  the value of the Hamiltonian at local maxima is overestimated by up to a factor two, when the local maximum enjoys symmetry. To avoid this issue, another scheme is for instance used  in \cite{GuerandKoumaiha2019}
\begin{equation}
\label{eq:Hupw}
\tag{$\dHe^\text{upw}$}
\forall (p,q)\in\R^2,\;\;\dHe^\text{upw}(p,q)=\max\left(\Hcont_\eta(p^+),\Hcont_\eta(-q^-)\right).
\end{equation}
In this case, the issue is the lack of regularity when $\Hcont_\eta(p^+)$ and $\Hcont_\eta(-q^-)$ are equal.
Both schemes \eqref{eq:HCrandLions}-\eqref{eq:Hupw} are in the flat and convex settings, as defined in Definition \ref{def:WellChosenScheme}. 
In \cite{CrandallLions}, a Lax-Friedrichs scheme is also introduced
\begin{equation}
\label{eq:HLF}
\tag{$\dHe^\text{LF}$}
\forall (p,q)\in\R^2, \;\;\dHe^\text{LF}(p,q)=\Hcont_\eta\left(\frac{p+q}{2}\right)-\theta(q-p),
\end{equation}
with $\theta$ large enough, so that $\dHe$ increasing in its first variable, and decreasing in the second one. 
This choice no longer satisfies the flat hypothesis, but is still in the convex setting. 
Another choice stems from the limit of a natural asymptotic-preserving scheme introduced in \cite{CEMRACS2018} 
and analyzed in Section \ref{sec:APscheme}.
Formally, introduce $h^+_\eta(p)$, $h^-_\eta(p)$, some monotonous and convex functions such that
\[
\left|h^-_\eta(p)-\int_{\R^-} \K(z) e^{-zp}dz+\frac{1}{2}\right|
\le\eta/2,
\;\;\;\;
\left|h^+_\eta(q)-\int_{\R^+}\K(z) e^{-zq}dz+\frac{1}{2}\right|\le\eta/2,\]
and
\[
h^+_\eta(0)=h^-_\eta(0)=0.
\]
Roughly speaking $h_\eta^+$ and $h_\eta^-$ respectively approximate the positive and negative part of the integral. This choice is intended to provide monotony properties \eqref{hyp:Heta_monotonous} to the approximation of $\Hcont$.
Using these notations,  let
\begin{equation}
\label{eq:HP1}
\tag{$\dHe^\text{P1}$}
\forall(p,q)\in\R^2,\;\;\dHe^\text{P1}(p,q)=h^-_\eta(p)+h^+_\eta(q).
\end{equation}
As intended,  $\dHe$ is convex and enjoys the wished monotony properties \eqref{hyp:Heta_monotonous}. However, when it is true for the solution of the continuous model \eqref{eq:HJ}, the monotony of $\bI$  is not true for the solution of scheme \eqref{eq:scheme},  see section \ref{sec:TestsNums_MonotonyI}.
This a drawback of \eqref{eq:HP1}. Indeed, when the net growth rate $R$ does not depend on $t$,  the monotony of $I$ is a property that should be preserved.

We propose in what follows, a concave-convex-split scheme. It is intended to enforce the flat setting, and preserve the monotony of $I$ when it is relevant.  
Assume that the choice of $\Hcont_\eta$ discretizing the whole integral is consistent with the choices of $h_\eta^+$ and $h_\eta^-$ made above, that is: $\Hcont_\eta(p)=h^+_\eta(p)+h^-_\eta(p)$.
We make use of the smoothness of the viscosity solutions in the convex area using \eqref{eq:Hupw}, and the lack of regularity in the concave area using  \eqref{eq:HP1}. Namely, we propose
\begin{equation}
\label{eq:HCSS}
\tag{$\dHe^\text{CCS}$}
\forall (p,q)\in\R^2, \;\;
 \dHe^\text{CCS}(p,q)=\begin{cases}
\dHe^\text{upw}(p,q)\qquad&\text{if } p<q\\
\dHe^\text{P1}(p,q)&\text{else.}
\end{cases}
\end{equation}
It is worth noticing that \eqref{eq:HCSS} is the maximum of \eqref{eq:Hupw} and \eqref{eq:HP1}. It is then a convex function of $(p,q)$.

Let us close this section with a remark about the discretization of the integral defining $\Hcont$ in \eqref{eq:discussion_H}.
In several cases, for instance the Gaussian kernel, the bilateral Laplace transform of $\K$ is known analytically, thus the addition of the parameter $\eta$ may seem artificial. 
Obviously, the construction and analysis of the schemes covered by this article is possible when the constant $\CLH$ in  \eqref{hyp:Heta_approximation} is zero.

\subsection{Asymptotic-preserving scheme}
\label{sec:APscheme}

The main goal of this paper is to provide a general framework for the numerical analysis of constrained Hamilton-Jacobi equations such as \eqref{eq:HJ}. The justification we have in mind for these models, is that they often appear as long-time and small mutations limits of population dynamics models.
In this section, we focus more precisely on Lotka-Volterra integral equations \eqref{eq:P_ep},
supplemented with an initial condition $\vin$ which satisfies \eqref{hyp:u0_Lipschitz}-\eqref{hyp:u0_coercive}.
We also suppose that there exists two constants $\cm$, $\cM$ such that
\begin{equation}
\label{hyp:vin_min}
\refstepcounter{assumption}
\tag{A\theassumption}
 \cm \ep\le\min\vin \le \cM\ep,
\end{equation}
and we assume that, $b$, $\Rcont$, and $R=\Rcont+b$ in \eqref{eq:P_ep} satisfy the same assumptions as the ones in \eqref{eq:HJ}, and that $\K$ is defined as in \eqref{eq:discussion_H}. It is worth noticing that, in comparison to \eqref{eq:HJ}, problem \eqref{eq:P_ep} is no longer a coupled problem where the unknowns are
a solution of a PDE, and a Lagrange multiplier associated to a constraint.
Instead, $\Jep$ is determined here with $\vep$ with the second line of \eqref{eq:P_ep}. It accounts for a weighted measure of the size of the population, where the weight $\psi$ is such that
\begin{equation}
 \label{hyp:psi}
 \refstepcounter{assumption}
\tag{A\theassumption}
 \exists\;(\psi_m,\psi_M)\in\R^2, \; \forall x\in\R^d, \;\;0<\psi_m\le \psi(x)\le \psi_M<+\infty, \;\;\text{and}\;\; \psi\in W^{2,\infty}(\R^d).
\end{equation}
The following result holds

\begin{thm}(\cite{BarlesMirrahimiPerthame2009, CalvezLam2020})
\label{thm:P_ep}
 Assume that the assumptions of Theorem \ref{thm:CalvezLam2020} are satisfied, as well as \eqref{hyp:vin_min}-\eqref{hyp:psi}, and that $\Rcont$ does not depend on $t$. Let $\vep$ be the solution of \eqref{eq:P_ep}, and $\Jep$ be defined in \eqref{eq:P_ep}. Suppose also that $(\vin)_\ep$ is a sequence of uniformly continuous functions which converges uniformly to $\uin$. Then, $(\vep)_\ep$ converges locally uniformly to a function $u\in\Cc([0,+\infty[\times\R^d)$, and $(\Jep)_\ep$ converges almost everywhere to a function $I$, such that $I\in BV(0,T)$ for all $T>0$, and that $(u,I)$ is the unique viscosity solution of \eqref{eq:HJ}, with $\Hcont$ defined in \eqref{eq:discussion_H}, and $\Rcont$ replaced by $b+\Rcont$.
\end{thm}

\begin{rmq}
 As for Theorem \ref{thm:CalvezLam2020},
 the
  above convergence result is \emph{a priori} only true when $\Rcont$ does not depend on $t$.
\end{rmq}

However, focusing on the asymptotic regime may be not completely relevant in terms of modeling.
Considering the problem at some distance of the asymptotic regime may then give more information.
Another issue when dealing only with asymptotic regimes such as \eqref{eq:HJ} is, that it restricts the possibilities of biological situations that can be represented. Indeed, such asymptotic regimes are obtained under strong hypothesis on the birth rate $b$ and the net growth rate $\Rcont$, that make sure that the population never extincts, nor grows uncontrolled.
The question of the asymptotic behavior of the population in a regime of long time and small mutations, under more general hypothesis on the parameters of the model, has been addressed in the parabolic case for some changing environments in recent works \cite{FigueroaIglesiasMirrahimi2018, FigueroaIglesiasMirrahimi2021, CostaEtchegarayMirrahimi2021, VanRensburgSpillDuong2023}. Up to our knowledge, it is an open problem for Lotka-Volterra integral equation \eqref{eq:P_ep}.
If some of the assumptions on $b$ and $\Rcont$ were not satisfied, Theorem \ref{thm:P_ep} might not hold.
Indeed, the population may vanish or explode, depending on the fact that the environment is poorly adapted or too advantageous.
Very formally, suppose that it is possible to show that $(\vep)_\ep$ has a limit when $\ep\to 0$. Then, these behaviors can be understood with the asymptotic behaviors of $(\min \vep)_\ep$, and consequently of $(\Jep)_\ep$:
\begin{itemize}\pt
\item if $\min \vep$ converges to a positive limit, then $\Jep\to_{\ep\to 0}0$, meaning that the population vanishes, 
\item conversely, the population explodes if the limit of $\min\vep$ is negative, meaning that $\Jep\to_{\ep\to 0} +\infty$.
\end{itemize}
The numerical resolution of \eqref{eq:P_ep} requires specific attention, as \eqref{eq:P_ep} becomes stiff when $\ep$ is small. If no specific strategy was employed, the convergence of the numerical scheme for \eqref{eq:P_ep} would fail in the asymptotic regime. Schemes specifically designed for such singular problems are said to be \emph{Asymptotic Preserving} (AP).
Their properties are often summarized in the following diagram
\begin{equation}
\label{APDiagram}
 \begin{array}{r c l}
  \eqref{eq:P_ep} &\xrightarrow{\;\;\;\ds\ep\to 0\;\;\;} & \eqref{eq:HJ}  \vspace{4pt} \\
  \left. \begin{array}{c} $ $ \\ h\to 0 \\ $ $ \end{array} \right\uparrow \hspace{7pt} &  & \hspace{7pt}
  \left\uparrow  \begin{array}{c} $ $ \\ h\to 0 \\ $ $ \end{array} \right. \vspace{4pt}
  \\
  \left(S_\ep^h\right) & \xrightarrow[\;\;\;\ds\ep\to 0\;\;\;]{} &  \left(S_0^h\right)
 \end{array}.
 \end{equation}
The first line represents the continuous level: the solution of \eqref{eq:P_ep} converges when $\ep\to 0$ to the viscosity solution of \eqref{eq:HJ}. The second line refers to the numerical schemes. The scheme $(S_\ep^h)$, where $h$ summarizes all the discretization parameters, is required to converge to the solution of \eqref{eq:P_ep} when $\ep>0$ is fixed and $h\to 0$. Moreover, it must degenerate when $\ep\to 0$ with fixed $h$ to another scheme $(S_0^h)$, which approximates properly the viscosity \eqref{eq:HJ} when $h\to 0$. An AP scheme can also enjoy the stronger property of being \emph{Uniformly Accurate} (UA), meaning that its precision is independent on $\ep$. AP schemes were introduced for kinetic equations \cite{Klar1, Klar2, Jin}, and various asymptotics have been considered \cite{JinReview2012, DimarcoPareschiReview2014}.

The design of AP schemes for Hamilton-Jacobi limits of biological models has been investigated more recently.
An AP scheme for a Hamilton-Jacobi limit of a kinetic equation has been proposed in \cite{Hivert_2}, and a model structured in age and phenotypic trait but in which no mutations are considered is treated in \cite{AlmeidaPerthameRuan2020}. Regarding Lotka-Volterra models, an AP scheme for the parabolic case has been proposed and analyzed in \cite{CalvezHivertYoldas2022}.
A finite-differences based scheme which enjoys stability properties in the limit $\ep\to 0$ has been proposed in \cite{CEMRACS2018}, for a problem close to \eqref{eq:P_ep} and in dimension $d=1$. 
It can be easily adapted in a scheme for \eqref{eq:P_ep}. Namely, 
let $\dz>0$, and define for all $k\in\Z$, $z_k=k\dz$. For the other variables, we use the grids of scheme \eqref{eq:scheme}. 
To write the scheme in the spirit of the notation \eqref{eq:MsJ}, we define, 
for all $J\in\R$, a function $\MepvI:\R^\Z\to\R^Z$, such that, for all $\bv=(v_i)_{i\in\Z}$, and for all $i\in\Z$, 
\begin{align}
 \label{eq:MepvI} 
 \MepvI(\bv)_i=v_i &- \dt \dz  \sum\limits_{k\in\Z,\;\ep z_k\in[0,\dx]}\ds \K(z_k) b\left(x_i+\ep z_k,J\right)\; \e^{-z_k(v_{i+1}-v_i)/\dx} \\
 &- \dt  \dz \sum\limits_{k\in\Z,\;\ep z_k\in[-\dx,0)} \K(z_k) b\left(x_i+\ep z_k, J\right)\; \e^{-z_k(v_i-v_{i-1})/\dx} \nonumber
\\
&-\dt\dz \sum\limits_{k\in\Z,\;|\ep z_k|>\dx} \K(z_k) b\left(x_i+\ep z_k, J\right) \e^{-(\tvik-v_i)/\ep}, \nonumber
\end{align}
where $\tvik$ denotes the linear interpolation of $(\bx,\bv)$ at abscissa $x_i+\ep z_k$.

\begin{rmq}
 As $\tvik$ is a linear interpolation of $\bv$, there is formally no difference between the difference quotients  $(\tvik-v_i)/\ep$ and $z_k (v_{i+1}-v_i)/\dx$, if $0<\ep z_k\le \dx$ (or $z_k(v_i-v_{i-1})/\dx$, if $-\dx\le \ep z_k<0$). In fact, they both are the slope of the linear interpolation of $(\bx,\bv)$.
 The three lines of \eqref{eq:MepvI} may then seem artificial at first sight. However, when implemented, this reformulation ensures stability in the limit $\ep\to 0$. Indeed, the expression $(\tvik - v_i)/\ep$ does not behave well when $\ep\to 0$, as approximation errors in the linear interpolation are excessively increased, when divided by $\ep\ll1$, and injected in the exponential.  
\end{rmq}
In order to state monotony properties for $\MepvI$, let us define for all $L\ge 0$,
\begin{equation}
 \label{eq:CHLAP}
 \CHLAP(L)=\dz\bM\sum\limits_{k\in\Z} |z_k|\K(z_k)\e^{L|z_k|},
\end{equation}
where $\bM$ is defined in \eqref{hyp:b_bounds}.
Note that $\CHLAP$ is well-defined since it was supposed that $z\mapsto |z|\K(z)\e^{-zp}$, is integrable for all $p\in\R$. The monotony of $\MepvI$ is stated in the following lemma
\begin{lem}
\label{lem:AP_monotonuous}
 Let $L\ge 0$,  $J\in\R$, and let $\bv$ and $\bw\in\R^\Z$ such that for all $i\in\Z$, $|v_{i+1}-v_i|\le L\dx$, and $|w_{i+1}-w_i|\le L\dx$. Moreover, assume that $\dt \CHLAP(L)\le \dx$, with $\CHLAP$ defined in \eqref{eq:CHLAP}. We have, 
 \begin{enumerate}
  \item if $v_i\le w_i$ for all $i\in\Z$, then for all $i\in\Z$, $\MepvI(\bv)_i\le\MepvI(\bw)_i$,
  \item If $\bv-\bw\in\ell^\infty(\Z)$, then, $\MepvI(\bv)-\MepvI(\bw)\in\ell^\infty(\Z)$, and $\|\MepvI(\bv)-\MepvI(\bw)\|_\infty\le\|\bv-\bw\|_\infty$.
  \item for all $i\in\Z$, $\left|\MepvI(\bv)_{i+1}-\MepvI(\bv)_i\right|\le L\dx$.
  \item
  For all $i\in\Z$, $J\mapsto \MepvI(\bv)_i$ is non-decreasing.
 \end{enumerate}
\end{lem}

\begin{proof}
 The proof of this proposition is straightforward once one notices that the linear interpolation preserves inequalities. Namely, it means that for all $(i,k)\in\Z^2$, $\tvik\le\twik$. Then, condition $\dt\CHLAP(L)\le\dx$ makes $\MepvI$ be a non-decreasing function of all its arguments, and the first point of the Lemma comes immediately. The second point is a consequence of the first one, of the inequalities $w_i-\|\bv-\bw\|_\infty\le v_i\le w_i+\|\bv-\bw\|_\infty$, and of
 \[
  \forall c\in\R, \;\forall i\in\Z,\;\;\MepvI(\bv+c)=\MepvI(\bv)+c.
 \]
The third point is done similarly  with $w_i=v_{i+1}\pm L\dx$,
while the last point comes from \eqref{hyp:b_decreasing}, since $\K$ is non-negative.
\end{proof}

Using notation \eqref{eq:MepvI}, the scheme for \eqref{eq:P_ep} is then defined for all 
$n\in\ccl 0,N_T-1\ccr$ and all $i\in\Z$, by
\begin{equation}
 \label{eq:scheme_ep}
 \tag{$\mathcal{S}_\ep$}
\left\{ 
\begin{array}{l}
\ds v^{n+1}_i  =\MepvInp(\bv^n)_i  - \dt \Rcont\left(\tnp,x_i,J^{n+1}\right) \vspace{4pt}\\
  \ds J^{n+1} = \dx \sum\limits_{i\in\Z} \psi(x_i) \e^{-v^{n+1}_i/\ep}.
\end{array}
\right.
\end{equation}
It has been constructed with a quadrature in the integral kernel representing the births in \eqref{eq:P_ep}, and the stiffest term, $J$, is taken implicit to ensure stability in the regime $\ep\to 0$.
In what follows, it will be supposed that $\dz$ is not too large, so that the quadrature in $\MepvI$ is accurate enough. In particular, we will assume that 
\begin{equation}
\label{eq:AP_QuadratureAccurate}
\left| \dz \sum\limits_{k\in\Z} \K(z_k)-1\right|\le \frac{1}{2\bM},
\end{equation}
with $\bM$ defined in \eqref{hyp:b_bounds}.
Note that both $(\bv^{n+1})_{n\in\ccl 0,N_T-1\ccr}$ and $\bJ$ depend on $\ep$, although this dependency is omitted to simplify notations.
The following proposition states the well-posedness of scheme \eqref{eq:scheme_ep}, and uniform w.r.t. $\ep$ stability properties

\begin{prop}[Scheme \eqref{eq:scheme_ep}: existence of solutions and stability properties]
 \label{prop:Scheme_ep_stabilite}
 Suppose that $\vin$ satisfies \eqref{hyp:u0_Lipschitz}-\eqref{hyp:u0_coercive}-\eqref{hyp:u0_minimum}-\eqref{hyp:vin_min}, that $\psi$ satisfies \eqref{hyp:psi}, that $b$ satisfies \eqref{hyp:b_bounds}-\eqref{hyp:b_lipschitz}-\eqref{hyp:b_decreasing},  that $R$ and $\Rcont = R+b$ satisfy \eqref{hyp:R_decreasing}-\eqref{hyp:R_ImIM}-\eqref{hyp:R_bounded}-\eqref{hyp:R_t-Lipschitz},
  and that $\dt$ and $\dx$ are fixed such that
 \begin{equation}
 \label{eq:CFLAP}
 \dx\; \CHLAP(\Lx(T))\le \dt,
 \end{equation} 
 with $\CHLAP$ defined in \eqref{eq:CHLAP}, $\Lx$ defined in Prop. \ref{thm:existence}-\ref{it:Lx}.
  Then, scheme \eqref{eq:scheme_ep} is well-posed: there exists $\bv\in\R^{\Z\times\ccl 0,N_T\ccr}$ satisfying \eqref{eq:scheme_ep}. Moreover,
 there exists an $\ep_0>0$, depending only on the constants arising in the assumptions, and on $\dt$ and $\dx$, such that for all $\ep\in(0,\ep_0)$, the sequence $(\bv^n)_{n\in\ccl 0,N_T\ccr}$ satisfies: 
 \begin{enumerate}
  \item\label{it:AP_Lx} \textbf{Uniform Lipschitz continuity in trait:} For all $n\in\ccl 0,N_T\ccr$, $\bv^n$ enjoys $\Lx(\tn)$-Lipschitz property:
\[
\forall i\in\Z, \;\left| \frac{v^n_{i+1}-v^n_i}{\dx} \right| \le \Lx(\tn),
\]
with $\Lx$ defined in Prop. \ref{thm:existence}-\ref{it:Lx}.
\item\label{it:AP_coercive} \textbf{Uniform bound from below for $\bv^n$:} For all $n\in\ccl 0,N_T\ccr$, for all $i\in\Z$,
\[
v^n_i\ge\ua |x_i| + \ubeta_\tn,
\]
 with $\ua$ defined in \eqref{hyp:u0_coercive}, and for all $t\in[0,T]$,
 \[
 \ubeta_t = \ub - \bM t\dz\sum\limits_{k\in\Z} \K(z_k) \e^{\ua|x_i|} -t K,
\]
with $\ub$, $\bM$ and $K$ defined in \eqref{hyp:u0_coercive}, \eqref{hyp:b_bounds} and \eqref{hyp:R_K}.
 \item\label{it:AP_Jbound} \textbf{Uniform bounds for $\bJ$:} For all $n\in\ccl 0,N_T-1\ccr$, 
 \[
  I_m/2\le J^{n+1}\le 2 I_M,
 \]
 where $I_m$ and $I_M$ are defined in \eqref{hyp:R_ImIM}.
 \item\label{it:AP_min} \textbf{Estimate for $\min\bv^n$:} There exists $c_m$ and $c_M$ such that for all $n\in\ccl 0,N_T\ccr$, 
 \[
  c_m\ep \le\min\bv^n\le c_M \ep,
 \]
and $c_m\le \cm$, $c_M\ge \cM$ depend only on the constants defined in the assumptions, and on $\dx$ and $\dt$. 
 \end{enumerate}
\end{prop}
This proposition and its proof are very close to the stability properties proved in \cite{CalvezHivertYoldas2022} for an AP scheme designed for parabolic Lotka-Volterra equations. The proof is adapted to the present case in Appendix \ref{sec:APProof_lemme}. Using the results of Prop. \ref{prop:Scheme_ep_stabilite}, the following proposition holds, that describes the asymptotic behaviour of \eqref{eq:scheme_ep} when $\ep\to 0$. It  is proved in Appendix \ref{sec:APProof_prop}.

\begin{prop}[Convergence of the scheme \eqref{eq:scheme_ep} when $\ep\to0$]
\label{prop:Scheme_ep_AP}
 Suppose that the assumptions of Theorem \ref{thm:P_ep} are satisfied,
 and that $\dt$ and $\dx$ are fixed such that \eqref{eq:CFLAP} holds. Let $\bv$ and $\bJ$ be the $\ep$-dependent sequences defined by \eqref{eq:scheme_ep}. Then, for all $n\in\ccl 0,N_T-1\ccr$, 
 \[
  \|\bv^{n+1}-\bu^{n+1}\|_\infty\underset{\ep\to0}\longrightarrow 0,
  \;\;\;\text{and}\;\;\; J^{n+1}\underset{\ep\to 0}\longrightarrow I^{n+1},
 \]
where $\bu$ and $\bI$ satisfy
\begin{equation}
\label{eq:scheme_limit}
\tag{S$_0$}
 \left\{ 
 \begin{array}{l}
 \ds u^{n+1}_i=\MzerovInp(\bu^n)_i -\dt \Rcont\left(\tnp,x_i,I^{n+1}\right) \vspace{4pt} \\
\ds  \min\bu^{n+1}=0,
 \end{array}
 \right.
\end{equation}
and for all $I\in\R$, for all $\bu=(u_i)_{i\in\Z}$ and for all $i\in\Z$, $\MzerovI:\R^\Z\to\R^\Z$ is defined by
\begin{equation}
 \label{eq:MzerovI}
 \MzerovI(\bu)_i=u_i-b\left(x_i,I\right)\dt\dz  \left( \sum\limits_{k\ge 0} \K(z_k) \e^{-z_k(u_{i+1}-u_i)/\dx}+ \sum\limits_{k\le -1} \K(z_k) \e^{-z_k(u_i-u_{i-1})/\dx}\right).
\end{equation}
\end{prop}

\begin{rmq}
\label{rmq:AP_Properties_LimitScheme}
The above proposition states the asymptotic behavior of scheme \eqref{eq:scheme_ep} when $\ep\to0$. It is easy to remark that the limit scheme \eqref{eq:scheme_limit} is convergent towards the viscosity solution of  the limit constrained Hamilton-Jacobi equation \eqref{eq:HJ}, with $\Hcont$ defined in \eqref{eq:discussion_H} and $\Rcont$ replaced by $b+R$. Indeed, it can be rewritten with the formalism of \eqref{eq:scheme}, with
\begin{equation}
\label{eq:Heta_AP}
 \dHeAP(p,q)=\dz\left( \K(0) + \sum\limits_{k\le-1} \K(z_k) \e^{-z_k p}+\sum\limits_{k\ge 1} \K(z_k) \e^{-z_k q}\right) - \dz\sum\limits_{k\in\Z} \K(z_k),
\end{equation}
and, if $R$ can be analytically computed,
\begin{equation}
 \label{eq:Reta_AP}
 \dReAP\left(t,x,I\right)=b\left(x,I\right) \dz\sum\limits_{k\in\Z} \K(z_k) + \Rcont\left(t,x,I\right).
\end{equation}
Notice then that $\dHeAP$ belongs to the class \eqref{eq:HP1} of choices  proposed for $\dHe$, in Section \ref{sec:schemes}. Then,  under the hypothesis   of Prop. \ref{prop:Scheme_ep_AP}, the discretization $\dHeAP$, $\dReAP$, $\dt$, $\dx$, $\eta$ is adapted if $\dz$ is small enough (Def. \ref{def:adapted}). In particular, $\bu$ and $\bI$ enjoy the properties of Prop. \ref{thm:existence} and Prop. \ref{thm:limitconv}.
\end{rmq}

\begin{rmq}
 \label{rmq:AP_properties_M0}
 It it worth noticing that $\MzerovI$, defined in \eqref{eq:MzerovI}, satisfies the properties of Lemma \ref{lem:AP_monotonuous}. As for  $\MepvI$, defined in \eqref{eq:MepvI}, this is a consequence of the monotony of $\MzerovI$ in all its arguments, if the constraint \eqref{eq:CFLAP} is satisfied.
\end{rmq}

\begin{rmq}
 Prop. \ref{prop:Scheme_ep_AP} only holds if its assumptions are satisfied, that is if there is no extinction nor explosion of the population. However, numerical tests suggest that scheme \eqref{eq:scheme_ep} is also stable when such behaviors are to be observed. In these cases, the asymptotics of \eqref{eq:P_ep} can be conjectured by considering scheme \eqref{eq:scheme_ep} with small $\ep$.
 We refer to section \ref{sec:TestsNums_AP_CasLimite} for more details.
\end{rmq}
\begin{rmq}
 The scheme \eqref{eq:scheme_ep} is defined for infinite grids in $x$ and $z$, that have to be truncated when implemented. The truncation in $z$ is done according to the decrease of $\K$. It is for instance easy  when $\K$ is Gaussian or compactly supported. Once the grid in $z$ is truncated, scheme \eqref{eq:scheme_ep} can be implemented on a grid in $x$ that is reduced at each time step, as in Remark \ref{rmq:truncation_scheme}. To reduce the computational cost, the values of $\bv$ outside of the grid may also be approximated, thanks to the coercivity of $\bv$, see Remark \ref{rmq:truncation_scheme}. In the numerical tests of Section \ref{sec:numericalsims}, we use linear extrapolation of $\bv$ outside of the grid.
\end{rmq}

\begin{rmq}
Prop. \ref{prop:Scheme_ep_AP} gives the asymptotic behavior of scheme \eqref{eq:scheme_ep} when $\ep\to 0$, and Prop. \ref{thm:limitconv} gives the convergence of the limit scheme \eqref{eq:scheme_limit}. To complete the AP diagram \eqref{APDiagram}, the convergence of scheme \eqref{eq:scheme_ep} for fixed $\ep>0$  is needed. Even though it yields tedious computations, this can be proved similarly as in \cite{CalvezHivertYoldas2022}, with stability estimates coming from the monotony properties \eqref{eq:MepvI}, and truncation errors propagated with implicit function theorem. The main issue here is the truncation error, which is roughly defined as the error made by replacing derivatives by finite differences in the equation. If the solution $v_\ep$ of \eqref{eq:P_ep} is smooth enough, it can be easily estimated with Taylor expansions.
However, contrary to the parabolic case treated in \cite{CalvezHivertYoldas2022}, we do not have any result about the regularity of $v_\ep$.
As a consequence,  the convergence of scheme \eqref{eq:scheme_ep} for fixed $\ep$ is true for smooth solutions with bounded derivatives, and remains an open question otherwise.
\end{rmq}

\section{Convergence of scheme \eqref{eq:scheme}}
\label{sec:Convergence}

In this section, we come back to \eqref{eq:HJ}, to scheme \eqref{eq:scheme}, and to its reformulation \eqref{eq:defudt}, and we prove Prop. \ref{thm:existence} and \ref{thm:limitconv}. These propositions
strongly rely on monotony properties of scheme \eqref{eq:scheme}. Using the formalism of \eqref{eq:defudt}, this means that the operator $\MsJ$ in \eqref{eq:MsJ} is monotonic if a stability condition is satisfied. More precisely, the following result holds

\begin{lem}\label{lem:contmonotonus}
Let $I\in [I_m-1,I_M]$, $s\in (0,\dt]$, $x\in\R$,  $L>0$, and assume that $\dt \CHL(L)\leq \dx$ we have:
\begin{enumerate}
\item If $\udt$ and $v_\dt$ are such that
$
\udt(\tn,x)\le v_\dt(\tn,x)$, $\udt(\tn,x\pm\dx)\le v_\dt(\tn,x\pm\dx),
$
and
\[|\udt(\tn,x\pm\dx) - \udt(\tn,x)|\le L\dx, \;\;\text{and}\;\; |v_\dt(\tn,x\pm\dx)-v_\dt(\tn,x)|\le L\dx,
\]
 then $\MsJ(\udt(\tn,\cdot))(x)\leq\MsJ(v_\dt(\tn,\cdot))(x)$, and $\udt(\tn+s,x)\le v_\dt(\tn+s,x)$.
\item If $\udt(\tn,\cdot)$ is $L$-Lipschitz, then $\MsJ(\udt(\tn,\cdot))$ is also $L$-Lipschitz. If $I^{n+1}\in[I_m-1,I_M]$, we have in addition that $\udt(\tn+s,\cdot)$ is $(L+sK)$-Lipschitz, with $K$ defined in \eqref{hyp:R_K}.
\end{enumerate}
\end{lem}

\begin{proof}
Let $f$  be a $L$-Lipschitz function.
Notice that $\MsJ(f)(x)$ is increasing in $f(x),f(x-\dx),f(x+\dx)$ thanks to \eqref{eq:Heta_defCH} so that the first point is straightforward.  
The second point, is a consequence of  
\[
\forall (x,y)\in\R^2, \; -L|x-y|+f(y)\le f(x)\le L|x-y|+f(y)
\] so that
$\MsJ(f)(x)-\MsJ(f)(y)\leq \MsJ(L|\cdot-y|+f(y))(x)-\MsJ(L|\cdot-y|+f(y))(y)$. Elementary computations give:
\[
\MsJ(L|\cdot-y|+f(y))(x)-\MsJ(L|\cdot-y|+f(y))(y)
\leq L|x-y|,
\]
and \eqref{hyp:R_bounded} then yields the Lipschitz bound of $\udt(\tn+s,\cdot)$.
\end{proof}

\subsection{Proof of Prop. \ref{thm:existence}}
\label{sec:ProofExistence}

In this section, the scheme is considered with formalism \eqref{eq:scheme}, and
the proof of Prop. \ref{thm:existence} is done by induction on the time step $n$. Thanks to 
 \eqref{eq:defbu0}, the properties of Prop. \ref{thm:existence} are true for $n=0$. Let us suppose that the hypothesis of Prop. \ref{thm:existence} are satisfied and that  $\bu^0,\; ...\, ,\; \bu^n\in\R^\Z$ satisfy \eqref{eq:scheme}-\ref{it:Lx}- \ref{it:Lt}-\ref{it:Rpos} and \ref{it:sc} if the convex setting is considered, and show that $\bu^{n+1}$ also satisfies these properties, that $I^{n+1}$ satisfies \ref{it:I_bound} and \ref{it:I_BVbound}, and that $\bu^n$ satisfies \ref{it:pseudomon}.
 
\subsubsection{$\bu^n$ satisfies property \ref{it:pseudomon}} 
\label{sec:ProofExistence_vii}

First of all , remark that the fact that $\bu^n$ satisfies property \ref{it:pseudomon} is straightforward in the flat setting (see Remark \ref{rem:worstcaseisok}), and that is a consequence of \ref{it:sc} in the convex setting. Indeed, in the convex setting, for all $i\in\Z$, we have  
\[
 \frac{u^n_{i+1}-u^n_i}{\dx}= \frac{u^n_i-u^n_{i-1}}{\dx}+ \frac{u^n_{i+1}-2u^n_i+u^n_{i-1}}{\dx} \le \frac{u^n_i-u^n_{i-1}}{\dx}+ \dx w\left(\tn\right),
\]
and the monotony of $\dHe$, see \eqref{hyp:Heta_monotonous}, and the stability condition \eqref{eq:CFLsatisfier} yield
\[
 \dHe\left( \frac{u^n_i-u^n_{i-1}}{\dx},\frac{u^n_{i+1}-u^n_i}{\dx}\right) \ge \dHe \left( \frac{u^n_i-u^n_{i-1}}{\dx}, \frac{u^n_i-u^n_{i-1}}{\dx}+ \dx w\left(\tn\right)\right).
\]
Property \ref{it:pseudomon} is then the consequence of  an inequality of  
convexity for $\dHe$, and of the definition of $\CHL$ in \eqref{eq:Heta_defCH}, namely
\[
  \dHe\left( \frac{u^n_i-u^n_{i-1}}{\dx},\frac{u^n_{i+1}-u^n_i}{\dx}\right) \ge \dHe \left( \frac{u^n_i-u^n_{i-1}}{\dx}, \frac{u^n_i-u^n_{i-1}}{\dx}\right) - \dx w\left(\tn\right) \CHL\left(\Lx(\tn)\right)/\bM.
\]
Eventually, since $\dHe\left(p,p\right)\ge 0$ for all $p$, see \eqref{hyp:Heta_nonnegative}, property \ref{it:pseudomon} holds.
 
 \subsubsection{Existence of $I^{n+1}$}
 \label{sec:ProofExistence_existence}
 
The well posedness of \eqref{eq:scheme} for a given $I^{n+1}$ is straightforward, since it is explicit in $\bu^n$. The only issue could stem from the implicit computation of $I$. Let us introduce, for any $I\in\R$,  
$\bu^{n+1,I}$ such that
\begin{equation}
\label{eq:explicitJ}
u_i^{n+1,I}=u_i^n-\dt b\left(x_i,I\right)\dHe\left(\frac{u_i^n-u_{i-1}^n}{\dx},\frac{u_{i+1}^n-u_{i}^n}{\dx}\right)-\dt \dRe\left(\tnp,x_i,I\right),
\end{equation}
and show that there exists $\mathcal{I}\in\R$, such that $\min\limits_i u^{n+1,\mathcal{I}}_i=0$. 

\begin{rmq}
\label{lem:monotonus}
Coming back to Lemma \ref{lem:contmonotonus}, one can notice that if $\bu^n$, $\boldsymbol{v}^n$ are such that $|u^n_{i+1}-u^n_i|\le L\dx$, $|v^n_{i+1}-v^n_i|\le L\dx$ for all $i\in\Z$, and if $\dt\CHL(L)\le\dx$ then,
\begin{itemize}\pt
 \item If for all $i\in\Z$, $u^n_i\le v^n_i$, then $u^{n+1,I}_i\le v^{n+1,I}_i$ for all $i\in\Z$, and for all $I\in\R$,
 \item for all $i\in\Z$, for all $I\in[I_m-1,I_M]$, $|u^{n+1,I}_{i+1}-u^{n+1,I}_i|\le (L+\dt K)\dx$.
\end{itemize}
\end{rmq}

Thanks to this remark, the bounds of $\bu^n$ in  \eqref{eq:discbounds}  can be propagated using the stability condition \eqref{eq:CFLsatisfier}. It yields, for all $i\in\Z$,  and for all $I\in\R$,
\begin{align}
 \label{eq:ProofExistence_bounds_unpI}
 &\ua |x_i|+\ub_\tn-\dt b_M \dHe\left(\ua,-\ua\right) -\dt \dRe\left( \tnp, x_i,I \right)\\ & \le u^{n+1,I}_i  \le \oa |x_i| + \ob_\tn -\dt b_M \dHe\left(-\oa,\oa\right) -\dt \dRe\left( \tnp, x_i,I  \right), \nonumber
\end{align}
where \eqref{hyp:b_bounds} was used, considering that $\dHe\left(-\oa,\oa\right)\le \dHe\left(0,0\right)=0\le \dHe\left(\ua,-\ua\right)$ thanks to \eqref{hyp:Heta_monotonous} and \eqref{hyp:Heta_zero}. Then, \eqref{hyp:R_bounded} yields the coercivity of $\bu^{n+1,I}$, so that
\begin{equation}
 \label{eq:ProofExistence_phi}
 \phi:I\mapsto \min\limits_{i\in\Z} u^{n+1,I}_i,
\end{equation}
is well-defined for all $I\in\R$. Moreover, the minimum is in fact taken over a finite number of indices, thanks to \eqref{eq:ProofExistence_bounds_unpI} and $\phi$ is hence continuous on all compact sets of $\R$. Let us now consider an index $j$ such that $u^n_j=\min \bu^n=0$, then for all $I\in\R$, 
\begin{align*}
 \phi(I)&\le u^{n+1,I}_j= -\dt b\left(x_j,I\right) \dHe\left(\frac{-u^n_{i-1}}{\dx},\frac{u^n_{i+1}}{x}\right) - \dt \dRe\left(\tnp,x_i,I\right) \\
 &\le  -\dt \bM \dHe\left(-\Lx(T),\Lx(T)\right) - \dt \dRe\left(\tnp,x_i,I\right),
\end{align*}
where the last inequality comes from \eqref{hyp:b_bounds}, \eqref{hyp:Heta_monotonous}, 
and \eqref{eq:disc_lipschitz}. Taylor expansion of $\dRe\left(\tnp,x_i,I\right)$ at $I_m$ and \eqref{hyp:R_decreasing} then yields for $I\le I_m$
\begin{align*}
 \phi(I)\le -\dt \bM \dHe\left(-\Lx(T),\Lx(T)\right)+\dt K^{-1} (I-I_m) -\dt \dRe\left(\tnp,x_i,I_m\right),
\end{align*}
and we use \eqref{hyp:R_ImIM} to get
\[
 \phi(I)\le -\dt \bM \dHe\left(-\Lx(T),\Lx(T)\right)+\dt K^{-1} (I-I_m) \underset{I\to -\infty}\longrightarrow -\infty.
\]
Hence, there exists $I^-\le I_m$ such that $\phi(I^-)<0$. Coming back to \eqref{eq:explicitJ}, one has since $u^n_i\ge 0$ for all $i\in\Z$,
\[
 u^{n+1,I}_i \ge -\dt b_M \dHe\left(\Lx(T),-\Lx(T)\right)-\dt \dRe\left(\tnp,x_i,I\right),
\]
where we used \eqref{hyp:b_bounds}, \eqref{hyp:Heta_monotonous} and \eqref{eq:disc_lipschitz}. Then, as previously, \eqref{hyp:R_decreasing} yields for all $i\in\Z$, and for $I\ge I_M$
\[
 u^{n+1,I}_i\ge -\dt b_M \dHe\left(\Lx(T),\Lx(T)\right) + \dt (I-I_M) K^{-1} \underset{I\to +\infty}{\longrightarrow} +\infty,
\]
so that there exists $I^+\ge I_M$ such that $\phi(I^+)>0$. As $\phi$ is continuous on $[I^-,I^+]$, there exists $\mathcal{I}$ such that $\phi(\mathcal{I})=0$. Altough it may be not uniquely determined, $I^{n+1}$ is then well-defined, and so is $\bu^{n+1}=\bu^{n+1,I^{n+1}}$.

\subsubsection{Properties \ref{it:Rpos} and \ref{it:I_bound}}

The upper bound for $I^{n+1}$ is a consequence of \eqref{hyp:R_ImIM} and \eqref{hyp:R_decreasing}. 
Indeed, since $u^n_i\ge 0$ for all $i\in\Z$, we use  Remark \ref{lem:monotonus}, and \eqref{eq:CFLsatisfier}-\eqref{eq:disc_lipschitz}-\eqref{hyp:Heta_zero} to propagate this  bound. It writes, for all $i\in\Z$
\[
 u^{n+1}_i \ge -\dt \dRe\left(\tnp,x_i,I^{n+1}\right).
\]
Considering now $j\in\Z$ such that $u^{n+1}_j=\min \bu^{n+1}=0$, then yields \ref{it:Rpos}. Moreover, remark that 
\[
 \dRe\left(\tnp,x_j,I^{n+1}\right)\ge 0\ge \dRe\left(\tnp,x_j,I_M\right),
\]
thanks to \eqref{hyp:R_ImIM}, and use \eqref{hyp:R_decreasing} to conclude that
$ I^{n+1}\le I_M$.
To get the bound from below of $\bI$, remark that 
as $\bu^{n+1}$ satisfies \eqref{eq:scheme}, $u^{n+1}_i\ge 0$ for all $i\in\Z$. Let $j\in\Z$ such that $u^n_j=\min\bu^n=0$, we have
\[
 u^{n+1}_j = -\dt b\left(x_j,I^{n+1}\right) \dHe\left( \frac{u^n_j-u^n_{j-1}}{\dx},\frac{u^n_{j+1}-u^n_j}{\dx}\right) -\dt \dRe\left(\tnp,x_j,I^{n+1}\right) \ge 0,
\]
so that property \ref{it:pseudomon} yields 
\begin{equation}
\label{eq:ProofExistence_SimplerIfFlat}
 \dRe\left(\tnp,x_j,I^{n+1}\right) \le \,\dx\;w\left(\tn\right)\,\CHL\left(\Lx\left(\tn\right)\right).
\end{equation}
Since $\dRe\left(\tnp,x_j,I_m\right)\ge 0$ (see \eqref{hyp:R_ImIM}), we obtain
\[
 \dRe\left(\tnp,x_j,I^{n+1}\right) - \dRe\left(\tnp,x_j,I_m\right) \le \dx w\left(\tn\right) \CHL\left(\Lx\left(\tn\right)\right).
\]
Remark now that the result is straightforward if $I^{n+1}\ge I_m$, and suppose then that $I^{n+1}<I_m$. Since $\dRe$ is decreasing w.r.t. $I$, \eqref{hyp:R_decreasing} yields
\begin{equation}
\label{eq:Minoration_Inp1}
 -K^{-1}(I^{n+1}-I_m)\le \dx w\left(\tn\right) \CHL\left(\Lx\left(\tn\right)\right),
\end{equation}
that is $I^{n+1}\ge I_m-1$, thanks to \eqref{eq:smooth_R}.

\begin{rmq}
 It is worth noticing that the right-hand side of \eqref{eq:ProofExistence_SimplerIfFlat} can be replaced by $0$ in the flat setting. In this case, the inequality $I^{n+1}\ge I_m$ holds, as in the continuous case, \cite{BarlesMirrahimiPerthame2009}. However, the convex setting does not preserve this qualitative property.
\end{rmq}

\begin{rmq}
 Note that \eqref{eq:Minoration_Inp1} yields in fact a more precise bound from below for $I^{n+1}$,
 \begin{equation}
 \label{eq:Minoration_Inp1_dx}
 I^{n+1} \ge I_m - \dx w_T \CHL(\Lx(T)),
 \end{equation}
 with $w_T$ defined by \eqref{eq:ProofExistence_wT}, and $\Lx$ in \ref{it:AP_Lx}.
\end{rmq}

\subsubsection{Property \ref{it:I_BVbound}}

The proof of property \ref{it:I_BVbound} is very similar to the one of \ref{it:I_bound}. 
Starting from \eqref{eq:ProofExistence_SimplerIfFlat}, we use property \ref{it:Rpos} to get
\[
  \dRe\left(\tnp,x_j,I^{n+1}\right)-\dRe\left(\tn,x_j,I^n\right)\le  \dx \,w\left(\tn\right)\, \CHL\left(\Lx\left(\tn\right)\right),
\]
where $j\in\Z$ is an index such that $u^n_j=\min\bu^n=0$.
Introduce then $\dRe\left(\tnp,x_j,I^n\right)$ in the left-hand side, and use the Lipschitz-in-time regularity of $\dRe$, see \eqref{hyp:R_t-Lipschitz}, to get
\details{\[
 \dRe\left(\tnp,x_j,I^{n+1}\right) -\dRe\left(\tnp,x_j,I^n\right) \le \LR\dt + \dx w\left(\tn\right) \CHL\left(\Lx\left(\tn\right)\right).
\]}
\[
 \dRe\left(\tnp,x_j,I^{n+1}\right) -\dRe\left(\tnp,x_j,I^n\right) \le K\dt + \dx w\left(\tn\right) \CHL\left(\Lx\left(\tn\right)\right).
\]
One can then notice as above that property \ref{it:I_BVbound} is straightforward if $I^{n+1}\ge I^n$, and use the fact that  $\dRe$ is decreasing w.r.t. $I$, \eqref{hyp:R_decreasing} to obtain if $I^{n+1}<I^n$
\[
 -K^{-1} (I^{n+1}-I^n)\le K \dt + \;\dx\;w\left(\tn\right) \CHL\left(\Lx\left(\tn\right)\right),
\]
which is property \ref{it:I_BVbound}.

\begin{rmq}
The above result can be made more precise in the flat setting. Indeed, according to Remark \ref{rem:worstcaseisok}, the right-hand side of \eqref{eq:ProofExistence_SimplerIfFlat} can be replaced by $0$. It means that the approximation of the Hamiltonian $\Hcont$ in \eqref{eq:HJ} does not make $\bI$ decrease, whereas it can happen in the convex case. This is the justification of the two expressions of $\kappa$ in property \ref{it:I_BVbound}.
\end{rmq}
\begin{rmq}
In the flat setting, and when $\dRe$ does not depend on $t$, the previous computations yield that $I^{n+1}\ge I^n$. As a consequence, $\bI$ is non-decreasing, as is $I$ defined in \eqref{eq:HJ}, see \cite{CalvezLam2020}. We emphasize on the fact that this qualitative property is not preserved in the convex setting, see Section \ref{sec:TestsNums_MonotonyI}.
\end{rmq}

To go on with the proof of Prop. \ref{thm:existence}, property \ref{it:sc} must be proved in the  convex setting (Def. \ref{def:WellChosenScheme}). This is the purpose of the next section.

\subsubsection{Property \ref{it:sc} - convex setting}
\label{sec:semiconcavity}

Suppose now that scheme \eqref{eq:scheme} is in the convex setting, so that $\dHe$ is convex. To show that it preserves an upper bound on the discrete second derivative, let us introduce for all $i\in\Z$, $q^n_i=(u^n_{i+1}-u^n_i)/\dx$. Then   $|q^n_i|\le \Lx(T)$ because of property \ref{it:Lx},  and 
\begin{align}
 \label{eq:ProofExistence_defq}
 q^{n+1}_i=&q^n_i -\frac{\dt}{\dx}\left( b\left(x_{i+1},I^{n+1}\right) \dHe\left(q^n_i,q^n_{i+1}\right) - b\left(x_i,I^{n+1}\right) \dHe\left(q^n_{i-1},q^n_i\right)\right) \\
 & - \frac{\dt}{\dx} \left( \dRe\left(\tnp,x_{i+1},I^{n+1}\right) - \dRe\left(\tnp,x_i,I^{n+1}\right) \right). \nonumber
\end{align}
In order to use the monotony of $\dHe$ 
 (see \eqref{hyp:Heta_monotonous}), 
remark that thanks to property \ref{it:pseudomon}, 
\[
 q^n_{i+1} \le q^n_i + \dx w\left(\tn\right), \;\; \mathrm{and}\;\; q^n_{i+2}\le q^n_{i+1} +\dx w\left(\tn\right),
\]
so that the following inequalities hold
\[
\dHe\left(q^n_{i-1},q^n_i\right) \ge \dHe\left( q^n_i-\dx w\left(\tn\right), q^n_i\right), \;\;\mathrm{and}\;\;   \dHe\left(q^n_{i+1}, q^n_{i+2} \right) \ge \dHe\left( q^n_{i+1}, q^n_{i+1} + \dx w\left(\tn\right)\right).
\]
Using then the convexity of $\dHe$, one has for all $\left(\dd p, \dd q\right)\in\partial \dHe\left(q^n_i,q^n_{i+1}\right)$,
\begin{align}
 \dHe\left(q^n_{i-1},q^n_i\right)&\ge  \dHe\left( q^n_i,q^n_{i+1}\right) - \dx w\left(\tn\right) \dd p -\left( q^n_{i+1}-q^n_i\right) \dd q
 \label{eq:ProofExistence_DL1}
 \\
 \dHe\left( q^n_{i+1},q^n_{i+2}\right) & \ge \dHe\left( q^n_i,q^n_{i+1}\right)  + \left(q^n_{i+1}-q^n_i\right) \dd p+ \dx w\left(\tn\right) \dd q,
 \label{eq:ProofExistence_DL2}
\end{align}
and we inject these inequalities in  $q^{n+1}_{i+1}-q^{n+1}_i$ expressed thanks to \eqref{eq:ProofExistence_defq}. We obtain
\begin{align}
\label{eq:ProofExistence_Decomp1}
 q^{n+1}_{i+1} - q^{n+1}_i \le & \mathfrak{A} + \mathfrak{B} + \mathfrak{C} + \mathfrak{D} + \mathfrak{E}
\end{align}
where 
\begin{align}
 \mathfrak{A}&= -\frac{\dt}{\dx}\left[ \dRe\left(\tnp,x_{i+2}, I^{n+1}\right) - 2 \dRe\left(\tnp, x_{i+1}, I^{n+1}\right) + \dRe\left(\tnp, x_i, I^{n+1}\right)\right]
 \label{eq:ProofExistence_frakA}
 \\
 \mathfrak{B}&=  -\frac{\dt}{\dx} \left[ b\left(x_{i+2},I^{n+1}\right) - 2b\left(x_{i+1},I^{n+1}\right) + b\left(x_i, I^{n+1}\right)\right] \dHe\left(q^n_i, q^n_{i+1}\right)
 \label{eq:ProofExistence_frakB}
 \\
 \mathfrak{C}&= - \frac{\dt}{\dx}\left[ b\left(x_{i+2}, I^{n+1}\right)-b\left(x_i,I^{n+1}\right)\right] \dx w\left(\tn\right)\dd q
 \nonumber
 \\
 \mathfrak{D}&= - \frac{\dt}{\dx}\left[ b\left(x_{i+2}, I^{n+1}\right)-b\left(x_i,I^{n+1}\right)\right] \left(q^n_{i+1}-q^n_i\right) \dd p
 \nonumber
 \\
 \mathfrak{E}&=  \left[\frac{\dt}{\dx} b\left(x_i,I^{n+1}\right)  \left(\dd p -\dd q\right)\right] \dx w\left(\tn\right) + \left[1- \frac{\dt}{\dx} b\left(x_i, I^{n+1}\right) \left(\dd p -\dd q\right) \right] \left(q^n_{i+1}-q^n_i\right),
 \nonumber 
\end{align}
and each term can be estimated separately. First of all, the $W^{2,\infty}$ estimate of $\dRe$ in  \eqref{hyp:R_bounded}, and the property $I_m-1\le I^{n+1}\le I_M$ yield
\begin{equation}
\label{eq:ProofExistence_frakA_majo}
 \left|\mathfrak{A}\right| \le \dt \dx K,
\end{equation}
and similarly, \eqref{hyp:b_lipschitz}-\eqref{eq:Heta_defCH} give
\[
 \left|\mathfrak{C}\right|\le 2\dt  \dx \CHL\left(\Lx\left(T\right)\right) w\left(\tn\right).
\]
The estimate of $\mathfrak{B}$ is a consequence of \eqref{hyp:b_bounds} and of
\[
 \dHe\left(-\Lx(T),\Lx(T)\right)
 \le \dHe\left(q^n_{i+1}, q^n_i\right)
 \le \dHe\left(\Lx(T),-\Lx(T)\right),
\]
so that 
\begin{equation}
\label{eq:ProofExistence_frakB_majo}
 \left|\mathfrak{B}\right|\le \dt\dx\bM\max\left(-\dHe\left(-\Lx(T),\Lx(T)\right), \dHe\left(\Lx(T),-\Lx(T)\right)\right).
\end{equation}
For the term $\mathfrak{E}$, it is worth noticing that since $\dHe$ is increasing in its first variable, and decreasing in its second variable, then $\dd p-\dd q\ge 0$ for all $(\dd p, \dd q)\in\partial \dHe\left(q^n_i, q^n_{i+1}\right)$. Moreover, using the stability condition \eqref{eq:CFLsatisfier} one has
\[
 \forall \left(\dd p, \dd q\right)\in\partial \dHe\left(q^n_i,q^n_{i+1}\right),\;\;\; 0\le \frac{\dt}{\dx} b\left(x_i,I^{n+1}\right) \left(\dd p - \dd q\right) \le 1,
\]
meaning that $\mathfrak{E}$ is a convex combination of $\dx w\left(\tn\right)$ and $q^n_{i+1}-q^n_i$.  It yields
\[
 \mathfrak{E}\le \dx\; w\left(\tn\right).
\]
However, $\mathfrak{D}$ cannot be estimated in all cases. Indeed, if $b\left(x_{i+2}, I^{n+1}\right)-b\left(x_i,I^{n+1}\right)\le 0$, it can be rewritten using \ref{it:pseudomon},
\[
 \mathfrak{D}= \frac{\dt}{\dx} \left|b\left(x_{i+2}, I^{n+1}\right)-b\left(x_i,I^{n+1}\right)\right| \left(q^n_{i+1} - q^n_i\right) \dd p \le 2\dt \dx   \CHL\left(\Lx\left(T\right)\right) w\left(\tn\right),
\]
but otherwise, the only estimate that can be used is $
 \mathfrak{D}\le 0$, 
if $q^n_{i+1}\ge q^n_i$. We can now gather these results to conclude that if $q^n_{i+1}\ge q^n_i$,
\details{ 
 \begin{align*}
  \frac{q^{n+1}_{i+1}-q^{n+1}_i}{\dx} \le&
  \dt \left[K + \bM  \max\left(-\dHe\left(-\Lx(T),\Lx(T)\right), \dHe\left(\Lx(T),-\Lx(T)\right)\right)\right] \\
  &+ w\left(\tn\right)\left[1+ 4\dt    \CHL\left(\Lx\left(T\right)\right)  \right],
 \end{align*}
}
 \begin{equation}
 \label{eq:ProofExistence_Estimate_dq_p}
  \frac{q^{n+1}_{i+1}-q^{n+1}_i}{\dx} \le
 \dt \gamma + w\left(\tn\right) \left(1+\beta\dt\right),
 \end{equation}
with $\gamma$, $\beta$ defined in \eqref{eq:ProofExistence_alpha}-\eqref{eq:ProofExistence_beta}.
In the other case, namely when $q^n_{i+1}<q^n_i$, we reformulate \eqref{eq:ProofExistence_Decomp1} as
\begin{align*}
 q^{n+1}_{i+1} - q^{n+1}_i \le q^n_{i+1}-q^n_i + \mathfrak{A} + \mathfrak{B} + \mathfrak{F} + \mathfrak{G},
\end{align*}
where $\mathfrak{A}$ and $\mathfrak{B}$ have been defined in \eqref{eq:ProofExistence_frakA}-\eqref{eq:ProofExistence_frakB}, and
\begin{align*}
\mathfrak{F}&= -\frac{\dt}{\dx} b\left(x_i,I^{n+1}\right) \left[ \dHe\left( q^n_{i-1},q^n_i\right) - \dHe\left( q^n_i,q^n_{i+1}\right)\right] \\
 \mathfrak{G}& = -\frac{\dt}{\dx}b\left(x_{i+2},I^{n+1}\right) \left[ \dHe\left(q^n_{i+1}, q^n_{i+2}\right) - \dHe\left( q^n_i, q^n_{i+1} \right) \right].
\end{align*}
Use again \eqref{eq:ProofExistence_DL1}-\eqref{eq:ProofExistence_DL2}, to get for all $(\dd p, \dd q)\in\partial\dHe\left(q^n_i,q^n_{i+1}\right)$
\details{ 
\begin{align*}
 \dHe\left(q^n_{i-1},q^n_i\right)&\ge  \dHe\left( q^n_i,q^n_{i+1}\right) - \dx w\left(\tn\right) \dd p -\left( q^n_{i+1}-q^n_i\right) \dd q
  \\
 \dHe\left( q^n_{i+1},q^n_{i+2}\right) & \ge \dHe\left( q^n_i,q^n_{i+1}\right)  + \left(q^n_{i+1}-q^n_i\right) \dd p+ \dx w\left(\tn\right) \dd q,
\end{align*}
\begin{align*}
 \mathfrak{F}&\le -\frac{\dt}{\dx}b\left(x_i,I^{n+1}\right) \left[ \dHe\left( q^n_i,q^n_{i+1}\right) - \dx w\left(\tn\right) \dd p -\left( q^n_{i+1}-q^n_i\right) \dd q- \dHe\left( q^n_i, q^n_{i+1} \right) \right]
 \\
 \mathfrak{G} &\le -\frac{\dt}{\dx}b\left(x_{i+2},I^{n+1}\right) \left[
 \dHe\left( q^n_i,q^n_{i+1}\right)  + \left(q^n_{i+1}-q^n_i\right) \dd p+ \dx w\left(\tn\right) \dd q - \dHe\left( q^n_i, q^n_{i+1} \right)
 \right] 
\end{align*}
}
\begin{align*}
 \mathfrak{F} &\le \frac{\dt}{\dx}b\left(x_i,I^{n+1}\right) \left[ \dx w\left(\tn\right) \dd p + (q^n_{i+1}-q^n_i)\dd q \right] \\
 \mathfrak{G} &\le \frac{\dt}{\dx}b\left(x_{i+2},I^{n+1}\right) \left[ \dx w\left(\tn\right) \left(-\dd q\right) - \left(q^n_{i+1}-q^n_i\right)\dd p \right],
\end{align*}
and use these estimates,  the ones of $\mathfrak{A}$, $\mathfrak{B}$ in \eqref{eq:ProofExistence_frakA_majo}-\eqref{eq:ProofExistence_frakB_majo}, the definition of $\gamma$ in \eqref{eq:ProofExistence_alpha} and the one of $\CHL$ in \eqref{eq:Heta_defCH} to obtain
\details{
\begin{align*}
 q^{n+1}_{i+1}-q^{n+1}_i \le& q^n_{i+1}-q^n_i - \frac{\dt}{\dx}\left[ b\left(x_i,I^{n+1}\right) \left(-\dd q\right) + b\left(x_{i+2},I^{n+1}\right) \dd p  \right] \left(q^n_{i+1}-q^n_i\right)
 \\&+ \dt \gamma + 2\dt \bM w\left(\tn\right) \left(\dd p -\dd q\right)
\end{align*}
\begin{align*}
 q^{n+1}_{i+1}-q^{n+1}_i \le&  \left[1 - \frac{\dt}{\dx}\left( b\left(x_{i+2},I^{n+1}\right) \dd p - b\left(x_i,I^{n+1}\right) \dd q  \right)\right] \left(q^n_{i+1}-q^n_i\right)
 \\&+ \dt \gamma + 2\dt \bM w\left(\tn\right) \left(\dd p -\dd q\right),
\end{align*}
where we remarked that $\dd p - \dd q\ge 0$, since $\dHe$ is increasing in its first variable and decreasing in the second. 
}
\begin{align*}
 q^{n+1}_{i+1}-q^{n+1}_i \le&  \left[1 - \frac{\dt}{\dx}\left( b\left(x_{i+2},I^{n+1}\right) \dd p - b\left(x_i,I^{n+1}\right) \dd q  \right)\right] \left(q^n_{i+1}-q^n_i\right)
 \\&+ \dt \gamma + 2\dt  w\left(\tn\right) \CHL\left(\Lx\left(T\right)\right).
\end{align*}
Eventually, the stability condition \eqref{eq:CFLsatisfier} yields that 
\[
 1 - \frac{\dt}{\dx}\left( b\left(x_{i+2},I^{n+1}\right) \dd p - b\left(x_i,I^{n+1}\right) \dd q  \right)\ge 0,
\]
and since $q^n_{i+1}-q^n_i<0$, the above estimate can be simplified
\details{
\[
 \frac{q^{n+1}_{i+1}-q^{n+1}_i}{\dx} \le \dt \gamma + 2\frac{\dt}{\dx}  w\left(\tn\right)  \CHL(\Lx(T)),
\]
and the stability condition \eqref{eq:CFLsatisfier} once again gives}
\[
 \frac{q^{n+1}_{i+1}-q^{n+1}_i}{\dx} \le \dt \gamma +   w\left(\tn\right),
\]
where the stability condition \eqref{eq:CFLsatisfier} also gave an upper bound of the last term.
Hence, estimate \eqref{eq:ProofExistence_Estimate_dq_p} holds in fact in both cases, and  
\[
 w\left(\tn\right)\le \e^{\beta\tn} \left(\gamma\tn + w\left(0\right)\right),
\]
so that property \ref{it:sc} is satisfied.

\subsubsection{Conclusion}
\label{sec:ProofExistence_conclusion}

To conclude the induction, remark that properties  \ref{it:Lx} and \ref{it:Lt} are straightforward using Remark \ref{lem:monotonus}, item \ref{it:I_bound} and \eqref{hyp:R_bounded}. Let us now discuss the fact that $I^{n+1}$ could be not uniquely determined as the solution of $\phi(I)=0$, where $\phi$ is defined in \eqref{eq:ProofExistence_phi}. It is worth noticing that it is not a crucial issue in the proof above, since any suitable $I^{n+1}$ can be chosen in Section \ref{sec:ProofExistence_existence}. Moreover, even if $I^{n+1}$ is not uniquely determined, Prop. \ref{thm:limitconv} holds.

Actually, the fact that $I^{n+1}$ could be not uniquely determined is a purely numerical phenomenon, and it is only associated to the convex setting defined in  Def. \ref{def:WellChosenScheme}. Indeed, as $\dHe\ge 0$, in the flat setting (see Remark \ref{rem:worstcaseisok}), \eqref{hyp:b_decreasing} and \eqref{hyp:R_decreasing} yield that for all $i\in\Z$ and for all $n\in\ccl 0,N_t-1\ccr$,
\begin{equation}
\label{eq:Imapstobu}
 I\mapsto u^{n+1,I}_i,
\end{equation}
with $\bu^{n+1,I}$ defined in \eqref{eq:explicitJ},
is increasing. As a consequence, if $I<J$ are such that $\min\bu^{n+1,I}=\min\bu^{n+1,J}=0$, then for all $i\in\Z$, 
\[
 u^{n+1,I}_i < u^{n+1,J}_i, \;\; u^{n+1,I}_i\ge 0, \;\; u^{n+1,J}_i\ge 0, 
 \]
and considering an index $j$ such that $u^{n+1,J}_j = \min\bu^{n+1,J} =0$ in the above expressions gives a contradiction. This remark also gives a sufficient condition for $I^{n+1}$ to be uniquely determined in the convex setting. 
Indeed, coming back to the definition of $\bu^{n+1,I}$ in \eqref{eq:explicitJ}, one can remark that the finite-differences in $\dHe$ are bounded thanks to the Lipschitz bounds of $\bu^n$ stated in Prop. \ref{thm:existence}-\ref{it:Lx}. Hence,  \eqref{eq:Imapstobu} is increasing if \eqref{eq:ConditionForUniqueness} is satisfied.

However,  this condition is too restrictive in practice, since the bound of $\dHe$ seems to be far too large around the minimal points of $\bu^{n+1,I}$. In fact, we observed in the numerical tests that $\bu^{n+1,I}$ is flat around its minimal points,
although we were not able to quantify it precisely.
As a consequence, 
\[
 \dHe\left(\frac{u^n_i-u^n_{i-1}}{\dx}, \frac{u^n_{i+1}-u^n_i}{\dx}\right),
\]
is small around the minimal points of $\bu^n$, as are the finite differences in $\dHe$. The condition \eqref{eq:ConditionForUniqueness} can then be relaxed in practice, and we could not exhibit a test case in which $I^{n+1}$ was not uniquely determined.

\subsection{Proof of Prop. \ref{thm:limitconv}}
\label{sec:ProofConvergence}

In this section, we consider $\udt$ and $\Idt$ defined by \eqref{eq:defJdt} and \eqref{eq:scheme}. As they are defined with a reformulation of scheme \eqref{eq:scheme}, such that $\Idt$ is a constant by parts reconstruction of $\bI$, and that $\udt$ coincides with $\bu$ on the grid points, it is natural to expect $(\udt,\Idt)$ to satisfy similar properties as in Prop. \ref{thm:existence}. Indeed, the following result holds

\begin{lem}\label{prop:Noel}
Let $\dHe$, $\dRe$, $\dt$, $\dx$, $\eta$ be an adapted discretization (Definition\ref{def:adapted}), and let $\udt$ and $\Jdt$ be  defined by \eqref{eq:defudt} and \eqref{eq:defJdt}. We have:
\begin{enumerate}
\item \label{it:udt_stability_lipspace}\textbf{Lipschitz-in-space property:} for all $t\in[0,T]$, $\udt(t,\cdot)$ is $\Lx(t)$-Lipschitz, with $\Lx$ defined in Prop. \ref{thm:existence}.
\item \label{it:udt_stability_liptime} \textbf{Lipschitz-in-time property:} for all $x\in\R$, $\udt(\cdot,x)$ is $\Lt$-Lipschitz with $\Lt$ given by
$\Lt=\bM \dHe\left(\LxT,-\LxT\right)+K$,
$\bM$ defined in \eqref{hyp:b_bounds},
and $K$ in \eqref{hyp:R_K}.
\item \label{it:udt_stability_boundsudt}\textbf{Bounds for $\udt$:} for all $(t,x)\in[0,T]\times\R$,
\[
\ua|x|+\ub_t\leq \udt(t,x)\leq \oa|x|+\ob_t,
\]
with $\ua$, $\oa$ defined in \eqref{hyp:u0_coercive}, and $\ub_t$, $\ob_t$ in Prop. \ref{thm:existence}.
\item \label{it:udtstability_boundsIdt}\textbf{Bounds for $\Jdt$:} for all $t\in[0,T]$, $I_m-1\leq\Jdt(t)\leq I_M$, with $I_m$ and $I_M$ defined in \eqref{hyp:R_ImIM}.
\end{enumerate}
\end{lem}

\begin{proof}
As $\Idt$ is constant by parts and takes only the values of $\bI$, \ref{it:udtstability_boundsIdt} is a reformulation of Prop. \ref{thm:existence}-\ref{it:I_bound}. Items \ref{it:udt_stability_lipspace} and \ref{it:udt_stability_boundsudt} can then be proved by induction using Lemma \ref{lem:monotonus}, and similarly to what is done in the proof of Prop. \ref{thm:existence}.
Item \ref{it:udt_stability_liptime} is a consequence of \ref{it:udt_stability_lipspace}. More precisely,
we only need to remark that
for all $n\in\ccl0,N_T-1\ccr$, $x\in\R$,
\[
\left|b\left(x,I^{n+1}\right)\dHe\left(\frac{\udt(\tn,x)-\udt(\tn,x-\dx)}{\dx},\frac{\udt(\tn,x+\dx)-\udt(\tn,x)}{\dx}\right)+\dRe\left(\tnp,x,I^{n+1}\right)\right|\le \Lt.
\]
This inequality is true thanks to \eqref{hyp:b_bounds}, \eqref{hyp:Heta_monotonous}, \eqref{hyp:R_bounded}, and items \ref{it:udt_stability_lipspace} and \ref{it:udtstability_boundsIdt}.
\end{proof}

The proof of Prop. \ref{thm:limitconv} then follows the lines of the proof of \cite[Prop. $2.3$]{CalvezHivertYoldas2022}, with technical adaptations to the case considered here.  It is indeed a simple case of the general framework that is dealt with in this paper.

In Section \ref{sec:limits}, we show that $\udt$, and $\Jdt$ have limits when $\dt$, $\dx$, and $\eta$ go to $0$, and some properties of these limits are stated.  Section \ref{ssec:smoothJ} is devoted to the regularization of $I$ and $\Idt$,   and associated numerical and viscosity solutions are introduced. Finally, Section \ref{sec:ident} is devoted to the identification of the limits derived in Section \ref{sec:limits}, as the unique viscosity solution of \eqref{eq:HJ}.

\subsubsection{Limits of $\udt$, and $\Jdt$}\label{sec:limits}

To start with, let us remark that Lemma \ref{prop:Noel} provides enough compactness to extract a limit. More precisely,
\begin{lem}\label{lem:extraction}
Under the assumptions of Prop. \ref{thm:limitconv}, there exists $\ul\in\Cc([0,T],\R)$, $\Jl$,$\Jlp\in BV(0,T)$, and a subsequence $(\dt^n)_n$ of the discretization such that $\dt^n\to_{n\to +\infty} 0$, and
\begin{enumerate}
\item \label{it:extraction_udt} $\udtn$ converges locally uniformly towards $\ul$: for all compact $\Omega$ of $\R$,
\[
\|\udtn-\ul\|_{L^\infty([0,T]\times \Omega)}\underset{n\to +\infty}\longrightarrow 0.
\]
\item \label{it:extraction_minu} $\ul$ is $\LxT$-Lipschitz in trait, $\Lt$-Lipschitz in time, with $\Lx$ defined in Prop. \ref{thm:existence}-\ref{it:Lx}, and $\Lt$ in Lemma \ref{prop:Noel}-\ref{it:udt_stability_liptime}. Moreover, $\ul$ satisfies:
\[
\forall t\in[0,T],\; \min_x \ul(t,x)=0.
\]
\item \label{it:extraction_Jdt}$(\Jdtn)_n$ converges almost everywhere on $(0,T)$ towards $\Jl$ and $\Jlp$.
\item \label{it:extraction_BoundsJ} $\Jl$ is lower semi-continuous, $\Jlp$ is upper semi-continuous, and both satisfy:
\[
I_m\le \Jl\le I_M,\;\;\;\;\text{and}\;\;\;\; I_m\le \Jlp\le I_M.
\]
\end{enumerate}
\end{lem}

\begin{proof}
This lemma is equivalent to \cite[lemma 4.3]{CalvezHivertYoldas2022}. The proof is similar, using Ascoli theorem for the convergence of $\udt$ in \ref{it:extraction_udt}, and Helly's selection theorem for $\Jdt$ in \ref{it:extraction_Jdt}. Indeed, one can notice that Lemma \ref{thm:existence}-\ref{it:I_bound}-\ref{it:I_BVbound} provides an uniform bound of $\Jdt$ in total variation. Actually, it is piecewise constant, equal to $I^{n+1}$ on $(\tn,\tnp)$ for all $n\in\ccl0,N_T-1\ccr$, uniformly bounded on $(0,T)$, and with uniformly bounded total height of jumps down.
\details{
Then, one has
\[
 TV(\Jdt)=\sum\limits_{\atop{n=0}{I^{n+1}\ge I^n}}^{N_T-1} (I^{n+1}-I^n) -\sum\limits_{\atop{n=0}{I^{n+1}<I^n}}^{N_T-1} (I^{n+1}-I^n),
\]
with
\begin{align*}
0\le-\sum\limits_{\atop{n=0}{I^{n+1}<I^n}}^{N_T-1} (I^{n+1}-I^n) \le \kappa T,
\end{align*}
where $\kappa$ is defined in Prop. \ref{thm:existence}-\ref{it:I_BVbound}. On the other hand,
\[
 I^{N_T}-I^0=\sum\limits_{n=0}^{N_T-1} (I^{n+1}-I^n)= \sum\limits_{\atop{n=0}{I^{n+1}\ge I^n}}^{N_T-1} (I^{n+1}-I^n) +\sum\limits_{\atop{n=0}{I^{n+1}<I^n}}^{N_T-1} (I^{n+1}-I^n),
\]
so that Prop. \ref{thm:existence}-\ref{it:I_bound} yields
\[
 0\le \sum\limits_{\atop{n=0}{I^{n+1}\ge I^n}}^{N_T-1} (I^{n+1}-I^n)\le I_M-I_m+1+\kappa T,
\]
and hence
\[
TV(\Jdt)\le I_M-I_m+1+2\kappa T.
\]
}
Item \ref{it:extraction_minu} is then a consequence of the local uniform convergence of the uniformly coercive and Lipschitz functions $(\udtn)_n$. Eventually, \ref{it:extraction_BoundsJ} comes from \eqref{eq:Minoration_Inp1_dx}.
\end{proof}

\begin{rmq}
 \label{rmq:Subsequence} In what follows, the mention $\dt\to 0$ will always refer to a subsequence for which the convergences of Lemma \ref{lem:extraction} hold true.
\end{rmq}

Note that neither $\Jl$, nor $\Jdt$, are  continuous on $(0,T)$, this behavior is showcased in Section \ref{sec:TestsNums_MonotonyI}. To be able to use the standard framework of Hamilton-Jacobi equations with Lipschitz Hamiltonian, we introduce  in the next subsection a Lipschitz regularization of $\Jl$, and $\Jdt$, and we regularize $\ul$ and $\udt$ accordingly.

\subsubsection{Regularization of $\Idt$ and $I_0$}\label{ssec:smoothJ}

Following \cite{CalvezHivertYoldas2022,CalvezLam2020,AmbrosioFuscoPallara2000}
let us introduce, for all $\dt>0$, $k>0$,  and $t\in[0,T]$:
\begin{equation}\label{eq:Jdtk}
\Jdtkm=\inf_{s\in[0,T]} \Jdt(s)+k|t-s|,\;\;\;\;\text{and}\;\;\;\; \Jdtkp=\sup_{s\in[0,T]} \Jdt(s)-k|t-s|.
\end{equation}
Similarly, for $\Jl$, we let for all $k>0$, $t\in[0,T]$:
\begin{equation}\label{eq:Jlk}
\Jlkm=\inf_{s\in[0,T]} \Jl\left(s\right)+k|t-s|,\;\;\;\;\text{and}\;\;\;\; \Jlkp=\sup_{s\in[0,T]} \Jlp(s)-k|t-s|.
\end{equation}
These choices of regularization allow for stronger convergence properties than established in previous section. The needed results are provided by \cite[Lemma 4.4]{CalvezHivertYoldas2022}, and they are recalled in the following lemma:

\begin{lem}\label{lem:imporvedconvergence} Let $\Jdtkm$, $\Jlkm$, $\Jdtkp$,$\Jlkp$, and $\Jdt$ be defined by \eqref{eq:Jdtk}, \eqref{eq:Jlk} and \eqref{eq:defJdt}. Suppose that the hypothesis of Lemma \ref{lem:extraction} are satisfied, so that $\Jl$, $\Jlp$ are well-defined. Then, the following results hold:
\begin{enumerate}
\item For all $\dt>0$, for all $k>0$,
\[
 \begin{array}{l c l}
  \ds I_m-1\leq \Jdtkm\leq I_M , & \;\; & \ds I_m\leq \Jlkm\leq I_M, \vspace{4pt} \\
   \ds I_m-1\leq \Jdtkp\leq I_M , & \;\; & \ds I_m\leq \Jlkp\leq I_M,
 \end{array}
\]
with $I_m$ and $I_M$ defined in \eqref{hyp:R_ImIM}.
\item For all $\dt>0$, for all $t\in[0,T]$, the following convergences hold when $k\to\infty$,
\[
\Jdtkm\nearrow\Jdt, \qquad\Jlkm\nearrow \Jl, \qquad \Jdtkp\searrow\Jdt, \qquad\Jlkp\searrow \Jlp.
\]
\item For all $\dt>0$, for all $k>0$, $\Jdtkm$, $\Jlkm$, $\Jdtkp$ and $\Jlkp$ are $k$-Lipschitz functions on $[0,T]$.
\item For all $k>0$,  $\|\Jdtkm-\Jlkm\|_\infty\to_{\dt\to 0}0$, and $\|\Jdtkp-\Jlkp\|_\infty\to_{\dt\to0}0$.
\end{enumerate}
\end{lem}
The proof of this lemma is only computational, and partly done in \cite{CalvezHivertYoldas2022}. The last point might seem surprising, but we emphasize on the fact that the convergences are not uniform in $k$.

Using these regularized versions of $I$  as parameters, one can then define associated viscosity solutions of \eqref{eq:HJ} and solutions of scheme \eqref{eq:scheme}. More precisely, for $t>0$ and $x\in\R$, let us introduce $\vk$, and $\wk$ the viscosity solutions of
\begin{align}\label{eq:vk}
\partial_t \vk\left(t,x\right)+b\left(x,\Jlkm\left(t\right)\right)\Hcont\left(\nabla_x\vk\right)\left(t,x\right)+\Rcont\left(t,x,\Jlkm\left(t\right)\right)&=0\\
\label{eq:wk}
\partial_t \wk\left(t,x\right)+b\left(x,\Jlkp\left(t\right)\right)\Hcont\left(\nabla_x\wk\right)\left(t,x\right)+\Rcont\left(t,x,\Jlkp\left(t\right)\right)&=0,
\end{align}
with initial condition $\uin$.
Similarly, let $\dt>0$, $n\in\ccl0,N_T-1\ccr$, $s\in[0,\dt]$, and define $\vdtk$, and $\wdtk$ as in \eqref{eq:scheme}, by
\begin{align}\label{eq:vdtk}
\vdtk\left(\tn+s,\cdot\right)=\mathcal{M}_s^{\Jdtkm\left(\tn+s\right)}\left(\vdtk\left(\tn,\cdot\right)\right)-s\dRe\left(\tn+s,\;\cdot\;,\Jdtkm\left(\tn+s\right)\right)
\\ \wdtk\left(\tn+s,\cdot\right)=\mathcal{M}_s^{\Jdtkp\left(\tn+s\right)}\left(\wdtk\left(\tn,\cdot\right)\right)-s\dRe\left(\tn+s,\;\cdot\;,\Jdtkp\left(\tn+s\right)\right), \label{eq:wdtk}
\end{align}
once again with initialization $\uin$.
One can notice that, in the four cases, the constraint $\min u=0$ is no longer enforced.
Properties of $\vk$, $\wk$, $\vdtk$ and $\wdtk$  are stated in the following lemma,

\begin{lem} \label{lem:propofregularizedsolutions}
Suppose that the hypothesis of Prop. \ref{thm:existence} are satisfied.
Let $\vk$, $\wk$, $\vdtk$, $\wdtk$ and $\udt$ be defined by   \eqref{eq:vk}-\eqref{eq:wk}-\eqref{eq:vdtk}-\eqref{eq:wdtk} and \eqref{eq:defudt}, and $\Jl$, and $\Jlp$ defined in Lemma \ref{lem:extraction}. Then,
\begin{enumerate}
\item \label{it:approximation_Lipx} \textbf{Uniform Lipschitz continuity in trait:} for all $t\in[0,T]$,  for all $k>0$, for all $\dt>0$, $\vk\left(t,\cdot\right)$, $\wk\left(t,\cdot\right)$, $\vdtk\left(t,\cdot\right)$,  and $\wdtk\left(t,\cdot\right)$ are $\LxT$-Lipschitz, with $\Lx$ defined in Prop. \ref{thm:existence}-\ref{it:Lx}.
\item  \label{it:approximation_Lt}\textbf{Uniform Lipschitz continuity in  time:} for all $x\in\R$, for all $k>0$, for all $\dt>0$, $\vk\left(\cdot,x\right)$, $\wk\left(\cdot,x\right)$, $\vdtk\left(\cdot,x\right)$, and $\wdtk\left(\cdot,x\right)$ are $\Lt$-Lipschitz, with $\Lt$ defined in Lemmas \ref{prop:Noel}-\ref{it:Lt}.
\item \label{it:approximation_bounds}\textbf{Uniform bounds:} for all $t\in[0,T]$, for all $x\in\R$, for all $k>0$, for all $\dt>0$,
\begin{align*}
\ua|x|+\ub_t \leq \vdtk\left(t,x\right)&\le \oa|x|+\ob_t \\
\ua|x|+\ub_t \leq \wdtk\left(t,x\right)&\le \oa|x|+\ob_t \\
\ua|x|+\ub-t\left(K+\sup_{p\in[-\ua,\ua]} H(p)\right) &\le \vk\left(t,x\right)\le \oa|x|+\ob +tK \\
\ua|x|+\ub-t\left(K+\sup_{p\in[-\ua,\ua]} H(p)\right) &\le \wk\left(t,x\right)\le \oa|x|+\ob +tK,
\end{align*}
with $\ua$, $\oa$, $\ub$, $\ob$ defined in \eqref{hyp:b_bounds}, $\ub_t$ and $\ob_t$ in Prop. \ref{thm:existence}-\ref{it:Lt}, and $K$ in \eqref{hyp:R_K}.
\item \label{it:approximation_limitcont} \textbf{Monotony of the approximation:} for $k\to \infty$ we have $\vk\nearrow\vl$, and $\wk\searrow\wl$, pointwise on $[0,T]\times\R$, where $\vl$, $\wl$ are  viscosity solution  of respectively
\begin{align}
\label{eq:defvl}
\partial_t \vl+b\left(x,\Jl\right)\Hcont\left(\nabla_x\vl\right)+\Rcont\left(t,x,\Jl\right)=0 \\
\label{eq:defwl}
\partial_t \wl+b\left(x,\Jlp\right)\Hcont\left(\nabla_x\wl\right)+\Rcont\left(t,x,\Jlp\right)=0,
\end{align}
with initialization $\uin$.
\item \label{it:approximation_inequality} We have $\vdtk\leq\udt\leq\wdtk$.
\end{enumerate}
\end{lem}

\begin{rmq}
\label{rmq:ToConclude}
 Item \ref{it:approximation_limitcont} and the fact that $\Jl=\Jlp$ a.e., yield that $\vl=\wl$ thanks to  Theorem \ref{thm:CalvezLam2020}.
\end{rmq}

\begin{proof}
Items \ref{it:approximation_Lipx}-\ref{it:approximation_Lt} are classical properties of viscosity solutions, \ref{it:approximation_bounds} stems from a comparison principle, and \ref{it:approximation_limitcont} from \cite{CalvezLam2020}. Eventually, \ref{it:approximation_inequality} is a consequence of Lemma \ref{lem:contmonotonus}, and of the fact that $I\mapsto \mathcal{M}_s^I$ is non-decreasing, with $\mathcal{M}_s^I$ defined in \eqref{eq:MsJ}. Indeed, if $\vdtk(t^n,\cdot)\leq\udt(t^n,\cdot)\leq\wdtk\left(t^n,\cdot\right)$, and since \eqref{eq:CFLsatisfier} is satisfied, then
\[
\mathcal{M}_s^{\Jdtkm\left(t^n\right)}\vdtk\left(t^n,\cdot\right)\le \mathcal{M}_s^{\Jdt\left(t^n\right)}\vdtk\left(t^n,\cdot\right)\le \mathcal{M}_s^{\Jdt\left(t^n\right)}\udt\left(t^n,\cdot\right),
\]
and
\[
\mathcal{M}_s^{\Jdt\left(t^n\right)}\udt\left(t^n,\cdot\right)\le \mathcal{M}_s^{\Jdt\left(t^n\right)}\wdtk\left(t^n,\cdot\right)\le \mathcal{M}_s^{\Jdtkp\left(t^n\right)}\wdtk\left(t^n,\cdot\right).
\]
These two chains of inequalities give the desired result thanks to \eqref{hyp:R_decreasing}, and a straightforward induction.
\end{proof}

\subsubsection{$\left(u_0, I_0\right)$ is the viscosity solution of \eqref{eq:HJ}}\label{sec:ident}

In this section, we show that $\vl=\wl=\ul$, and that Prop. \ref{thm:limitconv} holds.
Following \cite{CrandallLions,CalvezHivertYoldas2022}, let us introduce for $\alpha\in(0,1)$, $\sigma>0$, and $(t,x,\tau,\xi)\in[0,T)\times \R\times[0,T]\times\R$,
\begin{align}\label{eq:defpsim}
\psim\left(t,x,\tau,\xi\right)=&\vk\left(t,x\right)-\vdtk\left(\tau,\xi\right)-\frac{\left(x-\xi\right)^2}{2\sqrt{\dx}}-\frac{\left(t-\tau\right)^2}{2\sqrt{\dt}}
\\&-\alpha\frac{e^t}{2}(x^2+\xi^2)-\frac{\alpha}{T-t}-\left(\sigma+\Ch^2\alpha e^T\right)t, \nonumber
\\ \label{eq:defpsip}
\psip\left(t,x,\tau,\xi\right)=&\wk\left(t,x\right)-\wdtk\left(\tau,\xi\right)+\frac{\left(x-\xi\right)^2}{2\sqrt{\dx}}+\frac{\left(t-\tau\right)^2}{2\sqrt{\dt}}
\\&+\alpha\frac{e^t}{2}(x^2+\xi^2)+\frac{\alpha}{T-t}+\left(\sigma+\Ch^2\alpha e^T\right)t, \nonumber
\end{align}
with $\Ch$ defined in \eqref{eq:Linf}. Thanks to Lemma \ref{lem:propofregularizedsolutions}, $\psim$ and $\psip$ respectively admit a maximum and minimum on $[0,T)\times\R\times[0,T]\times\R$, and the following lemma holds:

\begin{lem}
\label{lem:psi}
Let $\psim$ and $\psip$ be defined respectively by \eqref{eq:defpsim}, \eqref{eq:defpsip}. We have the following properties:
\begin{itemize}
\item For all $\alpha\in(0,1)$, and for all $\sigma>0$, $\psim$ admits a maximum, and $\psip$ a  minimum on $[0,T)\times\R\times[0,T]\times\R$. They are reached respectively at $\left(t^-,x^-,\tau^-, \xi^-\right)$ and $\left(t^+,x^+,\tau^+, \xi^+\right)$.
\item There exists $\sigma^+\left(\dt,k\right)$, $\sigma^-\left(\dt,k\right)$ positive, with $\sigma^\pm\left(\dt,k\right)\to_{\dt\to0}0$ for all $k>0$ (though not uniformly in $k$), such that, for all $\alpha\in (0,1)$, and $\dt$ small enough
\begin{equation}
\label{eq:estimate_tplus}
t^\pm\leq 2\Lt\sqrt{\dt}.
\end{equation}
Moreover,
\begin{align}
\psim\left(t^-,x^-,\tau^-, \xi^-\right)&\le 2\Lt^2\sqrt{\dt}+10\LxT^2\sqrt{\dx}
\nonumber
\\
\psip\left(t^+,x^+,\tau^+, \xi^+\right)&\ge -2\Lt^2\sqrt{\dt}-10\LxT^2\sqrt{\dx}, \label{eq:estimate_psip}
\end{align}
where $\Lx$ is defined in Prop. \ref{thm:existence}-\ref{it:Lx}, and $\Lt$ in Prop. \ref{prop:Noel}-\ref{it:udt_stability_liptime}.
\end{itemize}
\end{lem}

\begin{proof}
As it is the adaptation of the particular case \cite[Lemma 4.6]{CalvezHivertYoldas2022}, the proofs of these two results are strongly related.
It is however worth detailing it in what follows, since the approximations $\dHe\left(p,p\right)$, and $\dRe$ of $\Hcont\left(p\right)$ and $\Rcont$, and the general framework considered here brings technical difficulties.   We focus here on the results for $\sigma^+(\dt,k)$ and $\psip$. The others can be proved with straightforward adaptations.

The proof is divided in three steps, first show that
\begin{equation}\label{eq:identb1}
\alpha e^{t^+}\max\left(\left|x^+\right|,\left|\xi^+\right|\right)\leq 4\LxT,
\end{equation}
and then that
\begin{equation}\label{eq:identclose}
\left|t^+-\tau^+\right|\leq 2 \Lt\sqrt{\dt},\;\;\text{and}\;\; \left|x^+-\xi^+\right|\leq 10\LxT \sqrt{\dx}.
\end{equation}
Finally, using these two results, define $\sigma^+(\tau,k)$, such that $\tau^+=0$. It yields $t^+\leq 2 \Lt\sqrt{\dt}$, and deduce the desired bound on $\psip$.

For the first step, notice that $\psip\left(t^+,x^+,\tau^+,\xi^+\right)\leq\psip\left(t^+,0,\tau^+,0\right)$, and so that
\[
\wk\left(t^+,x^+\right)-\wdtk\left(\tau^+,\xi^+\right)+\alpha\frac{e^{t^+}}{2}({x^+}^2+{\xi^+}^2)\leq \wk\left(t^+,0\right)-\wdtk\left(\tau^+,0\right).
\]
 Elementary computations and the Lipschitz-regularity of $\wk$, and $\wdtk$ provided by Lemma \ref{lem:propofregularizedsolutions} then give \eqref{eq:identb1}
\details{
\[
\alpha\frac{e^{t^+}}{2}({x^+}^2+{\xi^+}^2)\le \wk\left(t^+,0\right)-\wdtk\left(\tau^+,0\right)-\wk\left(t^+,x^+\right)+\wdtk\left(\tau^+,\xi^+\right)\leq \LxT\left(\left|x^+\right|+\left|\xi^+\right|\right).
\]
}
Similarly, the bound on $\left(x^+-\xi^+\right)$ is a consequence of $\psip\left(t^+,x^+,\tau^+,\xi^+\right)\le\psip\left(t^+,x^+,\tau^+,x^+\right)$.Thanks to the non-negativity of ${\xi^+}^2$, and after elementary computations, it yields
\details{
\[
-\wdtk\left(\tau^+,\xi^+\right)+\frac{\left(x^+-\xi^+\right)^2}{2\sqrt{\dx}}+\alpha\frac{e^t}{2}{\xi^+}^2\le -\wdtk\left(\tau^+,x^+\right)+\alpha\frac{e^t}{2}{x^+}^2
\]
}
\[
\frac{\left(x^+-\xi^+\right)^2}{2\sqrt{\dx}}\le \wdtk\left(\tau^+,\xi^+\right)-\wdtk\left(\tau^+,x^+\right)+\alpha\frac{e^t}{2}\left({x^+}^2-{\xi^+}^2\right).
\]
\details{
\[
\frac{\left(x^+-\xi^+\right)^2}{2\sqrt{\dx}}\le \LxT \left|x^+-\xi^+\right|+\alpha\frac{e^t}{2}\left({x^+}^2-{\xi^+}^2\right)
\]
}
Thus, \eqref{eq:identb1}, proved previously and Lemma \ref{lem:propofregularizedsolutions}, gives the desired result, since
\[
\left(x^+-\xi^+\right)^2\leq 2\sqrt{\dx}\left[\LxT \left|x^+-\xi^+\right|+\alpha\frac{e^t}{2}\left({x^+}-{\xi^+}\right)\left({x^+}+{\xi^+}\right)\right].
\]
\details{
\[
\left|x^+-\xi^+\right|\leq 2\sqrt{\dx}\left[\LxT +\alpha\frac{e^t}{2}\left({x^+}+{\xi^+}\right)\right]\leq 10\LxT\sqrt{\dx}.
\]}
For the other estimate in \eqref{eq:identclose}, use $\psip\left(t^+,x^+,\tau^+,\xi^+\right)\leq\psip\left(t^+,x^+,t^+,\xi^+\right)$ and Lemma \ref{lem:propofregularizedsolutions} to get
\details{
\[
 	-\wdtk\left(\tau,\xi\right)+\frac{\left(t-\tau\right)^2}{2\sqrt{\dt}}\le-\wdtk\left(t,\xi\right)
\]
}
\[
\frac{\left(t^+-\tau^+\right)^2}{2\sqrt{\dt}}\le\wdtk\left(\tau^+,\xi^+\right)-\wdtk\left(t^+,\xi^+\right)\leq \Lt \left|t^+-\tau^+\right|.
\]

Let us now prove that $\tau^+=0$. We argue by contradiction, supposing that $\tau^+>0$, and defining $\sigma^+\left(\dt,k\right)$ such that this assumption cannot hold.
Let us introduce for all $(t,x)\in [0,T)\times\R$,
\begin{align*}
\phip\left(t,x\right)=&\wdtk\left(\tau^+,\xi^+\right)-\frac{\left(x-\xi^+\right)^2}{2\sqrt{\dx}}-\frac{\left(t-\tau^+\right)^2}{2\sqrt{\dt}}
\\&-\alpha\frac{e^t}{2}\left(x^2+{\xi^+}^2\right)-\frac{\alpha}{T-t}-\left(\sigma+\Ch^2\alpha e^T\right)t,
\end{align*}
so that
$
\psip\left(t,x,\tau^+,\xi^+\right)=\wk\left(t,x\right)-\phip\left(t,x\right)$. As $\phip$ is minimal at $\left(t^+,x^+\right)$, and since $\wk$ is the viscosity solution of \eqref{eq:wk}, one has
\begin{equation}\label{eq:viscsolphip}
\partial_t\phip\left(t^+,x^+\right)+b\left(x^+, \Jlkp\left(t^+\right)\right)\Hcont\left(\nabla_x\phip\left(t^+,x^+\right)\right)+\Rcont\left(t^+,x^+,\Jlkp\left(t^+\right)\right)\ge 0,
\end{equation}
\details{
\begin{align*}
 \partial_t\phip\left(t^+,x^+\right)=& -\frac{t^+-\tau^+}{\sqrt{\dt}} - \alpha \frac{\e^{t^+}}{2}\left( {x^+}^2+{\xi^+}^2 \right) - \frac{\alpha}{\left(T-t^+\right)^2}-  \sigma - \Ch^2 \alpha \e^{T} \\
 \nabla_x \phip\left(t^+,x^+\right)=& -\frac{x^+-\xi^+}{\sqrt{\dx}} -\alpha \e^{t^+} x^+.
\end{align*}
}
that can be reformulated as
\begin{align}
\nonumber
-&\frac{t^+-\tau^+}{\sqrt{\dt}} - \alpha \frac{\e^{t^+}}{2}\left( {x^+}^2+{\xi^+}^2 \right) - \frac{\alpha}{\left(T-t^+\right)^2}-  \sigma - \Ch^2 \alpha \e^{T}
\\
&+b\left(x^+, \Jlkp\left(t^+\right)\right)\Hcont\left( -\frac{x^+-\xi^+}{\sqrt{\dx}} -\alpha \e^{t^+} x^+ \right)+\Rcont\left(t^+,x^+,\Jlkp\left(t^+\right)\right)\ge 0.  \label{eq:star}
\end{align}
A similar inequality is obtained considering $\wdtk$ and scheme \eqref{eq:wdtk}. As, for all $\left(\tau,\xi\right)\in[0,T]\times\R$, the inequality $\psip\left(t^+,x^+,\tau,\xi\right)\ge \psip\left(t^+,x^+,\tau^+,\xi^+\right)$ holds, one has
\begin{equation}
\label{eq:njgiozerp1}
\owk\left(\tau,\xi\right)+k^+\geq \wdtk\left(\tau,\xi\right),
\end{equation}
where  for all $\left(\tau,\xi\right)\in[0,T]\times\R$,
\begin{equation}
\label{eq:defwbar}
\owk\left(\tau,\xi\right)=\frac{\left(x^+-\xi\right)^2}{2\sqrt{\dx}}+\frac{\left(t^+-\tau\right)^2}{2\sqrt{\dt}}+\alpha\frac{e^{t^+}}{2}\xi^2,
\end{equation}
and
\begin{align}
k^+&=\wdtk\left(\tau^+,\xi^+\right) - \frac{\left(x^+-\xi^+\right)^2}{2\sqrt{\dx}}- \frac{\left(t^+-\tau^+\right)^2}{2\sqrt{\dt}}-\alpha \frac{\e^{t^+}}{2}{\xi^+}^2
=\wdtk\left(\tau^+,\xi^+\right)-\owk\left(\tau^+,\xi^+\right). \label{eq:defkplus}
\end{align}
Recall now that  $\tau^+\in[0,T)$ is supposed such that $\tau^+>0$. Then, there exists $n^+\in\ccl 0,N_T-1\ccr$, and $s^+\in(0,\dt]$ such that
$\tau^+=t^{n^+}+s^+$.
The next step consists in taking $\tau=t^{n^+}$ and $\xi=\xi^+$, and propagate the inequality  \eqref{eq:njgiozerp1} to time $\tau^+$ with Lemma \ref{lem:contmonotonus}.
 To that extend, we show that the arguments of $\dHe$ in
 $\mathcal{M}_{s^+}^{\Jdtkp\left(\tau^+\right)}\left(\wdtk\left(t^{n^+}\right),\cdot\right)\left(\xi^+\right)$ and $\mathcal{M}_{s^+}^{\Jdtkp\left(\tau^+\right)}\left(\owk\left(t^{n^+}\right),\cdot\right)\left(\xi^+\right)$ are bounded by $\Linf$.
 Concerning $\wdtk$, it is in fact a consequence of Lemma \ref{lem:propofregularizedsolutions}. Consider now $\owk$, one has
 \details{
 \begin{align*}
  \frac{\owk\left(\tau,\xi^++\dx\right)-\owk\left(\tau,\xi^+\right)}{\dx}&=-\frac{x^+-\xi^+}{\sqrt{\dx}}+\frac{\sqrt{\dx}}{2}+\alpha \e^{t^+} \left(\xi^++\frac{\dx}{2}\right) \\
  \frac{\owk\left(\tau,\xi^+\right)-\owk\left(\tau,\xi^+-\dx\right)}{\dx}&=-\frac{x^+-\xi^+}{\sqrt{\dx}}-\frac{\sqrt{\dx}}{2}+\alpha \e^{t^+} \left(\xi^+-\frac{\dx}{2}\right)
 \end{align*}
 }
 \[
 \frac{\owk\left(\tau,\xi^+\pm\dx\right)-\owk\left(\tau,\xi^+\right)}{\dx}=\mp\frac{x^+-\xi^+}{\sqrt{\dx}}+\frac{\sqrt{\dx}}{2}\pm\alpha \e^{t^+} \xi^++ \alpha \e^{t^+}\frac{\dx}{2},
 \]
so that \eqref{eq:identb1}-\eqref{eq:identclose}, and \eqref{eq:smooth_R} yield
\begin{equation}
\label{eq:wbar_majo_taux} \left| \owk\left(\tau,\xi^+\pm\dx\right)-\owk\left(\tau,\xi^+\right)\right|\le \Linf \dx.
\end{equation}
We may now compute the values of the image of each side of \eqref{eq:njgiozerp1} by $\Mm_{s^+}^{\Jdtkp}$. By definition of $\wdtk$,  one has
\begin{equation}\label{eq:njgiozerp2}
\Mm_{s^+}^{\Jdtkp\left(\tau^+\right)}\left(\wdtk\left(t^{n^+},\cdot\right)\right)\left(\xi^+\right)=\wdtk\left(\tau^+,\xi^+\right)+s^+\dRe\left(\tau^+,\xi^+,\Jdtkp\left(\tau^+\right)\right),
\end{equation}
and the expression of $\MsJ$ (see \eqref{eq:MsJ}) yields
\begin{align}
\label{eq:njgiozerp3}
&\Mm_{s^+}^{\Jdtkp\left(\tau^+\right)}\left(\owk\left(t^{n^+},\cdot\right)+k^+\right)\left(\xi^+\right)=\owk\left(t^{n^+},\xi^+\right)+k^+
\\-&s^+b\left(\xi^+, \Jdtkp\left(\tau^+\right) \right)\dHe\left(\frac{\owk\left(t^{n^+},\xi^+\right)-\owk\left(t^{n^+},\xi^+-\dx\right)}{\dx},\frac{\owk\left(t^{n^+},\xi^++\dx\right)-\owk\left(t^{n^+},\xi^+\right)}{\dx}\right). \nonumber
\end{align}
The propagation of inequality \eqref{eq:njgiozerp1} to time $\tau^+$ with Lemma \ref{lem:contmonotonus} then gives
\begin{align*}
 0\le&\owk\left(t^{n^+},\xi^+\right)+k^+ - \wdtk\left(\tau^+,\xi^+\right) -s^+\dRe\left(\tau^+,\xi^+,\Jdtkp\left(\tau^+\right)\right)
\\-&s^+b\left(\xi^+, \Jdtkp\left(\tau^+\right) \right)\dHe\left(\frac{\owk\left(t^{n^+},\xi^+\right)-\owk\left(t^{n^+},\xi^+-\dx\right)}{\dx},\frac{\owk\left(t^{n^+},\xi^++\dx\right)-\owk\left(t^{n^+},\xi^+\right)}{\dx}\right),
\end{align*}
that is, using \eqref{eq:defkplus} and \eqref{eq:defwbar},
\details{
\begin{align*}
 0\le  \owk\left(t^{n^+}, \xi^+\right) -\owk\left(\tau^+,\xi^+\right) &-s^+\dRe\left(\tau^+,\xi^+,\Jdtkp\left(\tau^+\right)\right)
 \\-s^+b\left(\xi^+, \Jdtkp(\tau^+) \right)\dHe&\left(
 -\frac{x^+-\xi^+}{\sqrt{\dx}} -\frac{\sqrt{\dx}}{2} + \alpha\e^{t^+} \xi^+ - \alpha\e^{t^+} \frac{\dx}{2}\right., \\&
\left. -\frac{x^+-\xi^+}{\sqrt{\dx}} +\frac{\sqrt{\dx}}{2} + \alpha\e^{t^+} \xi^+ +\alpha\e^{t^+} \frac{\dx}{2}
 \right),
\end{align*}
and this estimate can be simplified one more time, using \eqref{eq:defwbar} and dividing by $s^+>0$,
}
\begin{align}
 \label{eq:etoile}
 0 \le \frac{t^+-t^{n^+} - s^+/2}{\sqrt{\dt}}-&\dRe\left(\tau^+,\xi^+,\Jdtkp(\tau^+)\right)
 \\-b\left(\xi^+, \Jdtkp\left(\tau^+\right) \right)\dHe&\left(
 -\frac{x^+-\xi^+}{\sqrt{\dx}} -\frac{\sqrt{\dx}}{2} + \alpha\e^{t^+} \xi^+ - \alpha\e^{t^+} \frac{\dx}{2}\right.,
 \nonumber
 \\&
\left. -\frac{x^+-\xi^+}{\sqrt{\dx}} +\frac{\sqrt{\dx}}{2} + \alpha\e^{t^+} \xi^+ +\alpha\e^{t^+} \frac{\dx}{2}
 \right).
 \nonumber
\end{align}
The choice of $\sigma^+(\dt,k)$, and the conclusion of the argument are a consequence of \eqref{eq:star} and the previous estimate. Indeed, adding the two expressions yields
\begin{equation}
\label{eq:itispositive}
 \mathfrak{T}_1+\mathfrak{T}_2+\mathfrak{T}_3 \ge 0,
\end{equation}
 with
 \details{
\begin{align*}
\mathfrak{T}_1&= \frac{t^+ -t^{n^+} - s^+/2}{\sqrt{\dt}}-\frac{t^+ -t^{n^+} - s^+}{\sqrt{\dt}} - \alpha\frac{\e^{t^+}}{2} \left({x^+}^2+{\xi^+}^2 \right) - \frac{\alpha}{\left(T-t^+\right)^2}- \sigma - \Ch^2\alpha\e^T \\
\mathfrak{T}_2&=\Rcont\left(t^+,x^+,\Jlkp\left(t^+\right)\right)-\dRe\left(\tau^+,\xi^+,\Jdtkp(\tau^+)\right) \\
\mathfrak{T}_3&=b\left(x^+, \Jlkp\left(t^+\right)\right)\Hcont\left( -\frac{x^+-\xi^+}{\sqrt{\dx}} -\alpha \e^{t^+} x^+ \right)
-b\left(\xi^+, \Jdtkp(\tau^+) \right)\\&\dHe\left(
 -\frac{x^+-\xi^+}{\sqrt{\dx}} -\frac{\sqrt{\dx}}{2} + \alpha\e^{t^+} \xi^+ - \alpha\e^{t^+} \frac{\dx}{2},
 -\frac{x^+-\xi^+}{\sqrt{\dx}} +\frac{\sqrt{\dx}}{2} + \alpha\e^{t^+} \xi^+ +\alpha\e^{t^+} \frac{\dx}{2}
 \right).
\end{align*}
}
\begin{align*}
\mathfrak{T}_1&= \frac{s^+}{2\sqrt{\dt}} - \alpha\frac{\e^{t^+}}{2} \left({x^+}^2+{\xi^+}^2 \right) - \frac{\alpha}{\left(T-t^+\right)^2}- \sigma - \Ch^2\alpha\e^T \\
\mathfrak{T}_2&=\Rcont\left(t^+,x^+,\Jlkp\left(t^+\right)\right)-\dRe\left(\tau^+,\xi^+,\Jdtkp(\tau^+)\right) \\
\mathfrak{T}_3&=b\left(x^+, \Jlkp\left(t^+\right)\right)\Hcont\left( -\frac{x^+-\xi^+}{\sqrt{\dx}} -\alpha \e^{t^+} x^+ \right)
-b\left(\xi^+, \Jdtkp(\tau^+) \right)\\&\dHe\left(
 -\frac{x^+-\xi^+}{\sqrt{\dx}} -\frac{\sqrt{\dx}}{2} + \alpha\e^{t^+} \xi^+ - \alpha\e^{t^+} \frac{\dx}{2},
 -\frac{x^+-\xi^+}{\sqrt{\dx}} +\frac{\sqrt{\dx}}{2} + \alpha\e^{t^+} \xi^+ +\alpha\e^{t^+} \frac{\dx}{2}
 \right),
\end{align*}
and these three terms are then estimated independently. To begin with, as $s^+\in(0,\dt]$, one has
\begin{equation}
\label{eq:T1}
\mathfrak{T}_1\le \sqrt{\dt} - \alpha\frac{\e^{t^+}}{2} \left({x^+}^2+{\xi^+}^2 \right) - \sigma -\Ch^2\alpha\e^T.
\end{equation}
The second one is a consequence of the $K$-Lipschitz regularity of $\Rcont$ and $\dRe$ in \eqref{hyp:R_decreasing}-\eqref{hyp:R_bounded}-\eqref{hyp:R_t-Lipschitz}, of the $k$-Lipschitz regularity of $\Jlkp$ and $\Jdtkp$ (Lemma \ref{lem:imporvedconvergence}), and of the fact that $\dRe$ is supposed to be a good approximation of $\Rcont$,
\details{
\begin{align*}
 \mathfrak{T}_2&= \Rcont\left(t^+,x^+,\Jlkp\left(t^+\right)\right)- \Rcont\left(\tau^+,x^+,\Jlkp\left(t^+\right)\right) \\
 & + \Rcont\left(\tau^+,x^+,\Jlkp\left(t^+\right)\right) - \Rcont\left(\tau^+,\xi^+,\Jlkp\left(t^+\right)\right) \\
 & + \Rcont\left(\tau^+,\xi^+,\Jlkp\left(t^+\right)\right) - \Rcont\left(\tau^+,\xi^+,\Jdtkp\left(t^+\right)\right) \\
 & + \Rcont\left(\tau^+,\xi^+,\Jdtkp\left(t^+\right)\right) - \Rcont\left(\tau^+,\xi^+,\Jdtkp\left(\tau^+\right)\right) \\
 &+ \Rcont\left(\tau^+,\xi^+,\Jdtkp\left(\tau^+\right)\right)
- \dRe\left(\tau^+,\xi^+,\Jdtkp(\tau^+)\right)
\end{align*}
}
\[
\mathfrak{T}_2 \le K \left|t^+-\tau^+\right| + K\left|x^+-\xi^+\right| + K \left\|\Jlkp-\Jdtkp\right\|_\infty +Kk\left|\tau^+-t^+\right| + \left\| \Rcont-\dRe\right\|_\infty,
\]
and thanks to \eqref{eq:identclose}, and to \eqref{hyp:Reta_approximation}-\eqref{eq:etachoice}, we obtain
\details{
\begin{align*}
 \mathfrak{T}_2 &\le 2K\Lt\sqrt{\dt} +10K\Lx(T) \sqrt{\dx}  +  K \left\|\Jlkp-\Jdtkp\right\|_\infty + 2Kk\Lt \sqrt{\dt} + K\eta,
\end{align*}
 that is with \eqref{eq:etachoice},
\begin{equation*}
 \mathfrak{T}_2 \le K\left[ 2\Lt(1+k)  \right]+ L\left[ 1+10\Lx(T)\right] \sqrt{\dx} + K\left\|\Jlkp-\Jdtkp\right\|_\infty
\end{equation*}
}
\begin{equation}
  \label{eq:T2}
  \mathfrak{T}_2 \le C_2^k\left( \sqrt{\dt}+\sqrt{\dx}+  \left\|\Jlkp-\Jdtkp\right\|_\infty\right),
\end{equation}
\details{with $C_2^k=K\max\{ 2(1+k)\Lt; 10 \Lx(T)+1\}$. }where $C_2^k$ depends only on $K$, $\Lt$, $\Lx(T)$ and on $k$.
The estimate of $\mathfrak{T}_3$ is obtained similarly, using the Lipschitz property of $b$ in  \eqref{hyp:b_lipschitz},the fact that $\dHe\left(p,p\right)$ is an approximation of $\Hcont\left(p\right)$ in \eqref{hyp:Heta_approximation}, and the $\Ch$-Lipschitz property of $\dHe$. Note that the latter is true because \eqref{eq:wbar_majo_taux}
holds. More precisely, $\mathfrak{T}_3$ is reformulated as
\[
 \mathfrak{T}_3=\mathfrak{T}_3^a+\mathfrak{T}_3^b+\mathfrak{T}_3^c+\mathfrak{T}_3^d,
\]
with
\begin{align*}
 &\mathfrak{T}_3^a= b\left(\xi^+, \Jdtkp(\tau^+) \right)\dHe\left(
 -\frac{x^+-\xi^+}{\sqrt{\dx}} - \alpha\e^{t^+} x^+ ,
 -\frac{x^+-\xi^+}{\sqrt{\dx}} - \alpha\e^{t^+} x^+
 \right) \\
 &-b\left(\xi^+, \Jdtkp(\tau^+) \right)\dHe\left(
 -\frac{x^+-\xi^+}{\sqrt{\dx}} -\frac{\sqrt{\dx}}{2} + \alpha\e^{t^+} \xi^+ - \alpha\e^{t^+} \frac{\dx}{2},
 -\frac{x^+-\xi^+}{\sqrt{\dx}} +\frac{\sqrt{\dx}}{2} + \alpha\e^{t^+} \xi^+ +\alpha\e^{t^+} \frac{\dx}{2}
 \right) \\
 &\mathfrak{T}_3^b=-b\left(\xi^+, \Jdtkp(\tau^+) \right)\dHe\left(
 -\frac{x^+-\xi^+}{\sqrt{\dx}} - \alpha\e^{t^+} x^+ ,
 -\frac{x^+-\xi^+}{\sqrt{\dx}} - \alpha\e^{t^+} x^+
 \right)
 \\ &+ b\left(\xi^+, \Jdtkp(\tau^+) \right)\Hcont\left(
 -\frac{x^+-\xi^+}{\sqrt{\dx}} - \alpha\e^{t^+} x^+
 \right)
 \\
 &\mathfrak{T}_3^c = -b\left(\xi^+, \Jdtkp(\tau^+) \right)\Hcont\left(
 -\frac{x^+-\xi^+}{\sqrt{\dx}} - \alpha\e^{t^+} x^+
 \right)
 + b\left(x^+, \Jlkp(\tau^+) \right)\Hcont\left(
 -\frac{x^+-\xi^+}{\sqrt{\dx}} - \alpha\e^{t^+} x^+
 \right)
 \\
& \mathfrak{T}_3^d =  -b\left(x^+, \Jlkp(\tau^+) \right)\Hcont\left(
 -\frac{x^+-\xi^+}{\sqrt{\dx}} - \alpha\e^{t^+} x^+
 \right)
 + b\left(x^+, \Jlkp\left(t^+\right)\right)\Hcont\left( -\frac{x^+-\xi^+}{\sqrt{\dx}} -\alpha \e^{t^+} x^+ \right).
\end{align*}
Then, $\mathfrak{T}_3^a$ is estimated with \eqref{eq:Heta_defCH}-\eqref{eq:Linf}, $\mathfrak{T}_3^b$ with \eqref{hyp:Heta_approximation}, $\mathfrak{T}_3^c$ with the Lipschitz property of $b$ in \eqref{hyp:b_lipschitz}, and
\[
 \sup_{|p|\le\Linf}\left|\Hcont\left(p\right)\right|\le \Linf \Ch+\mathcal{C}_\Linf ,
\]
that holds true since $\Hcont\left(0\right)=0$, because $\dHe$ enjoys $\Ch$-Lipschitz regularity, and because it is an approximation of $\Hcont$, with \eqref{hyp:Heta_approximation}-\eqref{eq:etachoice}. Eventually, the $k$-Lipschitz property of $\Jlkp$ (Lemma \ref{lem:imporvedconvergence}) is used in $\mathfrak{T}_3^d$, so that
\details{
\begin{align*}
 \mathfrak{T}_3^a &\le \Ch\left( \frac{\sqrt{\dx}}{2}+ \alpha\e^{t^+} \left(|x^+|+|\xi^+| + \frac{\dx}{2}\right) \right) \\
 \mathfrak{T}_3^b &\le \bM C_\Linf \sqrt{\dx} \\
 \mathfrak{T}_3^c&\le\Lb \left(\left|x^+-\xi^+\right| + \left\|\Jdtkp-\Jlkp\right\|_\infty \right) \sup_{B\left(0,\Linf\right)}|\Hcont|
 \\ &\le\Lb\left(\Linf\Ch +{C}_\Linf\right)\left(\left|x^+-\xi^+\right| + \left\|\Jdtkp-\Jlkp\right\|_\infty \right)
 \\&\le 10\left(\Linf \Ch+{C}_\Linf\right) \Lb \Lx(T) \sqrt{\dx} +\left(\Linf\Ch+{C}_\Linf\right)\Lb \left\|\Jdtkp-\Jlkp\right\|_\infty \\
 \mathfrak{T}_3^d&\le \Lb k|t^+-\tau^+| \left(\Linf\Ch+{C}_\Linf\right)
 \\& \le 2\Lb \Lt k\left(\Linf\Ch+{C}_\Linf\right) \sqrt{\dt}
\end{align*}
}
\begin{equation}
\label{eq:T3}
 \mathfrak{T}_3\le C^k_3 \left( \sqrt{\dx}+\sqrt{\dt}+ \left\|\Jdtkp-\Jlkp\right\|_\infty\right)+ \Ch \alpha\e^{t^+}\left(\left|x^+\right|+\left|\xi^+\right|\right),
\end{equation}
\details{
with
\[
 C^k_3=\max\left\{
 \begin{array}{l}
 \ds\Ch + \bM C_\Linf +10 \left(\Linf \Ch+{C}_\Linf\right) \Lb \Lx(T);\vspace{4pt}\\
 \ds2\Lb\Lt k \left(\Linf \Ch+{C}_\Linf\right) ;\vspace{4pt}
 \\ \ds \left(\Linf \Ch+{C}_\Linf\right)\Lb
 \end{array}
 \right\},
\]
}
where $C_3^k$ depends only on the constants arising in the assumptions, namely $\Ch$, $\bM$, $C_\Linf$, $\Lx(T)$, $\Lt$, and on $k$. We now come back to \eqref{eq:itispositive}, and gather \eqref{eq:T1}, \eqref{eq:T2} and \eqref{eq:T3}, so that
\details{
\begin{align*}
 \sqrt{\dt} - \alpha\frac{\e^{t^+}}{2} \left({x^+}^2+{\xi^+}^2 \right) - \sigma -\Ch^2\alpha\e^T +\Ch \alpha\e^{t^+}\left(\left|x^+\right|+\left|\xi^+\right|\right)& \\ + \left(C^k_2+ C^k_3\right) \left( \sqrt{\dx}+\sqrt{\dt}+ \left\|\Jdtkp-\Jlkp\right\|_\infty\right) & \ge 0,
\end{align*}
or, equivalently
}
\begin{align*}
 \sqrt{\dt} & + \left(C^k_2+ C^k_3\right) \left( \sqrt{\dx}+\sqrt{\dt}+ \left\|\Jdtkp-\Jlkp\right\|_\infty\right)
 \ge  \sigma +\alpha\Ch^2\left(\e^T-\e^{t^+}\right) \\&+\alpha\frac{\e^{t^+}}{2} \left(\left|x^+\right|^2-2\Ch  |x^+| +\Ch^2  \right)
 + \alpha \frac{\e^{t^+}}{2} \left( \left|\xi^+\right|^2-2\Ch \left|\xi^+\right|+\Ch^2 \right),
\end{align*}
which finally gives the conclusion, as
\begin{equation}
\label{eq:defsigma}
 \sqrt{\dt}  + \left(C^k_2+ C^k_3\right) \left( \sqrt{\dx}+\sqrt{\dt}+ \left\|\Jdtkp-\Jlkp\right\|_\infty\right)
 \ge  \sigma.
\end{equation}
Indeed, as $\dt$ and $\dx$ depend on each other with \eqref{eq:CFLsatisfier}, one can choose $\sigma=\sigma^+(\dt,k)$, with $\sigma^+(\dt,k)\longrightarrow_{\dt\to 0} 0$ when $k$ is fixed, such that the inequality \eqref{eq:defsigma} is a contradiction. \details{For instance, twice the left-hand side of \eqref{eq:defsigma} is a good choice for $\sigma$.} With this choice $\sigma^+(\dt,k)$, the assumption $\tau^+>0$ cannot hold, which yields the conclusion $\tau^+=0$. The estimate \eqref{eq:estimate_tplus} is then a straightforward consequence of \eqref{eq:identclose}, and the inequality \eqref{eq:estimate_psip} comes from the expression of $\psip$ in \eqref{eq:defpsip}. Indeed,
\details{the following estimates holds true,
\[
 \psip\left(t^+,x^+,\tau^+,\xi^+\right) \ge \wk\left(t^+,x^+\right) - \wdtk\left(\tau^+,\xi^+\right),
\]
that is,
}
since $\tau^+=0$, and $\wk\left(0,\cdot\right)=\wdtk\left(0,\cdot\right)$, the following estimate holds true
\[
 \psip\left(t^+,x^+,\tau^+,\xi^+\right) \ge \wk\left(t^+,x^+\right) -\wk\left(0,x^+\right)+\wdtk\left(0,x^+\right)- \wdtk\left(0,\xi^+\right).
\]
The Lipschitz regularity of $\wk$ and $\wdtk$ stated in Lemma \ref{lem:propofregularizedsolutions}, and \eqref{eq:identclose} finally gives \eqref{eq:estimate_psip}.
\end{proof}

We may now show the convergence of the solution of scheme \eqref{eq:scheme} towards the viscosity solution of \eqref{eq:HJ}. It is a consequence of
\begin{equation}
\label{eq:conclusion_inequality}
\vl\le\ul\le\wl,
\end{equation}
 where $\vl$,  and $\wl$ are defined in \eqref{eq:defvl}-\eqref{eq:defwl} and $\ul$ in Lemma \ref{lem:extraction}.
 As in the proof of Lemma \ref{lem:psi}, let us detail only the upper bound here, as the bound from below is very similar. For all $(t,x)\in[0,T]\times\R$, one has, using \eqref{eq:estimate_psip}, and the fact that $\left(t^+,x^+, \tau^+,\xi^+\right)$ is a minimum point of $\psip$,
\[
-2\Lt^2\sqrt{\dt}-10\LxT^2\sqrt{\dx}\le \psip\left(t^+,x^+,\tau^+,\xi^+\right)\leq \psip\left(t,x,t,x\right),
\]
that is,
\[
-2\Lt^2\sqrt{\dt}-10\LxT^2\sqrt{\dx}\le \wk\left(t,x\right)-\wdtk\left(t,x\right)+\alpha e^t x^2+\frac{\alpha}{T-t}+\left(\sigma^+(\dt,k)+\Ch^2\alpha e^T\right)t.
\]
Let now  $\alpha\to0$ in previous inequality, it gives
\details{
\[
\wdtk\left(t,x\right)-2\Lt^2\sqrt{\dt}-10\LxT^2\sqrt{\dx}-\sigma^+\left(\dt,k\right) t\le \wk\left(t,x\right),
\]
and use the estimate $\udt\leq\wdtk$ stated in Lemma \ref{lem:propofregularizedsolutions}, so that
}
\[
\udt\left(t,x\right)-2\Lt^2\sqrt{\dt}-10\LxT^2\sqrt{\dx}-\sigma^+\left(\dt,k\right) t\le \wk\left(t,x\right),
\]
where we also used the estimate $\udt\leq\wdtk$ stated in Lemma \ref{lem:propofregularizedsolutions}.
Let now $\dt\to0$, with $k>0$ fixed. By definition of $\dx$, $\sigma^+\left(\dt,k\right)$ and $\udt$ the limit of the right-hand side is $\ul\left(t,x\right)$. Therefore, for all $k>0$, $t\in[0,T]$, and $x\in\R$,
\[
\ul\left(t,x\right)\le\wk\left(t,x\right).
\]
One can finally let $k\to\infty$ in the previous inequality, so that for all $t\in[0,T]$ and for all $x\in\R$, $\ul\left(t,x\right)\le\wl\left(t,x\right)$, which is the second estimate in \eqref{eq:conclusion_inequality}.

To conclude, recall that $\vl=\wl$, thanks to Remark \ref{rmq:ToConclude}, and  that $\ul$, $\vl$ and $\wl$ enjoy Lipschitz regularity. Then, inequalities in \eqref{eq:conclusion_inequality} are in fact equalities, and $\ul$ is the viscosity solution of \eqref{eq:defvl} (or \eqref{eq:defwl}, equivalently). In addition, for all $t\in[0,T]$, $\min \ul\left(t,\cdot\right)=0$, as proved in Lemma \ref{lem:extraction}, so that Theorem \ref{thm:CalvezLam2020} yields that $\left(\ul,\Jl\right)$ is the viscosity solution of \eqref{eq:HJ}. Eventually, as Theorem \ref{thm:CalvezLam2020} also states the uniqueness of the viscosity solution of \eqref{eq:HJ}, the restriction \emph{up to a subsequence} of Remark \ref{rmq:Subsequence}, can then be removed, so that Prop. \ref{thm:limitconv} holds.

\subsection{Relaxation of the CFL condition}
\label{sec:cflgest}

The convergence of scheme \eqref{eq:scheme} is proved under strong assumptions on the relative size of the discretization parameters $\dt$ and $\dx$. Indeed, the CFL condition \eqref{eq:CFLsatisfier} must be satisfied, but the definition of $\Linf$ in \eqref{eq:Linf} includes a security margin required by the proof of the convergence. It makes the CFL very restrictive, and refining time meshes, so that \eqref{eq:CFLsatisfier} holds, leads to very costly numerical computations. This is especially accurate when the Hamiltonian $\Hcont$ contains an exponential, as in the example \eqref{eq:discussion_H}.

In practice, considering the CFL with the larger slope that appears in the computations, instead of $\Linf$, is enough to ensure the stability of the computations. One can even prove that the scheme converges, using this relaxed CFL condition.
The key idea  is to linearize $\Hcont$ and $\dHe$ for slopes stronger than what appears in the simulations. The process is the following:
\begin{enumerate}
\item \label{it:CFL_i} Make an educated guess of the steepest slope $S$ that could appear. In the worst-case scenario, $\LxT$ suits, but smaller values are targeted.
\item Let, for all $p\in\R$:
\[
\tilde{\Hcont}(p)=\begin{cases}
\Hcont(S)+(p-S)\Hcont'(S)&\text{if } p>S\\
\Hcont(p)&\text{if } -S\leq p\leq S\\
\Hcont(-S)+(p-S)\Hcont'(-S)&\text{if } p<-S.
\end{cases}
\]
Note that if $S$ is greater than the stronger slope of the solution $u$ of \eqref{eq:HJ}, the values outside of $[-S,S]$ are never reached by $p=\nabla_x u$. Hence, replacing $\Hcont$ by $\tilde{\Hcont}$ in \eqref{eq:HJ} does not change the equation.
Then, linearize $\dHe$ for slopes stronger than $S$. Let for all $(p,q)\in\R^2$:
\[
\tilde{\dHe}(p,q)=\begin{cases}
\dHe(p,q)&\text{if } -S\leq p\leq S,\;-S\leq q\leq S\\
\dHe(p,S) + (q-S) \Hcont'(S)  & \text{if } -S\le p\le S, \;  q> S \\
 \dHe(p,-S) + (q+S) \Hcont'(-S) &  \text{if } -S\le p\le S, \;  q<- S \\
\dHe(S,q)+(p-S)\Hcont'(S)&\text{if } p>S,\;-S\leq q\leq S\\
\dHe(-S,q)+(p+S)\Hcont'(-S)&\text{if } p<-S,\;-S\leq q\leq S\\
\dHe(S,S)+(p-S)\Hcont'(S)+(q-S)\Hcont'(S)&\text{if } p>S,\;q> S\\
\dHe(S,-S)+(p-S)\Hcont'(S)+(q+S)\Hcont'(-S)&\text{if } p>S,\;q<-S\\
\dHe(-S,S)+(p+S)\Hcont'(-S)+(q-S)\Hcont'(S)&\text{if } p<-S,\;q> S\\
\dHe(-S,-S)+(p+S)\Hcont'(-S)+(q+S)\Hcont'(-S)&\text{if } p<-S,\;q<-S.
\end{cases}
\]
Note that, if $S$ is large enough, replacing $\dHe$ by $\tilde{\dHe}$ does not change the scheme. However, in any case, $\tilde{\dHe}$ satisfies the assumptions required for $\dHe$. As the slope of $\tilde{\dHe}$ does not increase for $|p|\ge S$ or $|q|\ge S$, no security margin is to be added in the proof of Section \ref{sec:ident}.
\item Choose $\dx$ according the CFL associated to $S$, for instance
\begin{equation}\label{eq:CFL3}
\dt \CHL(S)=\frac{\dx}{2},
\end{equation}
and $\eta=\sqrt{\dx}$ as in \eqref{eq:etachoice}.
\item Compute the solutions of the new system, and check that for all $i$ and all $n$,  $\left|u_i^n-u_{i-1}^n\right|\le S\dx$. If this does not hold, try again from step \ref{it:CFL_i} with a more educated guess. For instance, if the first time with larger slopes is $0<t<T$, one could try $
S'=\Lin+(S-\Lin)T/t$.
\end{enumerate}

\section{Numerical tests and discussion}
\label{sec:numericalsims}

In this section, we present some numerical tests to highlight the properties of schemes \eqref{eq:scheme} and \eqref{eq:scheme_ep}. Unless other choices are specified, the following birth rate, and net growth rate parameters will be used
\begin{align}
 \label{eq:tests_b}
 b(x,I)& = \frac{8-(x+1)^2}{1+(x+1)^2} \frac{1}{2+I},
 \\
 \label{eq:tests_R}
 \Rcont(t,x,I)&= - \frac{(x+3)^2}{1+(x+3)^2} - \frac{x^2}{1+x^2}(t+1) I.
\end{align}
Two initial data $\uin$ are considered, depending on the tests. The first one will be denoted $\uinconv$. It is convex, and defined for all $x\in\R$ by,
\begin{equation}
 \label{eq:uin_conv}
 \uinconv = \sqrt{1+x^2}-1.
\end{equation}
The second one has two local minima, is hence not convex, and will be denoted $\uinnotconv$. It is given for all $x\in\R$ by,
\begin{equation}
 \label{eq:uin_notconv}
 \uinnotconv(x)= \left( \left(1+\left(x+3\right)^2\right)^{1/4} -1\right)\;\left( \left(2+\left(x-2\right)^2\right)^{1/4} -1\right).
\end{equation}

In all tests, the $x$ domain is $[-10,10]$, but the final time $T$ varies and will always be specified. Regarding the discretization, $\dt$ is defined specifically for all tests, and $\dx$ is determined as a function of $\dt$ using the relation \eqref{eq:CFL3}. The value of $S$ will always be indicated, and an order of magnitude of $\dx$ will be given for information. Among the list of discretizations of $\Hcont$ proposed in Section \ref{sec:schemes}, we focus on \eqref{eq:HP1} and \eqref{eq:HCSS}. We indeed favored, \eqref{eq:HP1} as it  naturally appears as the limit of the AP scheme \eqref{eq:scheme_ep}. It is in the convex setting but not in the flat setting, and \eqref{eq:HCSS} is its modification to a scheme in the flat setting. The choice will be made clear for each test.

\subsection{Properties of \eqref{eq:scheme}}
\label{sec:TestsNums_MonotonyI}

We investigate the properties of scheme \eqref{eq:scheme}, starting with its convergence. It is established in Proposition \ref{thm:limitconv}, but the compactness argument that is used in the proof does not yield any convergence rate. We test it numerically  in what follows. Contrary to the particular case of the quadratic Hamiltonian $\Hcont(p)=p^2$ that has been tested in \cite{CalvezHivertYoldas2022}, no analytical solution of \eqref{eq:HJ} is available when $\Hcont$ is defined by \eqref{eq:discussion_H}. We hence define a reference solution for scheme \eqref{eq:scheme}, using the same scheme with a refined grid. Let $T=0.5$, $S=2$, and fix $b$, $\Rcont$ and a convex initial data as in \eqref{eq:tests_b}, \eqref{eq:tests_R} and \eqref{eq:uin_conv}. The reference solution $(\bu_{\mathrm{red}}, \bI_{\mathrm{ref}})$ is computed with $\dt=4\cdot10^{-5}$ (which yields $\dx \simeq 2.4\cdot10^{-3}$), and it is compared to solutions $(\bu_\dt,\bI_\dt)$ computed for a sequence of $\dt$ varying from $10^{-2}$ to $10^{-4}$. The corresponding $\dx$ hence vary approximately from $0.6$ to $6\cdot 10^{-3}$. Both \eqref{eq:HP1} and \eqref{eq:HCSS} are tested, and the convergence errors
\begin{equation}
\label{eq:Lschemeconv_err}
 \left\| \bu^{N_T^{\mathrm{ref}}}_{\mathrm{ref}}- \bu^{N_T^\dt}_\dt\right\|_\infty, \;\;\text{and}\;\;  \left\| \bI_{\mathrm{ref}}- \bI_\dt\right\|_{L^1(0,T)},
\end{equation}
are displayed as a function of $\dt$ in logarithmic scale in Figure \ref{fig:Lschemeconv}. These errors are defined using suitable norms regarding the expected regularity of the solution of $\eqref{eq:HJ}$, see \cite{CalvezHivertYoldas2022} for details.  This suggests that the numerical order of scheme \ref{eq:scheme} is $1$, both for \eqref{eq:HP1} and \eqref{eq:HCSS}.

\begin{figure}[!ht]
 \centering
 \begin{tabular}{@{}c@{}c@{}}
 \includegraphics[width=0.45\textwidth]{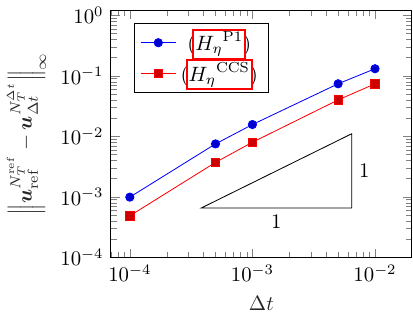} &
 \includegraphics[width=0.42\textwidth]{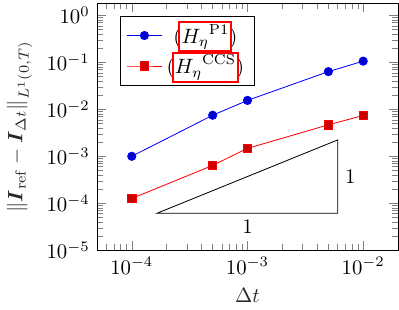}
 \end{tabular}
%
%
%
%
\caption{
Scheme \eqref{eq:scheme} with $T=0.5$, $\dt_{\mathrm{ref}} = 4\cdot 10^{-5}$, $S=2$, $b$, $\Rcont$ and $\uinconv$ defined in \eqref{eq:tests_b}, \eqref{eq:tests_b} and \eqref{eq:uin_conv}. Convergence test for the solution of \eqref{eq:scheme} with \eqref{eq:HP1} and \eqref{eq:HCSS}.
 Left: Error \eqref{eq:Lschemeconv_err} on the component $\bu$,  as a function of $\dt$ (log scale). Right: Error \eqref{eq:Lschemeconv_err} on the component $\bI$,  as a function of $\dt$ (log scale).
}
\label{fig:Lschemeconv}
\end{figure}

Focus now on the qualitative properties of scheme \eqref{eq:scheme}, and especially on the properties of $\bI$, or equivalently of $I_\dt$ defined by \eqref{eq:defJdt}. When the dependency in $t$ of $\Rcont$ is Lipschitz,  \eqref{eq:I_BVbound} yields that
 the constraint $I$ in \eqref{eq:HJ} can decrease, but only continuously. It is indeed the limit of $I_\dt$ when $\dt\to 0$, thanks to Proposition \ref{thm:limitconv}, and its finite differences are bounded from below.
Roughly speaking, this means that $I$ has only increasing jumps. This property is illustrated on the left-hand side of Fig. \ref{fig:decayandjump}, where $I_\dt$ is displayed for as a function of $t\in[0,1/2]$. One can notice that $I_\dt$ starts with a smooth decay, before a increasing jump. The right-hand side of Fig. \ref{fig:decayandjump} presents $\bu$, or equivalently $u_\dt$ defined by \eqref{eq:defudt}, as a function of $t$ and $x$. Remark that the location of the minimum of $u_\dt$ jumps when $I_\dt$ jumps. These figures have been obtained with discretization \eqref{eq:HP1}, parameters $b$ and $\Rcont$ defined by \eqref{eq:tests_b} and \eqref{eq:tests_R}, the initial data $\uinnotconv$ as in \eqref{eq:uin_notconv}, $\dt=2\cdot10^{-5}$ (hence $\dx\simeq 1.2\cdot 10^{-3}$).

\begin{figure}[!ht]
%

\centering
 \begin{tabular}{@{}c@{}c@{}}
 \includegraphics[width=0.4\textwidth]{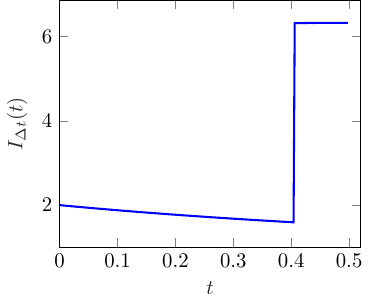} &
 \includegraphics[width=0.5\textwidth]{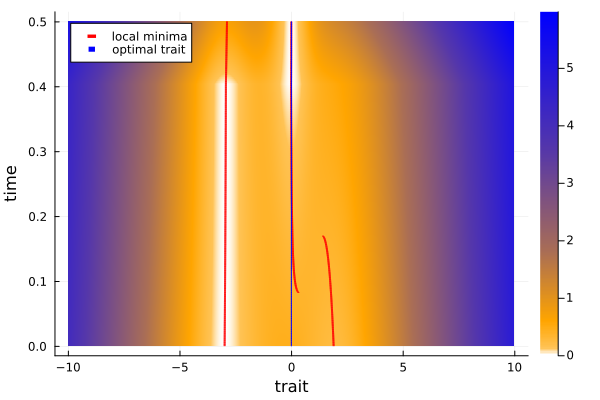}
 \end{tabular}

\caption{
Scheme \eqref{eq:scheme} with $T=0.5$, $\dt = 2\cdot 10^{-5}$, $S=2$, $b$, $\Rcont$ and $\uinnotconv$ defined in \eqref{eq:tests_b}, \eqref{eq:tests_R} and \eqref{eq:uin_notconv}. Qualitative properties of \eqref{eq:scheme} with \eqref{eq:HP1}.
Left: $I_\dt$ as a function of $t$.
Right: $u_\dt$ as a function of $x$ (trait) and $t$ (time). The solid lines represent respectively the local minima of $u_\dt$, and the optimal trait. Light colors emphasize small values of $u_\dt$.
}
\label{fig:decayandjump}
\end{figure}

When $\Rcont$ does not depend on $t$, it has been proved in \cite{CalvezLam2020}, that $I$ is non-decreasing. It is worth noticing that this property is numerically preserved only in the flat setting.
Fig.  \ref{fig:decaywhenincrease} highlights this behavior. We considered once again $b$ defined in \eqref{eq:tests_b} and the initial data $\uinnotconv$  \eqref{eq:uin_notconv}. We used $\Rcont(0,x,I)$ instead of $\Rcont$ defined in \eqref{eq:tests_R}, to make it independent of $t$. We represented $I_\dt$ as a function of $t$, when it is computed with a scheme in the flat setting as \eqref{eq:HCSS}, or with a scheme in the convex setting as \eqref{eq:HP1}. The left-hand side of Fig. \ref{fig:decaywhenincrease} is computed with a very coarse grid, $\dt=10^{-2}$ (that is, $\dx\simeq 0.6$). The numeric-induced decay of $I_\dt$ in the convex setting \eqref{eq:HP1} is clearly visible, while $I_\dt$ is non-decreasing in the flat setting \eqref{eq:HCSS}. Of course, this phenomenon is only due to numerical approximations, and the decay tends to $0$ when $\dt\to 0$, as the limit $I$ of $I_\dt$ is non-decreasing. The right-hand side of Fig. \ref{fig:decaywhenincrease} shows that the artificial decay of \eqref{eq:HP1} is much smaller when $\dt=10^{-4}$ (and $\dx\simeq 6\cdot10^{-3}$).

\begin{figure}[!ht]
%
%
%
%
%

\centering
 \begin{tabular}{@{}c@{}c@{}}
 \includegraphics[width=0.45\textwidth]{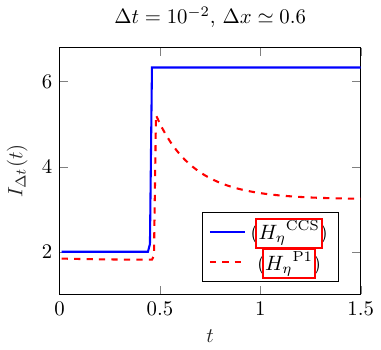} &
 \includegraphics[width=0.45\textwidth]{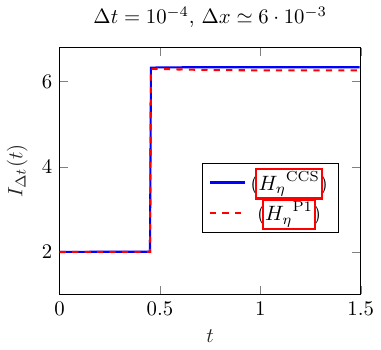}
 \end{tabular}

\caption{
Scheme \eqref{eq:scheme} with $T=1.5$, $S=2$, $\Rcont(0,x,I)$ defined in \eqref{eq:tests_R},  and $b$ and $\uinnotconv$ defined in \eqref{eq:tests_b} and \eqref{eq:uin_notconv}.  $I_\dt$ as a function of $t$.
Left: with $\dt=10^{-2}$.
Right: with $\dt=10^{-4}$.
}
\label{fig:decaywhenincrease}
\end{figure}

\subsection{Asymptotic-preserving property of scheme \eqref{eq:scheme_ep}}

Consider now the scheme \eqref{eq:scheme_ep}. As it is stated in Proposition \ref{prop:Scheme_ep_AP}, it enjoys the AP property. Indeed,  its solution converges to the solution of scheme \eqref{eq:scheme} when $\ep$ goes to $0$ with fixed discretization. The discretization being fixed, denote $(\bv^\ep,\bJ^\ep)$ the solution of scheme \eqref{eq:scheme_ep} for a given $\ep>0$,  $(\bu,\bI)$ the solution of scheme \eqref{eq:scheme}. Thanks to the AP property, $(\bv^\ep, \bJ^\ep)$ go to $(\bu,\bI)$ when $\ep\to 0$. This property is tested in Fig. \ref{fig:APness}, where the errors
\begin{equation}
 \label{eq:tests_AP_Err}
 \left\|  \bv^\ep-\bu\right\|_\infty, \;\;\left\|  \bJ^\ep-\bI\right\|_{L^1(0,T)},
\end{equation}
are displayed as functions of $\ep$, in logarithmic scale. Here again, the functional space have been chosen according to the expected regularity of the solutions, see \cite{CalvezHivertYoldas2022}. As expected, these errors decrease with $\ep$, a saturation for the small errors in the component $\bJ_\ep$ excepted. Moreover, according to this test the numerical order of convergence of the solution of \eqref{eq:P_ep}  to the solution of \eqref{eq:HJ} with \eqref{eq:discussion_H} is $1$. The schemes \eqref{eq:scheme} and \eqref{eq:scheme_ep} are run with  $b$, $\Rcont$ and $\uinnotconv$ defined in \eqref{eq:tests_b}, \eqref{eq:tests_R} and \eqref{eq:uin_notconv}, with $T=0.5$, $S=2$ and $\dt=10^{-3}$ (hence, $\dx\simeq 6\cdot10^{-2}$). Scheme \eqref{eq:scheme} is implemented with discretization \eqref{eq:HP1}.

\begin{figure}[!ht]
\centering

%
%
%
%

\centering
 \begin{tabular}{@{}c@{}c@{}}
 \includegraphics[width=0.45\textwidth]{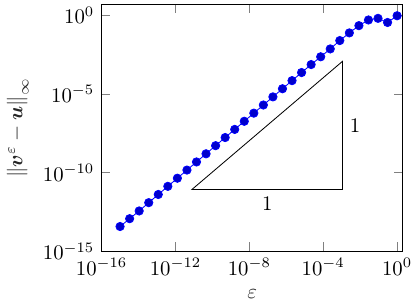} &
 \includegraphics[width=0.45\textwidth]{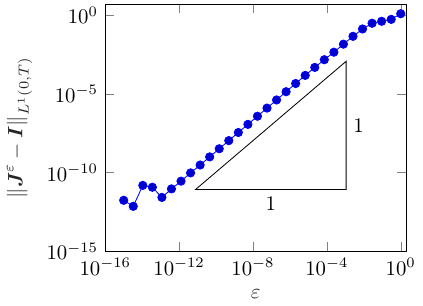}
 \end{tabular}

\caption{
With $T=0.5$, $\dt=10^{-3}$, $S=2$, $b$, $\Rcont$ and $\uinnotconv$ defined in \eqref{eq:tests_b}, \eqref{eq:tests_R} and \eqref{eq:uin_notconv}. AP property of scheme \eqref{eq:scheme_ep}. Scheme \eqref{eq:scheme} with \eqref{eq:HP1}. Left: Error \eqref{eq:tests_AP_Err} on the component $\bv$ as a function of $\ep$ (log scale). Right: Error \eqref{eq:tests_AP_Err} on the component $\bJ$ as a function of $\ep$ (log scale).
}
\label{fig:APness}
\end{figure}

A stronger property of AP schemes is the fact that they can be Uniformly Accurate (UA). Indeed, in the numerical test above, the discretization was fixed, because AP scheme may present an order degeneracy for some values of $\ep$, typically when their order of magnitude is comparable to the discretization parameters. On the contrary, the UA property certifies that the precision is independent of the scaling parameter $\ep$. The proof of this property is far beyond the scope of this paper, as the estimated order of convergence of the limit scheme \eqref{eq:scheme} is not even theoretically determined. It can however be tested numerically. As for convergence numerical tests, a reference solution must be computed to perform the test. It is done with $b$, $\Rcont$, and $\uinnotconv$ defined in \eqref{eq:tests_b}, \eqref{eq:tests_R} and \eqref{eq:uin_notconv}, and $T=0.5$, $S=2$. For all $\ep>0$,  the reference solution is  $(\bv^\ep_{\dt_{\mathrm{ref}}}, \bJ^\ep_{\dt_{\mathrm{ref}}})$, computed with \eqref{eq:scheme_ep}, and a refined grid $\dt_{\mathrm{ref}}=2\cdot10^{-5}$ (that is $\dx_{\mathrm{ref}}\simeq 1.2\cdot10^{-3}$). Then, for $\ep$ varying from $10^{-16}$ to $1$, and for $\dt$ varying from $10^{-4}$ to $10^{-2}$ ($\dx$ from $6\cdot10^{-3}$ to $0.6$), compute the $(\dt,\ep)$-dependent errors,
\begin{equation}
 \label{eq:tests_UA_Err}
 \left\| \bv^\ep_\dt - \bv^\ep_{\dt_{\mathrm{ref}}}\right\|_\infty,
 \left\| \bJ^\ep_\dt - \bJ^\ep_{\dt_{\mathrm{ref}}}\right\|_{L^1(0,T)}.
 \end{equation}
They are displayed in Fig. \ref{fig:UA}, each value of $\dt$ being represented in logarithmic scale as a function of $\ep$. The component in $\bv$ of the error is on the left-hand side of the figure, and the component in $\bJ$ is on the right-hand side. Remark that the supremum in $\ep$ of these errors decreases with $\dt$. It is in fact of the order of magnitude of $\dx$. This suggests that the scheme \eqref{eq:scheme_ep} enjoys the UA property, with uniform order of convergence $1$.


\begin{figure}[!ht]

 \centering
  \begin{tabular}{@{}c@{}c@{}}
  \includegraphics[width=0.45\textwidth]{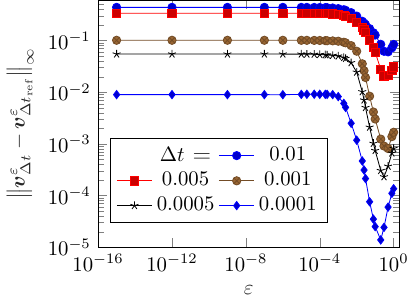} &
  \includegraphics[width=0.45\textwidth]{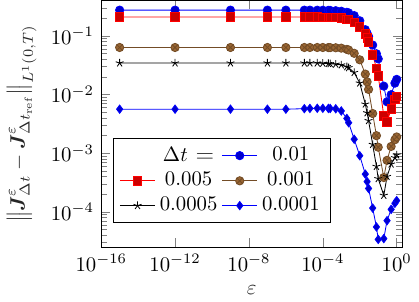}
  \end{tabular}

\caption{
With $T=0.5$, $\dt=10^{-3}$, $S=2$, $b$, $\Rcont$ and $\uinnotconv$ defined in \eqref{eq:tests_b}, \eqref{eq:tests_R} and \eqref{eq:uin_notconv}. UA property of scheme \eqref{eq:scheme_ep}.  Left: For a range of $\dt$, error \eqref{eq:tests_UA_Err} on the component $\bv$ as a function of $\ep$ (log scale). Right: For a range of $\dt$, error \eqref{eq:tests_UA_Err} on the component $\bJ$ as a function of $\ep$ (log scale).
}
\label{fig:UA}
\end{figure}

\subsection{Scheme \eqref{eq:scheme_ep} when the population vanishes}
\label{sec:TestsNums_AP_CasLimite}

Determining theoretically the asymptotic behaviour of \eqref{eq:P_ep},
for given $b$ and $\Rcont$ not necessarily satisfying nice hypothesis,
is -up to our knowledge- an open question.
It can however be conjectured using \eqref{eq:scheme_ep}, as it is illustrated in this section. In this section, the following birth rate will be considered,
\begin{equation}
 \label{eq:tests_oscillatory_b}
 b(x,I)=2,
\end{equation}
with the initial data $\uinconv$ defined in \eqref{eq:uin_conv}. All the test cases are run with $T=6$, $\dt=2\cdot10^{-5}$ (that is, $\dx\simeq 2.2\cdot10^{-2}$), and $S=3$. We focus on changing environments, as it is done in \cite{FigueroaIglesiasMirrahimi2018, FigueroaIglesiasMirrahimi2021, CostaEtchegarayMirrahimi2021} for the parabolic case.

Consider first a periodic varying environment, with
\begin{equation}
 \label{eq:tests_oscillatory_R_smooth}
 R(t,x,I)= - (x- 3\sin(\alpha t))^2- I.
\end{equation}
Note that $\Rcont$ does not satisfy hypothesis \eqref{hyp:R_ImIM}. However, with $\alpha=0.5$, that is a slowly-varying environment, scheme \eqref{eq:scheme_ep} suggests that the population does not extincts, and that the asymptotic behavior of \eqref{eq:P_ep} is \eqref{eq:HJ}, see Fig. \ref{fig:oscillatory}. The left-hand side of Fig. \ref{fig:oscillatory} displays $I_\dt$ (or, equivalently, $\bI$) computed with \eqref{eq:scheme}-\eqref{eq:HP1}, and $\bJ$ computed with \eqref{eq:scheme_ep} and $\ep=10^{-8}$. Note that both coincide, and that they are bounded from above and below by positive constants, meaning that the total size of the population does not vanish nor grow uncontrolled. The function $u_\dt$ (or the sequence $\bu$) defined by \eqref{eq:scheme}-\eqref{eq:HP1} is on the right-hand side of Fig. \ref{fig:oscillatory} ($\bv$ defined by \eqref{eq:scheme_ep} is not represented, as it cannot be distinguished from the one for $u_\dt$). The optimal trait and the minimum of $u_\dt$ are represented by solid lines on the figure. Note that the local minimum of $u_\dt$ follows the optimal trait with a delay, as expected thanks to \cite{FigueroaIglesiasMirrahimi2021}. Fig. \ref{fig:oscillatory} has been obtained with $T=6$, $\dt=2\cdot10^{-5}$ and $S=3$.


\begin{figure}[!ht]
%
%
%

\centering
  \begin{tabular}{@{}c@{}c@{}}
  \includegraphics[width=0.4\textwidth]{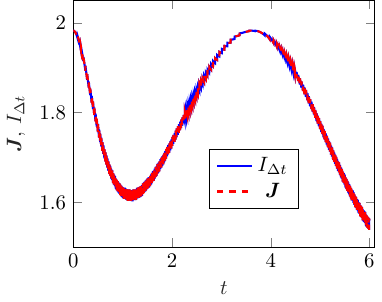} &
  \includegraphics[width=0.5\textwidth]{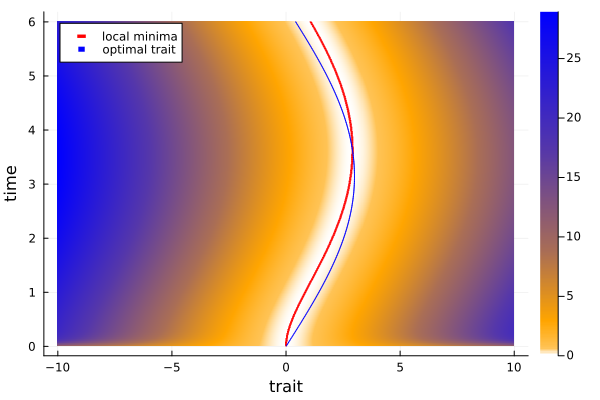}
  \end{tabular}

\caption{
With $T=6$, $\dt = 2\cdot 10^{-5}$, $S=3$, $b$, $\uinconv$ defined in \eqref{eq:tests_oscillatory_b}, \eqref{eq:uin_conv}, $\Rcont$ defined in \eqref{eq:tests_oscillatory_R_smooth} and $\alpha=0.5$.
Left: $I_\dt$ computed with \eqref{eq:scheme} and \eqref{eq:HP1}, and $\bJ$ computed with \eqref{eq:scheme_ep}, as functions of $t$.
Right: $u_\dt$ computed with \eqref{eq:scheme}-\eqref{eq:HP1} as a function of $x$ (trait) and $t$ (time). The solid lines represent respectively the local minima of $u_\dt$, and the optimal trait. Light colors emphasize small values of $u_\dt$.
}
\label{fig:oscillatory}
\end{figure}

The parameter $\alpha$ in \eqref{eq:tests_oscillatory_R_smooth} drives the speed at which the environment changes. When it varies too quickly, the population does not have enough time to adapt and may vanish, see \cite{FigueroaIglesiasMirrahimi2018} for periodic-varying environments, or \cite{FigueroaIglesiasMirrahimi2021} for shifted environments. This behavior is illustrated in Fig. \ref{fig:fast_oscillatory_eps}, where we took $\alp=3$ in \eqref{eq:tests_oscillatory_R_smooth}, and where \eqref{eq:P_ep} is approximated with \eqref{eq:scheme_ep} for $\ep=10^{-8}$.
With this value of $\alpha$, scheme \eqref{eq:scheme_ep} suggests that the population vanishes in the asymptotics $\ep\to 0$ of \eqref{eq:P_ep}. Indeed, $\bJ$ defined by \eqref{eq:scheme_ep} goes to $0$ rather quickly. It is displayed on the left-hand side of Fig.\ref{fig:fast_oscillatory_eps} as a function of $t$. The value of $\bv$ is on the right-hand side of Fig. \ref{fig:fast_oscillatory_eps}, as a function of $t$ and $x$. As previously, note that the minimum of $\bv$ follows the optimal trait with a delay. The fact that the minimum of $\bv$ is not close to $0$ after some time, informs on the extinction of the population.

\begin{figure}[!ht]
%
%

\centering
  \begin{tabular}{@{}c@{}c@{}}
  \includegraphics[width=0.4\textwidth]{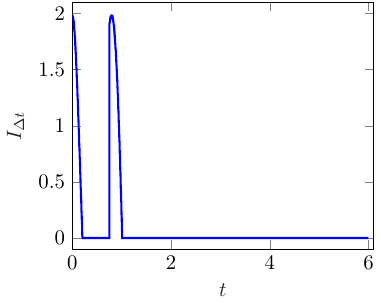} &
  \includegraphics[width=0.5\textwidth]{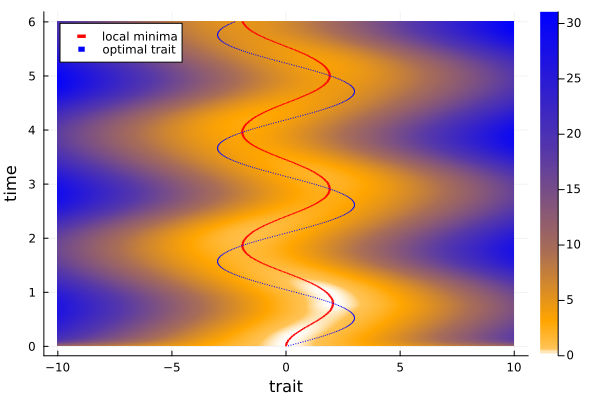}
  \end{tabular}

\caption{
With $T=6$, $\dt = 2\cdot 10^{-5}$, $S=3$, $b$, $\uinconv$ defined in \eqref{eq:tests_oscillatory_b}, \eqref{eq:uin_conv}, $\Rcont$ defined in \eqref{eq:tests_oscillatory_R_smooth} and $\alpha=3$.
Left: $\bJ$ computed with \eqref{eq:scheme_ep} as a function of $t$.
Right: $\bv$ computed with \eqref{eq:scheme_ep} as a function of $x$ (trait) and $t$ (time). The solid lines represent respectively the local minima of $\bv$, and the optimal trait. Light colors emphasize small values of $\bv$.
}
\label{fig:fast_oscillatory_eps}
\end{figure}

Let us emphasize on the fact that the solutions of \eqref{eq:scheme_ep} and \eqref{eq:scheme} do not coincide when the population vanishes. Indeed, still considering $\Rcont$ as in \eqref{eq:tests_oscillatory_R_smooth} with $\alpha=3$, the solution of \eqref{eq:scheme}-\eqref{eq:HP1} is displayed on Fig. \ref{fig:fast_oscillatory}, with the plot of $u_\dt$ as a function of $t$ and $x$ on the right-hand side.  Observe on the left-hand side that $I_\dt$ takes negative values. This is due to the fact that the constraint $\min u_\dt=0$ is always satisfied in scheme \eqref{eq:scheme}, possibly at the cost of the positivity of $I_\dt$. This suggests that, in this case of fast oscillations of the environment, the $\ep\to 0$ limit of \eqref{eq:P_ep} is not given by \eqref{eq:HJ}, and that the asymptotic behavior of \eqref{eq:P_ep} is better conjectured with \eqref{eq:scheme_ep} than with \eqref{eq:scheme}.

\begin{figure}[!ht]
%
%

\centering
  \begin{tabular}{@{}c@{}c@{}}
  \includegraphics[width=0.4\textwidth]{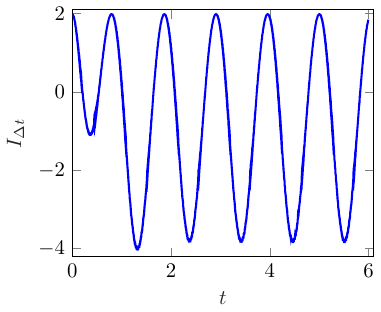} &
  \includegraphics[width=0.5\textwidth]{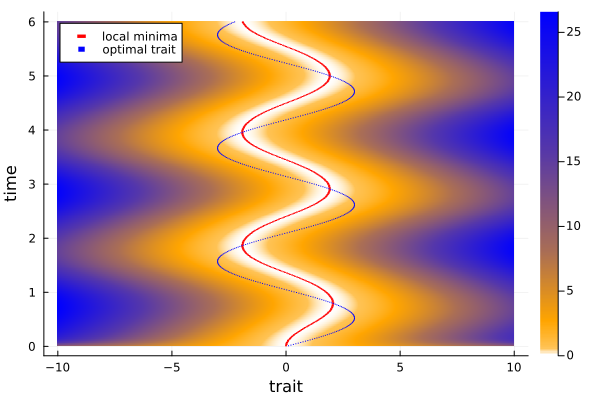}
  \end{tabular}

\caption{
With $T=6$, $\dt = 2\cdot 10^{-5}$, $S=3$, $b$, $\uinconv$ defined in \eqref{eq:tests_oscillatory_b}, \eqref{eq:uin_conv}, $\Rcont$ defined in \eqref{eq:tests_oscillatory_R_smooth} and $\alpha=3$.
Left: $I_\dt$ computed with \eqref{eq:scheme} and \eqref{eq:HP1} as a function of $t$.
Right: $u_\dt$ computed with \eqref{eq:scheme}-\eqref{eq:HP1} as a function of $x$ (trait) and $t$ (time). The solid lines represent respectively the local minima of $u_\dt$, and the optimal trait. Light colors emphasize small values of $u_\dt$.
}
\label{fig:fast_oscillatory}
\end{figure}

Eventually, we consider a constant by parts environment, as in \cite{CostaEtchegarayMirrahimi2021}, in a regime of quick changes in the environment. In what follows, we will define $\Rcont$ as
\begin{equation}
 \label{eq:tests_oscillatory_R_notsmooth}
 R(t,x,I)=-\left(x- 6\left\lceil \sin(3t-0.9) \right\rceil\mathds{1}_{t>0.3} \right)^2-I.
\end{equation}
Following the previous test case, we conjecture the small $\ep$ limit of \eqref{eq:P_ep} to vanish, because of the too fast changes of the environment. Fig. \ref{fig:fast_sq_oscillatory_eps} suggests that scheme \eqref{eq:scheme_ep} is robust enough to deal with the lack of regularity of $\Rcont$. Indeed, $\bJ$ is represented as a function of $t$ on the left-hand side of Fig. \ref{fig:fast_sq_oscillatory_eps}. Once again, it vanishes quickly, as expected. The evolution of $\bv$ is displayed on the right-hand side of Fig. \ref{fig:fast_sq_oscillatory_eps} as a function of $t$ and $x$. As previously, the minimum of $\bv$ is nonzero after some time, and its location follows the optimal trait with a delay, despite its lack of regularity. Fig. \ref{fig:fast_sq_oscillatory_eps} has been obtained with  $\ep=10^{-8}$ in \eqref{eq:scheme_ep}.

\begin{figure}[!ht]
%
%

\centering
  \begin{tabular}{@{}c@{}c@{}}
  \includegraphics[width=0.4\textwidth]{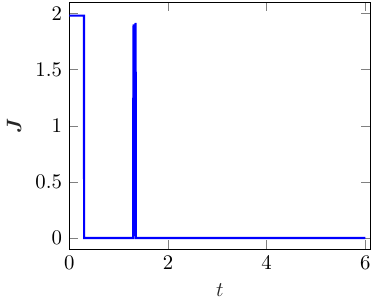} &
  \includegraphics[width=0.5\textwidth]{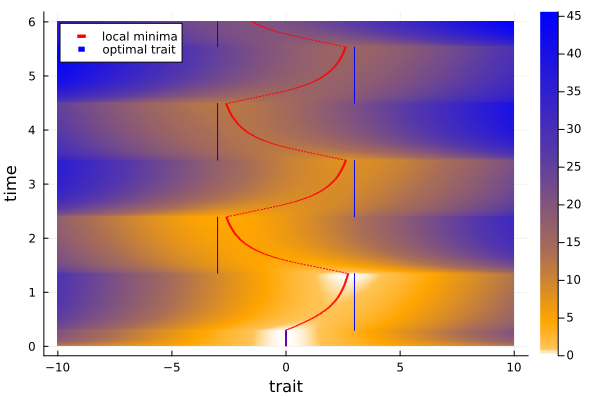}
  \end{tabular}

\caption{
With $T=6$, $\dt = 2\cdot 10^{-5}$, $S=3$, $b$, $\uinconv$ and $\Rcont$ defined in \eqref{eq:tests_oscillatory_b}, \eqref{eq:uin_conv} and in \eqref{eq:tests_oscillatory_R_notsmooth}.
Left: $\bJ$ computed with \eqref{eq:scheme_ep} as a function of $t$.
Right: $\bv$ computed with \eqref{eq:scheme_ep} as a function of $x$ (trait) and $t$ (time). The solid lines represent respectively the local minima of $\bv$, and the optimal trait. Light colors emphasize small values of $\bv$.
}
\label{fig:fast_sq_oscillatory_eps}
\end{figure}

Note that in this test case, scheme \eqref{eq:scheme} again gives results far from the ones of \eqref{eq:scheme_ep}. They are displayed in Fig. \ref{fig:fast_sq_oscillatory}, with $I_\dt$ as a function of $t$ on the left-hand side and $u_\dt$ as a function of $t$ and $x$ on the right-hand side. Once again, the constraint $\min u_\dt=0$ leads to negative values for $I_\dt$. Moreover, observe that $I_\dt$ has decreasing jumps. This could not have been possible if $\Rcont$ had Lipschitz regularity in $t$, see Prop. \ref{thm:existence}-\ref{it:I_BVbound}. Up to our knowledge, the well-posedness of \eqref{eq:HJ} is this context is an open question.   Despite the robustness of \eqref{eq:scheme}-\eqref{eq:HP1} regarding the lack of regularity of $\Rcont$, it corroborates the fact that \eqref{eq:HJ} is not the $\ep\to 0$ limit of \eqref{eq:P_ep} in this case.

\begin{figure}[!ht]
%
%

\centering
  \begin{tabular}{@{}c@{}c@{}}
  \includegraphics[width=0.4\textwidth]{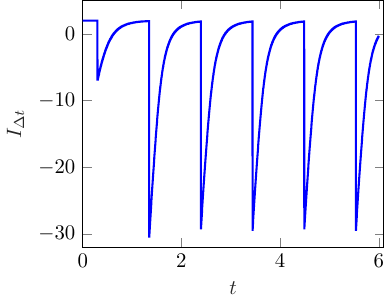} &
  \includegraphics[width=0.5\textwidth]{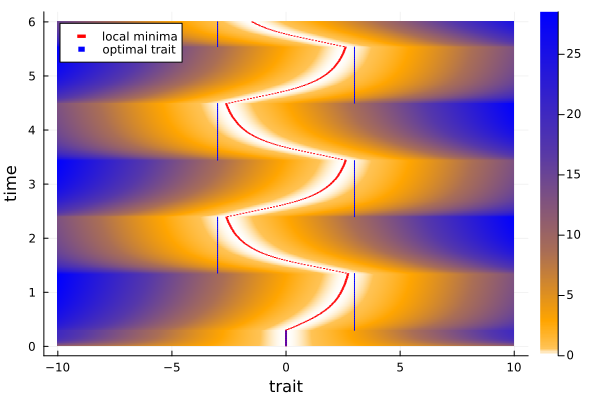}
  \end{tabular}

\caption{
With $T=6$, $\dt = 2\cdot 10^{-5}$, $S=3$, $b$, $\uinconv$ and $\Rcont$ defined in \eqref{eq:tests_oscillatory_b}, \eqref{eq:uin_conv} and in \eqref{eq:tests_oscillatory_R_notsmooth}.
Left: $I_\dt$ computed with \eqref{eq:scheme}-\eqref{eq:HP1} as a function of $t$.
Right: $u_\dt$ computed with \eqref{eq:scheme}-\eqref{eq:HP1} as a function of $x$ (trait) and $t$ (time). The solid lines represent respectively the local minima of $u_\dt$, and the optimal trait. Light colors emphasize small values of $u_\dt$.
}
\label{fig:fast_sq_oscillatory}
\end{figure}

\subsection{Generalization}

In this section, we investigate the generalization of the results of the paper, especially when the dimension is greater than $1$.
Let us start with
Lotka-Volterra integral equation in the multi-dimensional case, with Hamiltonian \eqref{eq:discussion_H} as in Section \ref{sec:Example}. Suppose that the mutation kernel $\K_d$ satisfies
\[
 \forall x\in\R^d, \; \K_d(x)=\prod\limits_{j=1}^d \K(x_j),
\]
with $\K$ defined in \eqref{eq:discussion_H}. Then, the Hamiltonian $\Hcont$ can be reformulated as
\[
 \forall \bp\in\R^d, \; \Hcont(\bp)=\prod\limits_{j=1}^d \int_\R \K(x_j) \e^{-z_j \bp} \dd z_j -1,
\]
so that the different choices of $\dHe$ proposed in Section \ref{sec:schemes} in dimension $1$ can be plugged in the product, to define a numerical Hamiltonian in dimension greater than $1$. Remark that one has to pay attention to the constant $-1$, that is not in the multiplication, but this is quickly addressed.
The AP scheme of Section \ref{sec:APscheme} is generalized similarly, and Prop. \ref{prop:Scheme_ep_AP} holds. However, the convergence of the scheme for the limit constrained Hamilton-Jacobi equation cannot be established as easily. This is the purpose of this section.

We studied
above
a class of numerical schemes for constrained Hamilton-Jacobi equations such as \eqref{eq:HJ}, in dimension $1$.
We
discuss here the design of a scheme for \eqref{eq:HJ} when the dimension is greater than $1$, and when the Hamiltonian $H$ also depends on the trait variable $x$. This latter case happens for instance in \cite{NordmannPerthameTaing2018}, where the Hamiltonian is implicitly defined.
Let us consider $u$, the viscosity solution of
\begin{equation}
 \label{eq:HJ_extension}
 \left\{
 \begin{array}{l l}
\ds \partial_t u(t,x) + b(x,I(t)) \Hcont(x,\nabla_x u(t,x)) + \Rcont(t,x,I(t))=0, & \ds x\in\R^d, \; t\ge 0, \vspace{4pt} \\
\ds \min\limits_{x\in\R^d} u(t,x) = 0, & \ds t\ge 0,
 \end{array}
 \right.
\end{equation}
where the initial condition $\uin$ satisfies \eqref{hyp:u0_Lipschitz}-\eqref{hyp:u0_coercive}-\eqref{hyp:u0_minimum}. As previously, we suppose that $b$ satisfies \eqref{hyp:b_bounds}-\eqref{hyp:b_lipschitz}-\eqref{hyp:b_decreasing}, and that $\Rcont$ satisfies \eqref{hyp:R_decreasing}-\eqref{hyp:R_ImIM}-\eqref{hyp:R_bounded}-\eqref{hyp:R_t-Lipschitz}. As $\Hcont$ now depends also on $x$, the hypothesis must be adapted. We suppose that for all $x\in\R^d$, $\bp\mapsto \Hcont(x,\bp)$ is convex, and \eqref{hyp:H0} is replaced by
\begin{equation}
 \label{hyp:H0_extension}
 \forall x\in\R^d, \; \Hcont(x,0) = 0, \;\; \text{and}\;\;\forall (x,\bp)\in\R^d\times\R^d, \; \Hcont(x,\bp)\ge 0,
\end{equation}
and we emphasize once again on the fact that the condition $\Hcont(x,0)=0$ for all $x\in\R^d$ can be relaxed by considering $(x,\bp)\mapsto \Hcont(x,\bp)-\Hcont(x,0)$, and modifying $\Rcont$ accordingly.

For a sake of simplicity, we define a uniform trait grid, with the same trait step in all the directions, but this could be easily relaxed. Let $x_0\in\R^d$ and $\dx >0$. For all $\bi=(i_1,\dots,i_d)\in\Z^d$, define then $x_\bi = x_0+\dx \bi$. The time grid is kept unchanged.

As in dimension $1$, introduce $\dHe:\R^d\times\R^d\times\R^d\to \R$ and $\dRe:\R\times\R^d\times\R\to \R$, such that $\dHe(x,\bp,\bp)$ is an approximation of the Hamiltonian $\Hcont(x,\bp)$ for all $x\in\R^d$, and $\bp\in\R^d$, and that $\dRe$ approximates $\Rcont(t,x,I)$ for all $t\ge 0$, all $x\in\R^d$ and all $I\in\R$.
Assumptions \eqref{hyp:Heta_monotonous}-\eqref{hyp:Heta_zero}-\eqref{hyp:Heta_nonnegative}-\eqref{hyp:Heta_nonnegative}-\eqref{hyp:Heta_approximation} and \eqref{hyp:Reta_approximation} are adapted to this multi-dimensional context. Introduce $\LH>0$, and suppose that $\dHe$ and $\dRe$ satisfy the following properties:
\begin{itemize}
\item \textbf{Regularity in $x$:} $\forall (\bp,\bq)\in\R^d\times\R^d$, $x\mapsto \dHe(x,\bp,\bq)$ is smooth. Moreover, $\dHe(\cdot,\bp,\bq)$ and its first and second derivatives are bounded uniformly with respect to $(\bp,\bq)\in\R^d\times\R^d$ such that $ \|\bp\|_\infty\le \LH$, and $\|\bq\|_\infty\le \LH$.
 \item \textbf{Monotony:} $\forall \|\bp\|_\infty\le\LH$, $\forall \|\bq\|_\infty\le \LH$, $\forall x\in\R^d$, $\forall \be\in\R^d$ such that $\be\ge 0$ (coordinate by coordinate),
 \[
  \dHe(x,\bp,\bq+\be)\le\dHe(x,\bp,\bq)\le \dHe(x,\bp+\be,\bq),
 \]
 \item \textbf{Exact value at $0$:} $\forall x\in\R^d$, $\dHe(x,0,0)=0$.
\item \textbf{Positivity on the diagonal: } $\forall \bp\in\R^d$ such that $\|\bp\|_\infty\le \LH$, $\forall x\in\R^d$, $\dHe(x,\bp,\bp)\ge 0$.
\item \textbf{Consistency for $\dHe$:} $\exists \CLH>0$, $\forall x\in\R^d$, $\forall \|\bp\|_\infty\le \LH$, $\forall \eta\in(0,1]$,
\[
|\dHe(x,\bp,\bp)- \Hcont(x,\bp)|\le \CLH\eta.
\]
\item \textbf{Consistency for $\dRe$:} $\forall t\ge 0$, $\forall x\in\R^d$, $\forall I\ge 0$, $\forall \eta\in(0,1]$,
\[
 |\dRe(t,x,I)-\Rcont(t,x,I)|\le K\eta,
\]
for some constant $K$ as in \eqref{hyp:Reta_approximation}.
\end{itemize}
To adapt the Definition \ref{def:WellChosenScheme} of a $\LH$-well chosen scheme, two classes of numerical Hamiltonians are considered:
\begin{itemize}
  \item \textbf{Flat setting:} If  $\|\bp\|_\infty\le\LH$, and $\|\bq\|_\infty\le\LH$ are such that $\exists\ell\in\ccl 1,d\ccr$, $p_\ell <0$ and $q_\ell >0$, then for all $x\in\R^d$, and for all $\eta\in(0,1]$,
 \[
  \dHe\left(x, \bp,\bq \right) = \dHe\left(x,\begin{pmatrix} p_1 \\ \vdots \\ p_\ell \\ \vdots \\ p_d \end{pmatrix},\begin{pmatrix} q_1 \\ \vdots \\ q_\ell \\ \vdots \\ q_d \end{pmatrix} \right)
   = \dHe\left(x,\begin{pmatrix} p_1 \\ \vdots \\ 0 \\ \vdots \\ p_d \end{pmatrix},\begin{pmatrix} q_1 \\ \vdots \\ 0 \\ \vdots \\ q_d \end{pmatrix} \right),
 \]
\item \textbf{Convex setting:} for all $x\in\R^d$, $\dHe(x,\cdot,\cdot):B_{\|\cdot\|_\infty}(0,\LH)\times B_{\|\cdot\|_\infty}(0,\LH)\longrightarrow\R$ is convex.
\end{itemize}
The adaptation of Definition \ref{def:adapted}  is straightforward.
Then, notations are introduced to write the adaptation of scheme \eqref{eq:scheme} in dimension $d>1$. For all $\ell\in\ccl 1,d\ccr$, denote $T_\ell^-: \R^{\Z^{d}}  \longrightarrow  \R^{\Z^d}$, and $T_\ell^+: \R^{\Z^{d}}  \longrightarrow  \R^{\Z^d}$. They are defined such that for all  $\bu=(u_\bi)_{i\in\Z^d}\in \R^{\Z^d}$,
$T_\ell^\pm(\bu)\in\R^{\Z^d}$,  and for all $\bi\in\Z^d$,
\begin{align*}
 T_\ell^-(\bu)_\bi
 =u_{i_1,\dots,i_{\ell-1},\dots,i_d},
 \text{\;\;\;\;and\;\;\;\;}
 T_\ell^+(\bu)_\bi
 =u_{i_1,\dots,i_{\ell+1},\dots,i_d},
\end{align*}
where we denoted $ u_\bi=u_{i_1,\dots,i_\ell,\dots,i_d}$.
Then, let
 $T^\pm: \R^{\Z^d} \longrightarrow \left(\R^{\Z^d}\right)^d$, such that,
\begin{equation*}
 \forall  \bu\in\R^{\Z^d},\;
 \forall \bi\in\Z^d,\;\; T^\pm(\bu)_\bi= \begin{pmatrix} T_1^\pm(\bu)_\bi \\ \vdots \\ T_d^\pm (\bu)_\bi \end{pmatrix},
\end{equation*}
and define the adaptation of scheme \eqref{eq:scheme} for all $n\in\ccl 0,N_T-1\ccr$ and all $\bi\in\Z^d$, by
\begin{equation}
 \label{eq:scheme_extension}
 \left\{
 \begin{array}{l}
\ds \frac{u^{n+1}_\bi-u^n_\bi}{\dt}+ b\left(x_\bi, I^{n+1}\right) \dHe\left(x_\bi, \frac{I-T^-}{\dx}(\bu^n)_\bi, \frac{T^+-I}{\dx}(\bu^n)_\bi \right)+\dRe\left(\tnp,x_\bi,I^{n+1}\right)=0 \vspace{4pt} \\
\ds \min_{\bi\in\Z^d} u^n_\bi=0,
\end{array}
\right.
\end{equation}
with initialization $u^0_\bi=\uin(x_\bi)$ for all $\bi\in\Z^d$. Using the assumptions and notations above, one can follow the proofs of Prop. \ref{thm:existence} and \ref{thm:limitconv}, and remark that everything can be rewritten (with a small change in \eqref{eq:CFLsatisfier}), except the semi-concavity of the numerical solution. It is in fact straightforward in the flat setting, as in dimension $1$. However, in the convex setting,
it seems that the adaptation of the proof of Prop. \ref{thm:existence}-\ref{it:sc} (Section \ref{sec:semiconcavity}) in dimension greater than $1$, or with $\Hcont$ depending also on $x$, is more delicate. This question will be addressed in a dedicated study.

To conclude roughly, scheme \eqref{eq:scheme_extension} converges when $\dHe$ is in the flat setting. More precisely, Prop. \ref{thm:existence} and Prop. \ref{thm:limitconv} hold, with the necessary adaptations of notations and hypothesis.
In the convex setting, no such results hold. Coming back to Lotka-Volterra integral equation, with Hamiltonian \ref{eq:discussion_H}, one can notice that the tensor product of non necessarily non-negative  convex discrete Hamiltonian  $H_\eta$ is not convex. Hence, the tensor product of the examples of discrete Hamiltonians of Section \ref{sec:Example} does not give any example of multi-dimensional discrete Hamiltonian which would be in the \emph{convex setting} but not in the \emph{flat setting}. The investigation of the convergence of scheme \eqref{eq:scheme_extension} in the \emph{convex setting} is postponed to a future work.

\appendix

\section{Asymptotic-preserving property of \eqref{eq:scheme_ep}}
\label{sec:APProof}

\subsection{Proof of Prop. \ref{prop:Scheme_ep_stabilite}}
\label{sec:APProof_lemme}

In this section, we show that scheme \eqref{eq:scheme_ep} is well-posed, and we prove stability estimates, which are uniform w.r.t. $\ep$, if $\ep$ is small enough.
Suppose that the hypothesis of Prop. \ref{prop:Scheme_ep_stabilite} are satisfied.
The proof we propose above strongly relies on the monotony of \eqref{eq:MepvI}, as for the proof of Prop. \eqref{thm:existence}. Note also that, it is close with the proof of the stability Lemma in \cite{CalvezHivertYoldas2022} although technical details are different. It is done by induction on the time step. Since, the items of Prop. \ref{prop:Scheme_ep_stabilite} are true for the initial data, suppose now that $\bv^n$ is constructed such that \ref{it:AP_Lx}-\ref{it:AP_coercive}-\ref{it:AP_min} hold. We show that there exists an unique pair $(\bv^{n+1}, J^{n+1})$ satisfying \eqref{eq:scheme_ep}, and that it satisfies \ref{it:AP_Lx}-\ref{it:AP_coercive}-\ref{it:AP_Jbound}-\ref{it:AP_min}.

 The well-posedness of \eqref{eq:scheme_ep} is immediate remarking that $J^{n+1}$ is solution of
\begin{equation}
\label{eq:AP_Jnp}
 J^{n+1} =\dx\sum\limits_{i\in\Z} \psi(x_i) \e^{-(\MepvInp(\bv^n)_i-\dt R(\tnp,x_i,J^{n+1}))/\ep},
\end{equation}
that the right-hand side of this equality is decreasing w.r.t $J^{n+1}$ thanks to \eqref{hyp:R_decreasing} and Lemma \ref{lem:AP_monotonuous}, and that the left-hand side in increasing.
In other words, $J^{n+1}$ is the unique solution of $\phiAP(J)=0$, with
\begin{equation}
\label{eq:AP_phi}
\phiAP:J\mapsto  J-\dx\sum\limits_{i\in\Z} \psi(x_i) \e^{-(\MepvI(\bv^n)-\dt R(\tnp,x_i,J))/\ep}.
\end{equation}
Since \eqref{eq:AP_Jnp} has an unique solution, it can be plugged in \eqref{eq:scheme_ep} so that $\bv^{n+1}$ is also uniquely determined.

\begin{rmq}
 In practice, $J^{n+1}$ is computed  using Newton's method for the equation $\phiAP(J)=0$. However, depending on the expressions of $R$ and $b$, the computation of the iterations may be tricky. Moreover, the dependency in $\ep$ forbids the use of approximated Newton's method. Indeed, the problem becomes stiff when $\ep\to0$, and more computational time would be needed for small $\ep$ (if the resolution does not simply break).  To avoid this issue,   automatic differentiation should be used.
\end{rmq}

 The bound from below for $J^{n+1}$ in \ref{it:AP_Jbound} is obtained from
\[
 \phiAP(J)\le J-\dx\psi(x_j) \e^{-(\MepvI(\bv^n)_j - \dt R(\tnp,x_j,J))/\ep},
\]
which holds for all $j\in\Z$ and for all $J\in\R$. Moreover,
as $b$ is positive, and  thanks to \ref{it:AP_min}, the expression of $\MepvI$ in \eqref{eq:MepvI} yields for $j\in\Z$, such that $v^n_j=\min\bv^n$,
\[
\MepvI(\bv^n)_j\le v^n_j\le c_M\ep,
\]
so that for all $J\le I_m/2$,
\[
 \phiAP(J)\le\phiAP(I_m/2) \le \frac{I_m}{2}-\dx\psi_m \e^{-c_M}\e^{\dt R(\tnp,x_j,I_m/2)/\ep},
\]
where the fact that $\phiAP$ is increasing was used for the first inequality, and \eqref{hyp:psi} in the second.
Then \eqref{hyp:R_decreasing}-\eqref{hyp:R_ImIM} yield,
\details{
\begin{align*}
 &R(\tnp,x_j,I_m/2) -R(\tnp,x_i,I_m) \ge \frac{I_m}{2K} \\
 &R(\tnp,x_j,I_m/2) \ge R(\tnp,x_i,I_m) + \frac{I_m}{2K} \ge \frac{I_m}{2K}
\end{align*}
}
\[
 R(\tnp,x_j,I_m/2) \ge R(\tnp,x_i,I_m) + \frac{I_m}{2K} \ge \frac{I_m}{2K},
\]
and hence,
\[
\forall J\le I_m/2,\;\phiAP(J)\le \frac{I_m}{2}-\dx \psi_m \e^{-c_M} \e^{\dt \frac{I_m}{2\ep K}} \underset{\ep\to0}\longrightarrow -\infty.
\]
One can conclude that there exists $\ep_1>0$, which depends only on $I_m$,  $\psi_m$, $c_M$, $K$, $\dt$ and $\dx$, such that for all $\ep <\ep_1$ and for all $J\le I_m/2$, $\phiAP(J)<0$. As $\phiAP$ is increasing, it gives the bound from below for $J^{n+1}$
\begin{equation}
 \label{eq:AP_minorationJ}
 \forall \ep<\ep_1, \;\; J^{n+1}\ge I_m/2.
\end{equation}

\begin{rmq}
 We emphasize on the fact that $\ep_1$ determined above can be fixed once for all, and such that is does not depend on $n$.
\end{rmq}

The propagation of the bound from below of $\bv^n$ in  \ref{it:AP_coercive} is a straightforward consequence of Lemma \ref{lem:AP_monotonuous}. Indeed, as \eqref{eq:CFLAP} is satisfied, for all $i\in\Z$,
\[
 \MepvInp(\bv^n)_i \ge \ua|x_i| + \ubeta_{\tn} -\bM\dt \dz \sum\limits_{k\in\Z} \K(z_k) \e^{\ua|z_k|},
\]
and since $R$ is decreasing w.r.t. $J$, assumption \eqref{hyp:R_bounded} gives
\[
 R(\tnp,x_i,J^{n+1}) \le R(\tnp,x_i,I_m/2)\le K.
\]
 Hence, property \ref{it:AP_coercive} is obtained with the choice
\[
 \ubeta_\tnp = \ubeta_\tn - \bM\dt\dz\sum\limits_{k\in\Z} \K(z_k) \e^{\ua|x_i|} -\dt K,
\]
and together with \eqref{eq:AP_minorationJ}, it yields the bound of $\min\bv^{n+1}$. Indeed, we have
\details{
\[
 I_m/2\le J^{n+1}=\dx\sum\limits_{i\in\Z} \psi(x_i)\e^{-v^{n+1}_i/\ep} \le \dx\psi_M\left(\sum\limits_{|i|\le N} \e^{-v^{n+1}_i/\ep}+ \sum\limits_{|i|>N} \e^{-(\ubeta_{T} + \ua|x_i|)/\ep}
 \right),
\]
}
\[
I_m/2\le J^{n+1}=\dx\sum\limits_{i\in\Z} \psi(x_i)\e^{-v^{n+1}_i/ep} \le \dx\psi_M\left(\sum\limits_{|i|\le N} \e^{-v^{n+1}_i/\ep}+ \sum\limits_{|i|>N} \e^{-(\ubeta_{T}-\ua|x_0| + \ua|x_i-x_0|)/\ep}
 \right),
\]
for any $N\in\N$. As $x_i-x_0=i\dx$, the above expression is simplified, such that
\details{
\[
 I_m/2\le (2N-1) \dx\psi_M \e^{-\min\bv^{n+1}/\ep} + 2\dx\psi_M \e^{-(\ubeta_{T} -\ua|x_0| + \ua N\dx)/\ep }\sum\limits_{i\ge 0} \e^{-\ua i\dx/\ep}
\]
}
\[
 I_m/2\le (2N-1) \dx\psi_M \e^{-\min\bv^{n+1}/\ep} + \frac{2\dx \psi_M \e^{-(\ubeta_{T} -\ua|x_0| + \ua N\dx)/\ep } }{1-\e^{-\ua\dx/\ep}}.
\]
One can now fix $N$ such that $\ubeta_{T} -\ua|x_0|+ \ua N\dx>0$, and determine $\ep_2>0$ such that
\[
 \forall \ep\in(0,\ep_2), \;\;  \frac{2\dx \psi_M \e^{-(\ubeta_{T} -\ua|x_0| + \ua N\dx)/\ep } }{1-\e^{-\ua\dx/\ep}}\le I_m/4.
\]
This gives $\min\bv^{n+1}\le \ep c_M$, with
\details{
\begin{align*}
& I_m/4\le (2N-1)\dx\psi_M\e^{-\min\bv^{n+1}/\ep} \\
 \Leftrightarrow & \frac{I_m}{4(2N+1) \dx\psi_M} \le \e^{-\min\bv^{n+1}/\ep} \\
 \Leftrightarrow & -\ep\ln\left( \frac{I_m}{4(2N-1) \dx\psi_M} \right) \ge \min\bv^{n+1}
\end{align*}
}
\[
 c_M=\max\left\{ -\ln\left( \frac{I_m}{4(2N-1) \dx\psi_M} \right), \cM \right\}.
\]

\begin{rmq}
Note that $N$, $\ep_2$ and $c_M$ only depend on the constants arising in the assumptions, and of $\dx$. They can hence be fixed once for all, and they are independent on $n$.
\end{rmq}

The upper bound for $J^{n+1}$ is a consequence of the monotony of $\MepvI$, and of the previous result. Indeed, as \eqref{eq:CFLAP} is satisfied, and since the inequality $v^{n}_i \ge c_m \ep$, holds for all $i\in\Z$, it can be propagated so that
\details{
\[
 \forall i\in\Z,\;\forall J\in\R,\;\MepvI(\bu=u)_i=  u-\dt\dz \sum\limits_{k\in\Z} \K(z_k) b\left(x_i+\ep z_k, I\right)
\]
}
\[
\forall i\in\Z,\;\forall J\in\R,\; \MepvI(\bv^n)_i \ge c_m \ep - \dt \dz\bM \sum\limits_{k\in\Z} \K(z_k),
\]
where we used \eqref{hyp:b_bounds}. Consider now $\bv^{n+1}$ at point $j$ such that $v^{n+1}_j=\min\bv^{n+1}$. Such a minimum point exists, since $\bv^{n+1}$ satisfies \ref{it:AP_coercive}. Coming back to the expression of \eqref{eq:scheme_ep},
\begin{align}
\label{eq:AP_etape}
 \dt R(\tnp,x_j,J^{n+1})&=\MepvInp(\bv^n)_j - v^{n+1}_j
 \\&\ge (c_m-c_M)\ep -\dt \bM \dz\sum\limits_{k\in\Z} \K(z_k) + \dt R(\tnp,x_j,I_M), \nonumber
\end{align}
thanks to \eqref{hyp:R_ImIM}. We conclude using \eqref{hyp:R_decreasing}, to get another bound from below for $R(\tnp,x_j,I_M)$,
\[
 R(\tnp,x_j,I_M)\ge R(\tnp,x_j,2I_M) + \frac{I_M}{K_1},
\]
\details{
\[
 \dt R(\tnp,x_j,J^{n+1}) \ge \dt R(\tnp,x_j,2I_M) -\ep(c_M -c_m) + \dt \left(\frac{I_M}{K_1} -\bM\dz \sum\limits_{k\in\Z} \K(z_k)\right).
\]
}
\details{
\begin{align*}
 \dt R(\tnp,x_j,J^{n+1}) \ge& \dt R(\tnp,x_j,2I_M) -\ep(c_M -c_m)
 \\&+ \dt \left(\frac{I_M}{K_1} -\bM+\bM\left(1-\dz \sum\limits_{k\in\Z} \K(z_k)\right)\right)
 \\& \ge \dt R(\tnp,x_j,2I_M) -\ep(c_M -c_m)
 + \dt\left( \frac{I_M}{K_1}-\bM - \frac{1}{2}\right)
 \\& \ge \dt R(\tnp,x_j,2I_M) -\ep(c_M -c_m)  +\frac{\dt}{2}
\end{align*}
}
which is plugged in \eqref{eq:AP_etape}, to get with
 \eqref{hyp:R_ImIM}-\eqref{eq:AP_QuadratureAccurate}
\[
 \dt R(\tnp,x_j,J^{n+1})\ge \dt R(\tnp,x_j,2I_M) -\ep(c_M -c_m)  +\frac{\dt}{2}.
\]
Then, we define $\ep_3$ such that for all $\ep\in(0,\ep_3)$, $-\ep(c_M -c_m)  +\frac{\dt}{2}\ge 0$, and \eqref{hyp:R_decreasing} eventually yields $J^{n+1}\le 2I_M$, that is the second inequality in \ref{it:AP_Jbound}.

\begin{rmq}
 Once again, we emphasize on the fact that $\ep_3$ depends only on the constants defined in the assumptions and on $\dt$. In particular, it does not depend on $n$.
\end{rmq}

The bound from below of $\min \bv^{n+1}$ is straightforward using $J^{n+1}\le 2I_M$. Indeed, for all $i\in\Z$,
\[
 \psi_m \dx \e^{-v^{n+1}_i/\ep}\le \dx \sum\limits_{i\in\Z}\psi(x_i)\e^{-v^{n+1}_i/\ep}=J^{n+1} \le 2I_M,
\]
\details{
\begin{align*}
\forall i\in\Z, &\; \e^{-v^{n+1}_i/\ep} \le \frac{2I_M}{\psi_m \dx} \\
\Leftrightarrow & v^{n+1}_i\ge -\ep\ln\left(\frac{2I_M}{\psi_m\dx}\right)
\end{align*}
}
so that \ref{it:AP_min} holds, with $c_m$ independent of $n$. One can for instance consider
\[
c_m=\min\left\{
-\ln\left(\frac{2I_M}{\psi_m\dx}\right), \cm
\right\}.
\]
The Lipschitz bound \ref{it:AP_Lx} is then a consequence of \eqref{eq:CFLAP}, of Lemma \ref{lem:AP_monotonuous}, and of \ref{it:AP_Jbound}.
To conclude the proof, we eventually define $\ep_0=\min(\ep_1,\ep_2,\ep_3)$, such that items \ref{it:AP_Lx}-\ref{it:AP_coercive}-\ref{it:AP_Jbound}-\ref{it:AP_min} hold.

\subsection{Proof of Prop. \ref{prop:Scheme_ep_AP}}
\label{sec:APProof_prop}

As previously, the proof of Prop. \ref{prop:Scheme_ep_AP} is done by induction. Thanks to the assumptions, there exists $\bu^0$ such that
$\|\bu^0-\bv^0\|_\infty\longrightarrow_{\ep\to0} 0$ and that $\min\bu^0=0$.
 We suppose that it is true for a given $n\in\ccl 0, N_T-1\ccr$, and we show that there exists $\bu^{n+1}$, and $I^{n+1}$ such that
\[
\left\|\bv^{n+1}-\bu^{n+1}\right\|_\infty\underset{\ep\to 0}\longrightarrow 0,
  \;\;\;\;\text{and}\;\;\;\; J^{n+1}\underset{\ep\to 0}\longrightarrow I^{n+1},
\]
and that $\bu$ and $\bI$ satisfy scheme \eqref{eq:scheme_limit}. First of all, Prop. \ref{prop:Scheme_ep_stabilite} gives that $(J^{n+1})_{\ep\in(0,\ep_0]}$ is uniformly bounded w.r.t. $\ep$. Hence, there exists $I^{n+1}$ such that $J^{n+1}\longrightarrow_{\ep\to 0} I^{n+1}$, up to an extraction. Consider now this extraction and define  $\bu^{n+1,I^{n+1}}$ as in the first line of \eqref{eq:scheme_limit}. Let us show that
\begin{equation}
\label{eq:AP_induction_goal}
 \left\|\bv^{n+1}-\bu^{n+1,I^{n+1}}\right\|_\infty \underset{\ep\to 0}\longrightarrow 0,
\end{equation}
and that $\min\bu^{n+1,I^{n+1}}=0$.
Let $i\in\Z$, one has, with \eqref{eq:scheme_ep} and \eqref{eq:scheme_limit},
\begin{align*}
 \left| v^{n+1}_i-u^{n+1,I^{n+1}}_i\right| \le &
 \dt \left|R\left(\tnp,x_i,I^{n+1}\right) - R\left(\tnp,x_i,J^{n+1}\right)\right|
 + \left|\mathcal{M}_0^{I^{n+1}}(\bu^n)_i - \mathcal{M}_0^{J^{n+1}}(\bu^n)_i\right|
 \\ & + \left|\mathcal{M}_0^{J^{n+1}}(\bu^n)_i - \mathcal{M}_0^{J^{n+1}}(\bv^n)_i\right|
 + \left| \mathcal{M}_0^{J^{n+1}}(\bv^n)_i - \mathcal{M}_\ep^{J^{n+1}}(\bv^n)_i \right|,
\end{align*}
where the three first terms can be easily estimated. Indeed, using \eqref{hyp:R_decreasing}, one has
\[
 \left|R\left(\tnp,x_i,I^{n+1}\right) - R\left(\tnp,x_i,J^{n+1}\right)\right| \le K|I^{n+1}-J^{n+1}|,
\]
while the definition of $\MzerovI$ in \eqref{eq:MzerovI}, the fact that $b$ is $\Lb$-Lipschitz (see \eqref{hyp:b_lipschitz}), and that $\bu^n$ is $\Lx(t^n)$-Lipschitz (Prop. \ref{thm:existence}-\ref{it:Lx}) yield
\details{
\begin{align*}
  \left|\mathcal{M}_0^{I^{n+1}}(\bu^n)_i - \mathcal{M}_0^{J^{n+1}}(\bu^n)_i\right| & \le \left|b\left(x_i,I^{n+1}\right)-b\left(x_i,J^{n+1}\right)\right| \dt\dz \sum\limits_{k\in\Z} \K(z_k) \e^{|z_k|\Lx(t^n)}
\end{align*}
}
\begin{align*}
  \left|\mathcal{M}_0^{I^{n+1}}(\bu^n)_i - \mathcal{M}_0^{J^{n+1}}(\bu^n)_i\right| & \le \Lb\left|I^{n+1}-J^{n+1}\right| \dt\dz \sum\limits_{k\in\Z} \K(z_k) \e^{|z_k|\Lx(t^n)}.
\end{align*}
Eventually, as $\mathcal{M}^{J^{n+1}}_0$ satisfies Lemma \ref{lem:AP_monotonuous},
\[
 \left\|\mathcal{M}_0^{J^{n+1}}\left(\bu^n\right) - \mathcal{M}_0^{J^{n+1}}(\bv^n)\right\|_\infty\le \|\bu^n-\bv^n\|_\infty.
\]
The estimate of the last term is obtained thanks to its expression. It is rewritten to obtain
\begin{align*}
 \frac{1}{\dt\dz}&\left|\mathcal{M}_0^{J^{n+1}}\left(\bv^n\right)_i - \mathcal{M}_\ep^{J^{n+1}}\left(\bv^n\right)_i\right| \\ \le&
 \sum\limits_{k\ge 0} \K\left(z_k\right) \left|b\left(x_i+\ep z_k,J^{n+1}\right)-b\left(x_i,J^{n+1}\right)\right| \e^{-z_k\left(v^n_{i+1}-v^n_i\right)/\dx}
 \\ &+\sum\limits_{k\le -1} \K\left(z_k\right) \left|b\left(x_i+\ep z_k,J^{n+1}\right)-b\left(x_i,J^{n+1}\right)\right|\e^{-z_k\left(v^n_i-v^n_{i-1}\right)/\dx}
 \\&+ \sum\limits_{\ep z_k\ge \dx}\K\left(z_k\right) b\left(x_i+\ep z_k,J^{n+1}\right) \e^{-z_k\left(v^n_{i+1}-v^n_i\right)/\dx}
 \\&+ \sum\limits_{\ep z_k\le -\dx}\K\left(z_k\right) b\left(x_i+\ep z_k,J^{n+1}\right) \e^{-z_k\left(v^n_{i}-v^n_{i-1}\right)/\dx}
 \\&+ \sum\limits_{\ep \left|z_k\right|\ge \dx}\K\left(z_k\right) b\left(x_i+\ep z_k,J^{n+1}\right) \e^{\left(\tvik-v^n_i\right)/\ep}.
\end{align*}
\details{Est-ce vraiment utile de mettre cette grosse expression ? Ce n'est pas intéressant, mais sinon, j'ai peur qu'on ne voie pas du tout d'où la ligne suivante vient.}
Then, Prop. \ref{prop:Scheme_ep_stabilite}-\ref{it:AP_Lx} and \eqref{hyp:b_lipschitz} yield
\begin{align*}
 \frac{1}{\dt\dz}\left|\mathcal{M}_0^{J^{n+1}}\left(\bv^n\right)_i - \mathcal{M}_\ep^{J^{n+1}}\left(\bv^n\right)_i\right| \le& \;\ep\; \Lb\sum\limits_{k\in\Z} \K\left(z_k\right) \left|z_k\right| \e^{\left|z_k\right| \Lx\left(\tn\right)}
 \\&+ 2\bM \sum\limits_{\ep\left|z_k\right|\ge \dx} \K\left(z_k\right) \e^{\left|z_k\right| \Lx\left(\tn\right)},
\end{align*}
so that $\left\|\mathcal{M}_0^{J^{n+1}}\left(\bv^n\right) - \mathcal{M}_\ep^{J^{n+1}}\left(\bv^n\right) \right\|_\infty \longrightarrow_{\ep\to 0}0$. Equation \eqref{eq:AP_induction_goal} is a consequence of these estimates, and  property $\min\bu^{n+1,I^{n+1}}=0$ comes from Prop. \ref{prop:Scheme_ep_stabilite}-\ref{it:AP_min}.

To conclude the proof, one has to show that all extractions of $\left(J^{n+1}\right)_{\ep\in(0,\ep_0)}$ converge to the same limit.We argue by contradiction, supposing that there are two extractions converging respectively to $I^{n+1}_a$ and $I^{n+1}_b$, with $I^{n+1}_a<I^{n+1}_b$. With the same notations as above,  it provides two extractions of $(v^{n+1}_i)_{\ep\in(0,\ep_0)}$, which converge respectively to $u^{n+1,I^{n+1}_a}_i$ and $u^{n+1,I^{n+1}_b}_i$ when $\ep\to 0$, with $\min \bu^{n+1,I^{n+1}_a}=\min \bu^{n+1,I^{n+1}_b}=0$. However, considering
\[
 u^{n+1,I^{n+1}_a}_i - u^{n+1,I^{n+1}_b}_i = \mathcal{M}^{I^{n+1}_a}_0(\bu^n)_i - \mathcal{M}^{I^{n+1}_b}_0(\bu^n)_i + \dt \left(R\left(\tnp,x_i,I^{n+1}_b\right) - R\left(\tnp,x_i,I^{n+1}_a\right)\right),
\]
at the minimum point of $\bu^{n+1,I^{n+1}_b}$ yields a contradiction, since $I\mapsto \mathcal{M}^I_0(\bu^n)_i$ is non-decreasing (see Remark \ref{rmq:AP_properties_M0}), and thanks to \eqref{hyp:R_decreasing}.

\bibliographystyle{plain}
\bibliography{Biblio}
\end{document}